\documentclass[11pt]{article}
\usepackage{hyperref}
\usepackage{amsmath,amssymb,amsthm}
\usepackage[margin=1in]{geometry}
\usepackage{graphicx}
\usepackage{color}
\usepackage{bbm}
\usepackage{booktabs}
\usepackage{breakcites}
\usepackage{bm}
\usepackage[numbers,sort&compress]{natbib}
\usepackage{thm-restate}
\usepackage{color-edits}
\usepackage{thmtools}
\usepackage{mathtools}
\usepackage[inline, shortlabels]{enumitem}
\usepackage{enumitem}
\usepackage{dsfont}
\usepackage{xspace}
\usepackage{hyperref}
\usepackage[algo2e,linesnumbered,ruled,vlined,algosection]{algorithm2e}
\usepackage{algorithmic}
\usepackage{subcaption}
\usepackage{array}
\usepackage[dvipsnames]{xcolor}
\usepackage{cleveref}
\usepackage{makecell}

\usepackage{colortbl}

\theoremstyle{plain}
\declaretheorem[name=Theorem,numberwithin=section]{theorem}
\newtheorem{lemma}[theorem]{Lemma}

\newtheorem{corollary}[theorem]{Corollary}

\newtheorem{fact}[theorem]{Fact}

\theoremstyle{definition}
\newtheorem{definition}[theorem]{Definition}

\theoremstyle{remark}

\newcommand{\paren}[1]{\left(#1 \right )}

\newcommand{\Brac}[1]{\left[#1\right]}

\newcommand{\abs}[1]{\left\lvert#1\right\rvert}

\newcommand{\ceil}[1]{\lceil #1 \rceil}

\DeclarePairedDelimiter{\inceil}{\lceil}{\rceil}

\newcommand{\norm}[1]{\left\lVert#1\right\rVert}

\newcommand{\defeq}{\coloneqq}

\newcommand{\normInline}[1]{\lVert#1\rVert}

\newcommand{\Z}{{\mathbb Z}}

\newcommand{\R}{\mathbb R}

\newcommand{\cB}{\mathcal B}

\newcommand{\cK}{\mathcal K}

\newcommand{\cO}{\mathcal O}
\newcommand{\cP}{\mathcal P}

\newcommand{\cS}{\mathcal S}
\newcommand{\cT}{\mathcal T}
\newcommand{\cU}{\mathcal U}
\newcommand{\cV}{\mathcal V}

\newcommand{\cX}{\mathcal X}
\newcommand{\cY}{\mathcal Y}
\newcommand{\cZ}{\mathcal Z}

\DeclareMathOperator*{\argmin}{argmin} 
\DeclareMathOperator*{\argmax}{argmax}

\allowdisplaybreaks
\definecolor{violet}{RGB}{148, 0, 211}
\hypersetup{
    colorlinks=true,
    linkcolor=blue!70!black, %
    citecolor=violet, %
    filecolor=blue!70!black, %
    urlcolor=magenta
}

\newcommand{\otilde}{\tilde{O}}
\newcommand{\Otilde}{\otilde}

\newcommand{\code}[1]{\textnormal{\texttt{#1}}}
\SetKwInput{KwInput}{Input}
\SetKwInput{KwOutput}{Output}
\SetKwInput{KwParameter}{Parameter}
\SetKwProg{Function}{function}{}{}
\SetKwIF{IfNot}{ElseIfNot}{}{if not}{then}{else if not}{}{} %

\SetCommentSty{mycommfont}
\SetKwFor{Repeat}{repeat}{}{}
\SetKwRepeat{Do}{do}{while} %

\usepackage{algorithmic}

\usepackage{booktabs} %
\usepackage{array} %

\newcommand{\ellOneEllOne}{\ell_1\text{-}\ell_1}
\newcommand{\ellTwoEllOne}{\ell_2\text{-}\ell_1}

\newcommand{\gap}{\mathrm{gap}}

\newcommand{\prox}{\mathrm{prox}}

\newcommand{\true}{\code{True}}

\newcommand{\guilty}{\code{guilty}}

\newcommand{\smooth}{\code{smooth}}
\newcommand{\verdict}{\code{verdict}}
\newcommand{\Step}{\textsc{Step}}
\newcommand{\SUG}{\textsc{SUPGSolver}}

\newcommand{\flag}{\code{flag}}

\newcommand{\B}{\mathbb{B}}
\newcommand{\KL}{\mathrm{KL}}

\newcommand{\xset}{\cX}
\newcommand{\yset}{\cY}
\newcommand{\zset}{\cZ}
\newcommand{\simplex}{\Delta}
\usepackage{etoolbox}
\usepackage{xparse}
\usepackage{mathtools}
\DeclarePairedDelimiterXPP{\inangle}[1]{}{\langle}{\rangle}{}{#1}
\DeclarePairedDelimiterXPP{\inbraces}[1]{}{\{}{\}}{}{#1}
\DeclarePairedDelimiterXPP{\inparen}[1]{}{(}{)}{}{#1} %
\DeclarePairedDelimiterXPP{\insquare}[1]{}{[}{]}{}{#1}
\DeclarePairedDelimiter{\innorm}{\|}{\|}

\NewDocumentCommand\breg{s m m }{ %
  \IfBooleanTF{#1} %
    { V^{r}_{#2} \left( #3 \right) }
    { V^{r}_{#2} ( #3 ) }
}
\NewDocumentCommand\bregwr{s O{} m m }{ %
  \IfBooleanTF{#1} %
    { V^{#2}_{#3} \left( #4 \right) }
    { V^{#2}_{#3} ( #4 ) }
}

\newcommand{\rx}{r_\mathsf{x}}
\newcommand{\ry}{r_\mathsf{y}}

\NewDocumentCommand\xbreg{s m m }{
  \IfBooleanTF{#1} %
    { \bregwr*[\rx]{#2}{#3} }
    { \bregwr[\rx]{#2}{#3} }
}
\NewDocumentCommand\ybreg{s m m }{
  \IfBooleanTF{#1} %
    { \bregwr*[\ry]{#2}{#3} }
    { \bregwr[\ry]{#2}{#3} }
}

\newcommand{\gm}{\nabla_{\pm}}

\newcommand{\x}{_\mathsf{x}}
\newcommand{\y}{_\mathsf{y}}

\newcommand{\zopt}{z^\star}

\newcommand{\grad}{\nabla}

\DeclarePairedDelimiterXPP{\dualnorm}[1]{}{\|}{\|}{_{*}}{#1}

\newcommand{\AS}{\textsc{AS}}
\newcommand{\DB}{\textsc{DB}}
\newcommand{\DMP}{\textsc{DMP}}

\newcommand{\overeq}[1]{\overset{#1}{=}}
\newcommand{\overle}[1]{\overset{#1}{\le}}
\newcommand{\overge}[1]{\overset{#1}{\ge}}

\newcommand{\pathd}{\mathcal{P}}

\newcommand{\Holder}{H\"{o}lder}

\newcommand{\diag}{\mathrm{diag}}

\newcommand{\judge}{\textsc{Judge}}

\newcommand{\normalize}{\mathrm{unit}}

\newcommand{\xtrunc}{\xset_\nu}
\newcommand{\ytrunc}{\yset_\nu}
\newcommand{\ztrunc}{\zset_\nu}

\newcommand{\ground}[2]{{(#1)}_{#2}}
\newcommand{\unground}[2]{{(#1)}_{#2, *}}

\newcommand{\ball}{\B}
\DeclarePairedDelimiter{\inmaxnorm}{\|}{\|_{\mathrm{max}}}

\newcommand{\regret}{\mathrm{regret}}

\newcommand{\walpha}{w_\alpha}
\newcommand{\wbeta}{w_\beta}

\newcommand{\hess}{\nabla^2}

\newcommand{\inKL}[2]{\KL(#2||#1)} %

\newcommand{\bilinear}[1]{f_{#1}}

\newcommand{\zcenter}{z_\mathsf{center}}

\newcommand{\epsprim}{\epsilon'}

\newcommand{\sball}{\cB}

\newcommand{\gammav}{\gamma_{\mathrm{v}}}

\newcommand{\bestresponse}{\mathsf{map}}

\newcommand{\oracleSuccess}{\code{success}}
\newcommand{\oracleFailure}{\code{failure}}

\newcommand{\MDMP}{\textsc{MDMP}}
\newcommand{\MDMPImp}{\textsc{MDMPSearch}}

\newcommand{\oracle}{\cO_{\textsc{SEARCH}}}

\newcommand{\Validate}{\textsc{ConstrainedSolve}}
\newcommand{\geomean}[1]{{\mathsf{g}({#1})}}
\newcommand{\mean}[1]{{\mathsf{m}({#1})}}
\newcommand{\collapsed}[1]{{\mathsf{q}({#1})}} %

\newcommand{\cZint}{\cZ_{\mathrm{int}}}

\newcommand{\zbar}{\bar{z}}

\usepackage{tikz}

\newcommand{\Range}{\Gamma}

\newcommand{\uset}{\mathcal{U}}
\newcommand{\idset}{\mathcal{I}}

\newcommand{\biset}{\mathcal{J}}

\newcommand{\size}{\mathrm{size}}

\newcommand{\dgfsetup}{\mathcal{S}}

\newcommand{\xhat}{\hat{x}}
\newcommand{\yhat}{\hat{y}}

\newcommand{\Omegatilde}{\tilde{\Omega}}

\newcommand{\cautiousSearch}{\textsc{CautiousBisectionSearch}}
\newcommand{\ysetstable}{\yset_{\mathrm{stable}}}
\newcommand{\xsetstable}{\xset_{\mathrm{stable}}}
\newcommand{\zsetstable}{\zset_{\mathrm{stable}}}

\newcommand{\binpairs}{\mathcal{K}}

\newcommand{\prodsetup}{\mathrm{prod}}

\newcommand{\ODMP}{\mathcal{O}_{\DMP}}
\newcommand{\OMDMP}{\mathcal{O}_{\MDMP}}
\newcommand{\OAPPROX}{\mathcal{O}_{\AS}}
\newcommand{\ODB}{\mathcal{O}_{\DB}}

\newcommand{\ktstar}{k_t^\star}

\newcommand{\qU}{\mathsf{q}(\uset)} %

\newenvironment{assumptions}
  {
   \paragraph{Assumptions.}
  }
  {
   \par
   \medskip
  }

\newcommand{\Mtilde}{\widetilde{M}}

\newcommand{\innerpath}{\mathcal{P}_{\mathrm{inner}}}
 \usepackage{microtype}
  \usepackage{nth}
  \usepackage{intcalc}

  \newcommand{\cFOCS}[1]{\nth{\intcalcSub{#1}{1959}}\ Annual\ IEEE\ Symposium\ on\ Foundations\ of\ Computer\ Science\ (FOCS)}

  \newcommand{\cCOLT}[1]{\nth{\intcalcSub{#1}{1987}}\ Annual\ Conference\ on\ Computational\ Learning\ Theory\ (COLT)}

  \newcommand{\cSODA}[1]{\nth{\intcalcSub{#1}{1989}}\ Annual\ ACM-SIAM\ Symposium\ on\ Discrete\ Algorithms\ (SODA)}
  \newcommand{\cNIPS}[1]{Advances\ in\ Neural\ Information\ Processing\ Systems\ \intcalcSub{#1}{1987} (NeurIPS)}

  \newcommand{\cAAAI}[1]{AAAI\ Conference\ on\ Artificial (AAAI)}

\title{Solving Matrix Games with Near-Optimal Matvec Complexity}
\author{%
    Ishani Karmarkar\thanks{Stanford University, \texttt{\string{ishanik,ocarroll,sidford\string}@stanford.edu}} 
    \and
    Liam O'Carroll\footnotemark[1] 
    \and
    Aaron Sidford\footnotemark[1] 
}

\begin{document}

\pagenumbering{gobble}

\maketitle

\begin{abstract}
We study the problem of computing an $\epsilon$-approximate Nash equilibrium of a two-player, bilinear game with a bounded payoff matrix $A \in \R^{m \times n}$, when the players' strategies are constrained to lie in simple sets. We provide algorithms which solve this problem in $\tilde{O}(\epsilon^{-2/3})$ matrix-vector multiplies (matvecs) in two well-studied cases: $\ell_1$-$\ell_1$ (or zero-sum) games, where the players' strategies are both in the probability simplex, and $\ell_2$-$\ell_1$ games (encompassing hard-margin SVMs), where the players' strategies are in the unit Euclidean ball and probability simplex respectively. These results improve upon the previous state-of-the-art complexities of $\tilde{O}(\epsilon^{-8/9})$ for $\ell_1$-$\ell_1$ and $\tilde{O}(\epsilon^{-7/9})$ for $\ell_2$-$\ell_1$ due to [KOS '25]. In both settings our results are nearly-optimal as they match lower bounds of [KS '25] up to polylogarithmic factors. 
\end{abstract}

\setcounter{tocdepth}{2}

\tableofcontents \clearpage

\pagenumbering{arabic}

\section{Introduction}
\label{sec:intro}

In this paper, we consider the fundamental problem of computing \emph{$\epsilon$-solutions of matrix games} \cite{nesterov2005smooth,nem04,carmon2019variance,karmarkar2025solvingzerosumgames,kornowski2024oracle,kornowski2024oracleupdated,carmon2024whole,carmon2020coordinate,grigoriadis1995sublinear,clarkson2012sublinear}. In a \emph{matrix game}, we must solve the following pair of minimax and maximin optimization problems for a matrix $A \in \R^{m \times n}$ and compact, convex $\xset \subset \R^n$ and $\yset \subset \R^m$:
\begin{equation}
\label{eq:intro-general-matrix-game}
\min_{x \in \xset} \max_{y \in \yset} y^\top A x
~~\text{and}~~
\max_{y \in \yset} \min_{x \in \xset} y^\top A x\,.
\end{equation}
We call $(\hat{x},\hat{y}) \in \xset \times \yset$ an \emph{$\epsilon$-solution} if it is an \emph{$\epsilon$-approximate Nash equilibrium} in the sense that
\[
\gap(\hat{x},\hat{y}) \leq \epsilon, ~~\text{where}~~
\gap(\hat{x},\hat{y}) \defeq \max_{y \in \yset} y^\top A \hat{x} - \min_{x \in \xset} \hat{y}^\top A x\,.
\]
$\epsilon$-approximate solutions for matrix games always exist \cite{freund1999adaptive,nemirovskij1983problem,beck2003mirrordescent} and are a standard approximate solution concept. In particular, if $(\xhat, \yhat)$ is an $\epsilon$-solution, then $\xhat$ is an (additive) $\epsilon$-approximate minimizer of $\min_{x \in \xset} \max_{y \in \yset} y^\top A x$,\footnote{In other words, $\max_{y \in \yset} y^\top A \xhat \le \max_{y \in \yset} y^\top A x + \epsilon$ for all $x \in \xset$.} and $\yhat$ is an $\epsilon$-approximate maximizer of $\max_{y \in \yset} \min_{x \in \xset} y^\top A x$.

We focus on this problem of solving matrix games in two foundational, well-studied special cases described below. To define these cases (throughout the paper) we let $\Delta^{k} \defeq \{u \in \R^k_{\geq 0} : \innorm{u}_1 = 1 \}$ and $\ball^k \defeq \inbraces{u \in \R^k : \innorm{u}_2 \le 1}$ denote the $k$-dimensional probability simplex and Euclidean unit ball respectively (see Section~\ref{sec:prelims} for additional notation).

\begin{itemize}
    \item \emph{$\ell_1$-$\ell_1$ games}: In this setting, $\xset = \simplex^n$, $\yset = \simplex^m$, and $|A_{ij}| \leq 1$ for all $i\in [m]$ and $j \in [n]$. Such games encompass solving normal-form zero-sum games \cite{vonNeumann1928} and linear programming \cite{Adler2013,Dantzig1953}.
    
    \item \emph{$\ell_2$-$\ell_1$ games}: In this setting, $\xset = \ball^n$, $\yset = \Delta^m$, and $\norm{A_{i, :}}_2 \leq 1$ for all $i \in [m]$. Such games encompass hard-margin support vector machines (SVMs) \cite{rosenblatt1958perceptron,shwartz2014understandingML,mcculloch1943logical,soheili2012smoothperceptron,yu2014saddlepointsacceleratedperceptron,wang2023accelerated}, namely, computing a maximum-margin linear classifier/separating hyperplane.\footnote{Formally, the $\ellTwoEllOne$ matrix game corresponds to computing a maximum-margin linear classifier through the origin.
      However, this can be extended to capture arbitrary affine hyperplanes via standard reductions.}
\end{itemize}

We study these games under the assumptions that $n$ and $m$ are known, but $A$ is unknown and only accessible via \emph{matvec (queries)}, namely, matrix-vector multiplies of the form $(A^\top y, Ax)$ for an input $(x,y)\in\xset \times \yset$. In the context of zero-sum ($\ellOneEllOne$) games, this corresponds to both players observing the expected payoffs of each individual action, when the other player's strategy is fixed. In the context of SVMs ($\ellTwoEllOne$ games) where the rows of $A$ are data points (multiplied by the corresponding labels), this corresponds to taking linear combinations of data points $(A^\top y)$ and inner products with data points $(Ax)$. 

The fundamental question we study in this paper is:
\begin{center}
    \emph{How many matvecs are necessary to compute $\epsilon$-solutions of $\ellOneEllOne$ and $\ellTwoEllOne$ games?}
\end{center}
Until recently, the state-of-the-art query complexity for these problems was $\otilde(\epsilon^{-1})$ due to seminal, independent works of Nesterov and Nemirovski two decades ago \cite{nem04,nesterov2005smooth}.\footnote{Throughout the paper, we use $\otilde(\cdot)$ and $\Omegatilde(\cdot)$ to hide multiplicative polylogarithmic factors in $n$, $m$, and $\epsilon^{-1}$.} Despite extensive research and the development of alternative algorithms (see \Cref{table:complexities}), this $\Otilde(\epsilon^{-1})$ complexity was only recently improved by \citet{karmarkar2025solvingzerosumgames} to $\Otilde(\epsilon^{-8/9})$ for $\ellOneEllOne$ games and $\Otilde(\epsilon^{-7/9})$ for $\ellTwoEllOne$ games.

Excitingly, \cite{karmarkar2025solvingzerosumgames} showed that the $\otilde(\epsilon^{-1})$ query complexity could be improved, but was unable to match the state-of-the-art lower bounds of \citet{kornowski2024oracle, kornowski2024oracleupdated}. In particular, \citep{kornowski2024oracle} showed that deterministic algorithms require $\Omegatilde(\epsilon^{-2/3})$ queries to solve $\ellTwoEllOne$ games and $\Omegatilde(\epsilon^{-2/5})$ queries to solve $\ellOneEllOne$ games. More recently, \citep{kornowski2024oracleupdated} improved the $\ellOneEllOne$ lower bound to $\tilde{\Omega}(\epsilon^{-2/3})$ queries.

The central goal of this paper is to make progress on closing the gap between upper and lower bounds for this problem. Given the fundamental and well-studied nature of this problem and recent progress of \cite{karmarkar2025solvingzerosumgames} and \cite{kornowski2024oracle, kornowski2024oracleupdated}, we defer to these works for a more comprehensive motivation of and introduction to this problem, as well as additional discussion of related work.

\paragraph{Our results.}
The main result of this paper is a general framework for solving matrix games which improves the state-of-the-art deterministic query complexity for both problems to $\Otilde(\epsilon^{-2/3})$.

\begin{restatable}{theorem}{simplesmain}\label{thm:final-result-l1-l1-aka-zero-sum}
There is a \emph{deterministic} algorithm that computes an $\epsilon$-solution of any $\ell_1$-$\ell_1$ game with $\otilde(\epsilon^{-2/3})$ matvecs to $A$.
\end{restatable}

\begin{restatable}{theorem}{svmmain}
\label{thm:final-result-l2-l1-aka-SVM}
There is a \emph{deterministic} algorithm that computes an $\epsilon$-solution of any $\ell_2$-$\ell_1$ game with $\otilde(\epsilon^{-2/3})$ matvecs to $A$.
\end{restatable}

Theorems~\ref{thm:final-result-l1-l1-aka-zero-sum} and~\ref{thm:final-result-l2-l1-aka-SVM} improve upon the prior state-of-the-art query complexities due to \cite{karmarkar2025solvingzerosumgames} by a factor of $\Omegatilde(\epsilon^{-2/9})$ for $\ellOneEllOne$ games and $\Omegatilde(\epsilon^{-1/9})$ for $\ellTwoEllOne$ games. Importantly, they resolve the deterministic matvec complexity of $\ellTwoEllOne$ and $\ellOneEllOne$ games up to polylogarithmic factors in light of the aforementioned $\Omegatilde(\epsilon^{-2/3})$ lower bounds of \cite{kornowski2024oracle, kornowski2024oracleupdated}. See  Table~\ref{table:complexities} for a summary of advancements for both problems.\footnote{Independently, Arun Jambulapati has claimed improvements for this problem.}

\begin{table}[h]
   \centering
   \begin{tabular}{@{}p{8cm}p{2.5cm}p{2.5cm}@{}} %
   \toprule
    Method & $\ellOneEllOne$ & $\ellTwoEllOne$  \\ \midrule
   Accelerated gradient descent \cite{nesterov2005smooth}  & $\epsilon^{-1}$ & $\epsilon^{-1}$ \\
   Mirror prox \cite{nem04} & $\epsilon^{-1}$ & $\epsilon^{-1}$ \\
   Dual extrapolation \cite{Nesterov2007dualextrapolation} & $\epsilon^{-1}$ & $\epsilon^{-1}$ \\
   Optimistic mirror descent/FTRL \cite{Rakhlin2013online,steinhardt2014adaptivity,joulani2017modular} & $\epsilon^{-1}$ & $\epsilon^{-1}$ \\
   \citet{karmarkar2025solvingzerosumgames} & $\epsilon^{-8/9}$ & $\epsilon^{-7/9}$ \\
   \rowcolor[HTML]{EFEFEF}
   This paper & $\epsilon^{-2/3}$ & $\epsilon^{-2/3}$  \\ \midrule
    State-of-the-art lower bounds \cite{kornowski2024oracle, kornowski2024oracleupdated} & $\epsilon^{-2/3}$ & $\epsilon^{-2/3}$ \\
   \end{tabular}
   
   \caption{\label{table:complexities}
       Asymptotic matvec complexities for $\ellOneEllOne$ and $\ellTwoEllOne$ games. Constants and polylogarithmic factors in $n$, $m$, and $\epsilon^{-1}$ are omitted for brevity. 
       }
\end{table}

\paragraph{Other lower bounds.} Beyond the lower bounds listed in Table~\ref{table:complexities}, we note that \cite{kornowski2024oracle,kornowski2024oracleupdated} improved upon \cite{hadiji2024towards}, which achieved a $\Omega(\log(1 / (n \epsilon)))$ lower bound for $\ellOneEllOne$ games when $m = n$ for sufficiently small $\epsilon = \mathrm{poly}(1 / n)$. In addition, \citep{daskalakis2011near} showed that any \emph{no regret} online learning algorithm for zero-sum games must have regret that scales as $\Omega(\epsilon^{-1})$. The prior work of \citep{karmarkar2025solvingzerosumgames} and our own circumvent this by using additional structure of the game. We defer to \citep{karmarkar2025solvingzerosumgames} for further details.

\paragraph{Techniques.} 
Our techniques build directly upon the algorithmic framework of \cite{karmarkar2025solvingzerosumgames}, which consists of an \emph{outer loop}, \emph{bisection search procedure}, and \emph{inner loop}. Their outer loop is based on the 
\emph{prox(imal) point method} \cite{rockafellar1976monotone,martinet1970regularisation}, which reduces solving the original matrix game \eqref{eq:intro-general-matrix-game} to solving a sequence of \emph{regularized} matrix game subproblems. By dynamically searching (via their bisection search procedure) for a particular level of regularization at each iteration of the outer loop, they ensure that each of the regularized matrix game subproblems is \emph{stable} (see \Cref{sec:overview-of-approach}). This in turn enables their \emph{inner loop} subproblem solver to compute a high accuracy solution for the regularized matrix game subproblem with $\Otilde(\epsilon^{- c})$ matvecs for a suitable constant $c > 0$.\footnote{Informally, we say that an algorithm solves a problem to ``high accuracy'' if it can compute an $\eta$-approximate solution to the problem with a matvec complexity scaling at most poly-logarithmically in $\eta^{-1}$.} The subproblem solver consists of a \emph{smooth-until-proven-guilty} procedure which leverages the fact that matvecs which do not directly contribute to progress in solving the subproblem must contribute to progress in \emph{learning} the matrix $A$, and therefore can be bounded with careful algorithmic modifications.

Their framework ultimately yields a $\Otilde(\epsilon^{-8/9})$ query complexity for $\ellOneEllOne$ and for $\ellTwoEllOne$ games. Additionally, \cite{karmarkar2025solvingzerosumgames} obtain an improved $\Otilde(\epsilon^{-7/9})$ query complexity for $\ellTwoEllOne$ games via an amortized analysis which involves maintaining an approximation of the matrix $A$ between regularized matrix game subproblems (so that progress made in learning $A$ is not lost between subproblems). 

At a high level, our framework follows a similar approach to their algorithm for $\ellTwoEllOne$ games. In particular, we also use an amortized analysis and have an outer loop, bisection search procedure, and inner loop. However, our outer loop and amortized analysis differ substantially from \cite{karmarkar2025solvingzerosumgames}. Regarding the former, we develop what we term a \emph{prox multi-point method}, which generalizes the standard 
prox point method. We show that by carefully applying this new general method, we can achieve tighter control of the total change in the regularized matrix game subproblems that the inner loop solves. 

Beyond yielding an improved query complexity, the prox multi-point method primitive enables a simpler and perhaps more flexible amortization argument than the $\Otilde(\epsilon^{-7/9})$ algorithm of \cite{karmarkar2025solvingzerosumgames} for $\ellTwoEllOne$ games. Via this new framework and improved analysis, our framework also extends the amortized argument directly to $\ellOneEllOne$ games. Additionally, our framework arguably simplifies aspects of \cite{karmarkar2025solvingzerosumgames}, as discussed in \Cref{sec:overview-of-approach}, albeit at the expense of a more complicated outer loop. That said, we believe the prox multi-point method outer loop may be of independent interest and we hope this work provides valuable technical tools for improving the complexity of solving broader classes of structured optimization problems beyond matrix games. 

While our work builds on the recent work of \citep{karmarkar2025solvingzerosumgames, kornowski2024oracle, kornowski2024oracleupdated} to settle the deterministic matvec complexity of $\ell_2$-$\ell_1$ and $\ell_1$-$\ell_1$ games, our results do not immediately imply runtime, memory, or parallel depth improvements for $\ell_2$-$\ell_1$ or $\ell_1$-$\ell_1$ games. We leave this as an interesting direction for future work.

\paragraph{Paper organization.} We define notation and cover preliminaries in Section~\ref{sec:prelims}. With this notation, we provide a detailed technical overview in Section~\ref{sec:overview-of-approach} which reviews the framework of \cite{karmarkar2025solvingzerosumgames} in greater depth and motivates our approach. (We also give a more detailed guide to the rest of the paper at the end of Section~\ref{sec:overview-of-approach} once our algorithmic approach has been described.)
The remainder of the paper gives our outer loop (\Cref{sec:combined-outer-loop-section}), bisection search (\Cref{sec:MDMP-implementation}), and inner loop (\Cref{sec:sug-solver}), which we put together in \Cref{sec:putting-together} to obtain our results. Standard technical details are in the appendix.

\section{Preliminaries}\label{sec:prelims}

\paragraph{General notation.} For a vector $z \in \R^d$, we write $[z]_i$ for its $i$-th entry, $\innorm{z}_p$ for its $\ell_p$-norm, and $\diag(z) \in \R^{d \times d}$ for the diagonal matrix where the $(i, i)$-entry is $[z]_i$. If $z \in \zset \subseteq \R^d$ where $\zset = \xset \times \yset$ is a product space for $\xset \subseteq \R^n$ and $\yset \subseteq \R^m$, we write $z\x \in \xset$ and $z\y \in \yset$ for the first $n$ and last $m$ components of $z$, respectively. We refer to vectors in the $\ell_2$-unit ball, denoted in $d$-dimensions by $\ball^d$, as \emph{unit vectors}, and define $\normalize(z) \defeq z / \innorm{z}_2$ for vectors $z \ne 0$ and $\normalize(0) \defeq 0$. 
For $k \in \Z_{> 0}$ and $\ell \in \Z_{\ge 0}$, we use the notation $[k] = \inbraces{1, 2, \dots, k}$ and $[\ell]_0 = \inbraces{0, 1, \dots, \ell}$. We let $[0] \defeq \emptyset$ and use the convention that a summation over an empty index set is zero (e.g., $\sum_{t \in [0]} 1 = 0$). For sequences (of numbers, vectors, etc.) $\smash{u_1, u_2, \dots, u_T}$ or $\smash{u_1, w_1, u_2, w_2, \dots, u_T, w_T}$ we may use the notation $\{u_t\}_{t \in [T]} = \{u_t\}_{t = 1}^T$ and $\{u_t, w_t\}_{t \in [T]} = \{u_t, w_t\}_{t = 1}^T$ respectively. If $\uset$ is a multiset and $\zset$ is a set, we write $\uset \subseteq \zset$ to denote that $u \in \uset$ implies $u \in \zset$.

We write, e.g., $0_n$ and $0_{n \times m}$ for the zero vector in $\R^n$ and zero matrix in $\R^{n \times m}$ respectively. For any vectors $x, x' \in \R^d$ and $c > 1$, we use the shorthand $x \approx_c x'$ to denote that for every $i \in [d]$, $[x']_i/c \leq [x]_i \leq c [x']_i$. For a matrix $B$, we denote its $i$-th row and $i$-th column by $\smash{B_{i,:}}$ and $\smash{B_{:,i}}$ respectively. We further let $\innorm{B}_F \defeq \sqrt{\sum_{i, j} B_{i j}^2}$ denote its Frobenius norm, $\inmaxnorm{B} \defeq \max_{i, j} |B_{ij}|$ denote its max norm, and $\norm{B}_{p \to q} \defeq \max_{\innorm{x}_p \le 1} \innorm{Bx}_q$ denote the $p \to q$ induced norm. We let $f_B(x, y) \defeq y^\top B x$ denote the bilinear form in $B$. For symmetric matrices $A, B \in \R^{d \times d}$, we use $B \preceq A$ to denote that $(A-B)$ is positive semi-definite. 

\paragraph{Simplices, entropy, and KL divergence.} We let $\simplex^d$ denote the $d$-dimensional probability simplex, and further define, for $\nu > 0$, the sets $\Delta_\nu^d \defeq \{p \in \Delta^d: [p]_i \ge \nu, ~\forall i \in [d]\}$ and $\Delta_{>0}^d \defeq \{p \in \Delta^d : [p]_i > 0, ~\forall i \in [d]\}$. For any $d > 0$, we let $e: \R^d_{\geq0} \to \R$ denote the negative entropy function, i.e., $e(x) = \sum_{i \in [d]} [x]_i\log([x]_i)$ with $0 \log 0 \defeq 0$. We denote the KL divergence by $\KL(x || x') \defeq \sum_{i \in [d]} [x]_i \log([x]_i/[x']_i)$ for $x\in \Delta^d$ and $x' \in \Delta^d_{>0}$, where we let $\smash{0 \log 0 \defeq 0}$.

\paragraph{Problem setups.} 
Next, we introduce \emph{dgf setups}, which enable us to concisely instantiate a set equipped with a distance-generating function (dgf) and corresponding \emph{Bregman divergence}. The following definition is adapted from \cite[Definition 1.6]{karmarkar2025solvingzerosumgames} and modified to assume the dgf $r$ is twice differentiable. This assumption, while nonstandard in general, is typical when working with local norms as we frequently do throughout (see \Cref{sec:overview-of-approach}).

\begin{definition}[dgf setup]\label{def:dgf-setup}
We say $\dgfsetup = (\zset, r)$ is a \emph{dgf setup} if: (i) $\zset \subset \R^d$ is compact and convex; and (ii) $r : \zset \to \R$, referred to as the \emph{distance-generating function (dgf)}, is twice differentiable and 1-strongly convex over $\zset$ with respect to some norm $\normInline{\cdot} : \R^d \to \R$. For any $z, z' \in \zset$, $\breg{z}{z'} \defeq r(z') - r(z) - \inangle{ \grad r(z),  z' - z }$ denotes the \emph{Bregman divergence} induced by the dgf $r$.
\end{definition}

In the rest of this section, we introduce further notations and definitions associated with dgf setups which will be used in the remainder of the paper.

\paragraph{Monotone operators and proximal mappings.} First, we review notation related to monotone operators and proximal mappings. Given a dgf setup $(\zset, r)$ (as in Definition~\ref{def:dgf-setup}), an operator $g: \zset \to \R^d$ is said to be \emph{$\alpha$-strongly monotone} (with respect to $r$) if for any $z, z' \in \cZ$, we have $\inangle*{g(z') - g(z), z' - z} \geq \alpha \breg{z'}{z}$. If $g$ is $0$-strongly monotone, we may simply say it is \emph{monotone}. In particular, in Sections~\ref{sec:MDMP-implementation} through~\ref{sec:putting-together} we use the following definition and associated notation extensively.

    \begin{definition}[Proximal mappings, Definition 2.2 of \citep{karmarkar2025solvingzerosumgames}, restated]
        \label{def:proximal-mappings}
        For a given dgf setup $(\zset, r)$, continuous monotone operator $g : \zset \to \R^d$, points $z, w \in \zset$, regularization levels $\lambda > 0, \mu \ge 0$, and compact, convex $\zset' \subseteq \zset$, we let $\prox_{z, w}^{\lambda, \mu}(g; \zset')$ denote the unique $z' \in \zset'$ such that
        \begin{align*}
            \inangle{g(z'), z' - u}
             \le \lambda \insquare{\breg{z}{u} - \breg{z'}{u} - \breg{z}{z'}} + \mu \insquare{\breg{w}{u} - \breg{z'}{u} - \breg{w}{z'}} ~~\text{for all $u \in \zset'$},
        \end{align*}
        and similarly let $\prox_{z}^{\lambda}(g; \zset')$ denote $\prox_{z,z}^{\lambda,0}(g; \zset')$, i.e., the unique $z' \in \zset'$ such that
        \begin{align}
            \label{eq:prox-single}
            \inangle{g(z'), z' - u}
            \le \lambda \insquare{\breg{z}{u} - \breg{z'}{u} - \breg{z}{z'}} ~~\text{for all $u \in \zset'$}.
        \end{align}
        We drop $\zset'$ (e.g., writing $\prox_{z, w}^{\lambda, \mu}(g)$) when $\zset' = \zset$ for brevity. 
        
        Furthermore, in the context of the input to a proximal mapping, we may write a vector $v \in \R^d$ as a stand-in for the associated constant operator $z \mapsto v$.
        As an example, supposing $g : \zset \to \R^d$ is a continuous monotone operator and $v \in \R^d$, then $\prox_{z}^{\lambda}(v + g; \zset')$ denotes the unique $z' \in \zset'$ such that
        \begin{align*}
            \inangle{v + g(z'), z' - u}
            \le \lambda \insquare{\breg{z}{u} - \breg{z'}{u} - \breg{z}{z'}} ~~\text{for all $u \in \zset'$}.
        \end{align*}
\end{definition}

Note that the proximal mappings above all correspond to the solutions of continuous, strongly monotone variational inequalities, thereby guaranteeing existence and uniqueness (e.g., \cite{facchinei2003finitevariational}). Indeed, recall that Bregman divergences satisfy the following (e.g., \cite[Sec. 3.1]{carmon2019variance}),
\begin{align}
    \label{eq:Bregman-three-point-equality}
  \inangle{ - \grad \breg{z}{z'},  z' - u } = \breg{z}{u} - \breg{z'}{u} - \breg{z}{z'},\text{ for all }
  z, z', u \in \cZ
\end{align}
where in general $\grad \breg{z}{z'} = \grad r(z') - \grad r(z)$ denotes the gradient of $u \mapsto \breg{z}{u}$ evaluated at $z'$. Therefore, for example, \eqref{eq:prox-single} is equivalent to $\inangle{g(z') + \lambda \grad \breg{z}{z'}, z' - u} \le 0$. We define the proximal mappings as in \Cref{def:proximal-mappings} for more direct use in our applications.

\paragraph{Convex-concave functions.} 
We say $f : \xset \times \yset \to \R$ is a \emph{convex-concave} function if the restrictions of $f$ to the first $n$ and last $m$ inputs are convex and concave functions respectively. We recall the following notions of solutions to minimax games:

\begin{definition}[$\epsilon$-solution and gap function, Definition 1.1 of \cite{karmarkar2025solvingzerosumgames}, restated]\label{def:epsilon-solution} Let $\epsilon \geq 0$ and $f: \cX \times \cY \to \R$ be a convex-concave function. We say $z = (x, y) \in \cX \times \cY$ is an \emph{$\epsilon$-solution} of $\min_{x \in \cX} \max_{y \in \cY} f(x,y)$ if it is an \emph{$\epsilon$-saddle point}, i.e.,
    \begin{align*}
        \gap(z) \defeq \max_{y' \in \cY} f(x, y') - \min_{x' \in \cX} f(x', y) \leq \epsilon.
    \end{align*}
We say $z$ is an \emph{exact solution} if it is a $0$-solution. 
\end{definition}

For notational convenience, when $f$ is differentiable, we denote the natural \emph{monotone operator associated with} $f$ by $\gm f$ which is defined via $\gm f(z) \defeq (\nabla\x f(z), - \nabla\y f(z))$, where $\nabla\x f(z)$ and $\nabla\y f(z)$ denote the partial gradients of $f$ with respect to the first $n$ and last $m$ coordinates. Note that $\gm f$ is a monotone operator when $f$ is convex-concave. Additionally, letting $(\xset, \rx)$, $(\yset, \ry)$, and $(\zset \defeq \xset \times \yset, r \defeq \rx + \ry)$ denote dgf setups, a useful fact (e.g., \citep{karmarkar2025solvingzerosumgames}) which we leverage, for example, in Section~\ref{sec:putting-together}, is that for $\alpha > 0$ and $z \in \zset$, $\prox_z^\alpha(\gm f)$ is the exact solution of 
\begin{align*}
    \min_{x \in \xset} \max_{y \in \yset} f(x, y) + \alpha \xbreg{z\x}{x} - \alpha \ybreg{z\y}{y}.
\end{align*}

\addtocontents{toc}{\protect\setcounter{tocdepth}{1}} %

\section{Technical overview}
\label{sec:overview-of-approach}

In this section, we motivate and provide an overview of our algorithmic framework for proving Theorems~\ref{thm:final-result-l1-l1-aka-zero-sum}~and~\ref{thm:final-result-l2-l1-aka-SVM}. In Section~\ref{sec:additional-notation}, we briefly introduce notation pertaining to local norms and product-space setups which are used extensively in the overview and throughout the paper. In Section~\ref{subsec:old-approach}, we provide an overview of the approach we build upon, namely, the algorithm due to \cite{karmarkar2025solvingzerosumgames} which obtains a $\Otilde(\epsilon^{-8/9})$ matvec complexity for $\ellOneEllOne$ and $\ellTwoEllOne$ games. In Sections~\ref{subsec:technique-1-telescoping-sums},~\ref{subsec:technique-2-dyadic-decompositions},~and~\ref{subsec:technique-3-prox-multi-point}, we cover the key techniques which enable our improvements over \cite{karmarkar2025solvingzerosumgames}. We discuss how to put these techniques together and describe the remaining components of our algorithm in \Cref{subsec:overview-putting-it-all-together}. Finally, we provide a guide to the rest of the paper in \Cref{sec:paper-organization}.

\subsection{Local norms and product-space setups}\label{sec:additional-notation}

In this section, we introduce the general setup (\Cref{def:product-dgf-setups}) and associated notation we use for handling \emph{local norms,} distance-generating functions, and change of bases over product spaces $\xset \times \yset$ (including those arising in $\ellOneEllOne$ and $\ellTwoEllOne$ games in particular). We use this notation extensively throughout the paper (including in this overview), as local norms and corresponding changes of bases are key to our algorithmic developments, as well as those of \cite{karmarkar2025solvingzerosumgames} which we build upon.

We start by recalling local norms, which have been leveraged very extensively in prior work on optimization theory and matrix games \citep{karmarkar2025solvingzerosumgames, carmon2019variance, clarkson2012sublinear, anuran2015studyoflocalapproximationsininfotheory, shwartz2012onlinelearning, alvarez2004hessian}. In general, a local norm over a set $\cZ \subseteq \R^d$ is a function $\norm{\cdot}_{z}^{\mathsf{loc}}: \cZ \to \R_{\geq 0}$ which, for every $z \in \cZ$, is a norm. A key fact (discussed further in \Cref{subsec:old-approach}) which enables the analysis of \citep{karmarkar2025solvingzerosumgames} and our own, is that the KL divergence can be approximated by an appropriate local norm over certain subsets of the probability simplex. 

In order to introduce the specific local norms we consider in this paper, for convenience, we capture general dgf setups arising from product spaces in the following definition. Recall from Section~\ref{sec:prelims} that we use $z\x \in \xset$ and $z\y \in \yset$ to denote the components of $z \in \cX \times \cY = \zset$.

\begin{definition}[Product dgf setup and local-norm notation]
    \label{def:product-dgf-setups}
    For dgf setups $\dgfsetup\x = (\xset \subset \R^n, \rx)$ and $\dgfsetup\y = (\yset \subset \R^m, \ry)$, we say $\dgfsetup = (\zset \subset \R^d, r)$ is the \emph{product dgf setup induced by $\dgfsetup\x$ and $\dgfsetup\y$,} denoted $\dgfsetup = \prodsetup(\dgfsetup\x, \dgfsetup\y)$, if $\zset = \xset \times \yset$ as well as $r(z) = \rx(z\x) + \ry(z\y)$ for all $z \in \zset$. We associate the following local-norm notation with product dgf setups. For any $z, z' \in \zset$, we define the \emph{local norm} $\norm{z}_{z'}^2 \defeq \inangle*{z, \hess r(z') z}.$ Moreover, we define
\begin{align*}
    (z)_{z'} \defeq ((\hess \rx(z'\x))^{1/2} z\x , (\hess \ry(z'\y))^{1/2} z\y ) \in \R^d. 
\end{align*}
and for any $B \in \R^{m \times n}$ and $z' \in \zset$, we define
\begin{align*}
    (B)_{z'} &\defeq (\hess \ry(z'\y))^{-1/2} B (\hess \rx(z'\x))^{-1/2} \in \R^{m \times n}, ~\text{and} \\ 
    (B)_{z', *}  &\defeq (\hess \ry(z'\y))^{1/2} B (\hess \rx(z'\x))^{1/2} \in \R^{m \times n}.
\end{align*}
\end{definition}

In other words, we use local norms which scale the product space $\zset = \xset \times \yset$ using the Hessian of $r$. Note that in Definition~\ref{def:product-dgf-setups}, the transformation $(z)_{z'}$ performs the appropriate \emph{change of basis} such that $\normInline{z}_{z'}^2 = \normInline{(z)_{z'}}_2^2$. Similarly, the mapping $\ground{A}{z'}$ performs the corresponding change of basis to $A$ to maintain the invariant that $\inangle{z\y, A z\x} = \inangle{{\ground{z}{z'}}\y, \ground{A}{z'} {\ground{z}{z'}}\x}$. In turn, $\unground{A}{z'}$ inverts this change of basis. This is formalized in the following straightforward fact. 

\begin{fact}\label{lemma:ungrounding} Letting $\dgfsetup\x, \dgfsetup\y, \dgfsetup$ be as in Definition~\ref{def:product-dgf-setups}, for any $z, z' \in \cZ$, we have $\inangle{z\y, A z\x} = \inangle{{\ground{z}{z'}}\y, \ground{A}{z'} {\ground{z}{z'}}\x}$ and $\ground{\unground{A}{z'}}{z'} = \unground{\ground{A}{z'}}{z'} = A$. Moreover, $\normInline{\ground{z}{z'}}_2^2 = \norm{z}_{z'}^2$. 
\end{fact}

Throughout the remainder of the technical overview (Section~\ref{sec:overview-of-approach}), we instantiate the natural product dgf setups associated with the geometry of $\ellTwoEllOne$ and $\ellOneEllOne$ games. In particular, we let $\xset \defeq \ball^n$, $\rx(x) \defeq \frac{1}{2} \innorm{x}_2^2$ in the context of $\ellTwoEllOne$ games, and $\xset \defeq \simplex^n$, $\rx(x) \defeq e(x)$ in the context of $\ellOneEllOne$ games, recalling $e$ denotes the negative entropy function. We let $\yset \defeq \simplex^m$, $\ry(y) \defeq e(y)$, and fix the dgf setups $\dgfsetup\x \defeq (\xset , \rx)$ and $\dgfsetup\y \defeq (\yset , \ry)$ with $\dgfsetup = (\zset, r) \defeq \prodsetup(\dgfsetup\x, \dgfsetup\y)$ (Definition~\ref{def:product-dgf-setups}).\footnote{For technical reasons, \cite{karmarkar2025solvingzerosumgames} starts by reducing the original $\ellTwoEllOne$/$\ellOneEllOne$ games to the same games except the probability simplex domains are appropriately truncated. We truncate simplex domains in our paper for similar reasons. We omit details related to this point in our technical overview for brevity.} With these choices, the $\ellTwoEllOne$ and $\ellOneEllOne$ games are given by \eqref{eq:intro-general-matrix-game}. When we do not explicitly distinguish between the two games, we use the above notation to refer to both simultaneously.

\subsection{The approach of \citep{karmarkar2025solvingzerosumgames} for $\ellOneEllOne$ and $\ellTwoEllOne$ games}
\label{subsec:old-approach}

In this section, we provide an overview of the $\Otilde(\epsilon^{-8/9})$-matvec-complexity algorithm for $\ellOneEllOne$ and $\ellTwoEllOne$ games in \citep{karmarkar2025solvingzerosumgames}, which we build upon to prove Theorems~\ref{thm:final-result-l1-l1-aka-zero-sum}~and~\ref{thm:final-result-l2-l1-aka-SVM}. The algorithm of \citep{karmarkar2025solvingzerosumgames} uses two powerful algorithmic techniques. The first is a dynamic primal-dual prox(imal)-point \emph{outer loop} and the second is an \emph{inner loop} which implements what the authors term a ``smooth-until-proven guilty'' variant of mirror prox \citep{nem04} (SUPG Mirror Prox). In order to leverage both of these techniques, \citep{karmarkar2025solvingzerosumgames} stitch their outer and inner loops together via a \emph{bisection-search} procedure. We discuss these three components below in further detail. 

\paragraph{The outer loop.} The outer loop in \citep{karmarkar2025solvingzerosumgames} is a \textit{dynamic} variant of the \emph{prox point method} \citep{rockafellar1976monotone, martinet1970regularisation}, which reduces solving matrix games to solving
 a sequence of regularized subproblems of the form 
\begin{align}
    \label{eq:prox-point-subproblem-new}
    \min_{x \in \xset} \max_{y \in \yset} f_A(x, y) + \alpha^{(t)} \xbreg{z^{(t - 1)}\x}{x} - \alpha^{(t)} \ybreg{z^{(t - 1)}\y}{y}
    \text{ where }
    f_A(x, y) \defeq y^\top A x
\end{align}
to high-accuracy for $t = 1, 2, \dots, T$, where $\alpha^{(t)} \ge \beta > 0$ and $z^{(t - 1)} \in \zset$. In particular, the next iterate $z^{(t)}$ is set to a (high-accuracy, approximate) solution of \eqref{eq:prox-point-subproblem-new}. The algorithm then outputs a weighted average of the $z^{(t)}$, where the weighting of $z^{(t)}$ depends on the regularization levels $\alpha^{(t)}$.

Solving the sequence of subproblems \eqref{eq:prox-point-subproblem-new} with a \textit{fixed} regularization level $\alpha^{(t)} = \beta$
yields an $\epsilon$-solution of the game in $T = \Tilde{O}(\beta \epsilon^{-1})$ iterations \citep{carmon2019variance}. However, in each iteration, solving \eqref{eq:prox-point-subproblem-new} would require solving a variational inequality in a $(1 + \beta)$-(relatively) Lipschitz, $\beta$-strongly monotone 
operator, which generally requires $\Tilde{O}(\beta^{-1})$-matvecs \citep{nem04}. This would yield an overall $\Tilde{O}(\epsilon^{-1})$-matvec complexity, i.e., no improvement over simpler methods (e.g., mirror prox).

Consequently, in order to obtain their improvement, \citep{karmarkar2025solvingzerosumgames} develop a \textit{dynamic} prox point method\footnote{This approach can perhaps be seen as a primal-dual variant of techniques such as Monteiro-Svaiter acceleration and acceleration with a ball optimization oracle \citep{renato2013monteirosvaiteroriginalpaper, carmon2020acceleration, carmon2021thinking, carmon2024whole}.} with an improved iteration bound, under the additional condition that each pair $(z^{(t)}, \alpha^{(t)})$ satisfies the following \emph{kineticness} requirement (see \cite[Def. 4.1]{karmarkar2025solvingzerosumgames} with $c = 2$):
\begin{align}\label{eq:kineticness}
    \text{for each $t \in [T]$, either } \alpha^{(t)} = \beta \text{ or else } \breg{z^{(t-1)}}{z^{(t)}} \geq  (\alpha^{(t)})^2  .
\end{align}
In other words, \eqref{eq:kineticness} allows the regularization $\alpha^{(t)}$ to be larger than $\beta$, but \textit{only} if the divergence movement $z^{(t - 1)}$ to $z^{(t)}$ is at least $(\alpha^{(t)})^2$. Under \eqref{eq:kineticness}, \cite[Lemma 4.3]{karmarkar2025solvingzerosumgames} shows that the dynamic prox-point method converges to an $\epsilon$-solution after $T = \Tilde{O}(\beta \epsilon^{-1} + \epsilon^{-2/3})$-iterations. Importantly, as we discuss below, \cite{karmarkar2025solvingzerosumgames} use this added flexibility to set $\alpha^{(t)} \gg \beta$ to ensure that in each iteration, the induced subproblem \eqref{eq:prox-point-subproblem-new} can be solved with only $\Tilde{O}((\alpha^{(t)})^{-2/3}) = \Tilde{O}(\beta^{-2/3})$-matvecs, ultimately yielding their overall $\Tilde{O}(\epsilon^{-8/9})$-matvec complexity when $\beta = \epsilon^{1/3}$.

\paragraph{The inner loop.} 
The second key insight of \citep{karmarkar2025solvingzerosumgames} is to show how to solve the subproblem \eqref{eq:prox-point-subproblem-new} induced by the outer loop more efficiently than the aforementioned naive $\Tilde{O}((\alpha^{(t)})^{-1})$-matvecs.
To improve, \citep{karmarkar2025solvingzerosumgames} observe that appropriately \textit{constrained} versions of \eqref{eq:prox-point-subproblem-new} can be solved more efficiently, using their \emph{``smooth-until-proven-guilty'' (SUPG) mirror prox} inner loop \cite[Alg. 6.3]{karmarkar2025solvingzerosumgames} (a variant of composite mirror prox). More concretely, for $\alpha > 0$ and $z \in \zset$, the SUPG mirror prox inner loop of  
\citep{karmarkar2025solvingzerosumgames} solves constrained problems of the form
\begin{align}\label{eq:sub-problem-constrained}
    \min_{x \in \xsetstable} \max_{y \in \ysetstable} f_A(x, y) + \alpha \xbreg{z\x}{x} - \alpha \ybreg{z\y}{y} 
\end{align}
using only $\Tilde{O}(\alpha^{-2/3})$-matvecs, provided that the constrained regions $\xsetstable \times \ysetstable =: \zsetstable \subseteq \cZ$ satisfy the following $C$-\emph{stability condition}, 
\begin{align}\label{eq:stability}
    \normInline{\hat{z} - z'}_{\zcenter}^2 \approx_{C} \breg{\hat{z}}{z'} \text{ for any } \hat{z}, z' \in \zsetstable
\end{align}
for some absolute constant $C > 0$ and $\zcenter \in \zsetstable$. Importantly, the choice of $\zcenter$ may vary with the outer loop iteration $t$, and we discuss this in greater detail when introducing the bisection search procedure and challenges in improving \citep{karmarkar2025solvingzerosumgames}.

Note that the stability condition \eqref{eq:stability} says that within $\zsetstable$, Bregman divergences can be multiplicatively approximated by the \emph{local norm} at $\zcenter$. To make use of this property, \citep{karmarkar2025solvingzerosumgames} design SUPG mirror prox as follows. The method first initializes a \textit{model} $M \gets 0_{m \times n}$, which maintains an explicit, low-rank, approximation of $(A)_{\zcenter}$. (Correspondingly, note that $(M)_{\zcenter, *}$ is an explicit, low-rank approximation of $A$). Each iteration of SUPG mirror prox makes $O(1)$ matvecs to $A$ and either makes progress in converging to a solution of \eqref{eq:sub-problem-constrained} or else improves the approximation quality of $M$. Concretely, SUPG mirror prox runs \emph{composite} mirror prox with a step size of $1 / \tau \gg 1$ to solve 
\begin{align}\label{eq:sub-problem-constrained-composite}
    \min_{x \in \xsetstable} \max_{y \in \ysetstable} f_{A- (M)_{\zcenter, *}}(x, y) + f_{(M)_{\zcenter, *}}(x,y) + \alpha \xbreg{z\x}{x} - \alpha \ybreg{z\y}{y}. 
\end{align}
Note that \eqref{eq:sub-problem-constrained-composite} is equivalent to \eqref{eq:sub-problem-constrained}; however, while $f_{A-(M)_{\zcenter, *}}(x, y)$ can only be accessed via matvecs to $A$, the \textit{composite} term $f_{(M)_{\zcenter, *}}(x,y)$ is always explicitly maintained---and hence can be accessed without any additional matvecs to $A$. Using the stability condition \eqref{eq:stability}, they show that each iteration of their mirror prox variant is either a \textit{progress iteration}, which makes $(1 - \alpha/\tau)^{-1}$-multiplicative progress in converging to the solution of \eqref{eq:sub-problem-constrained}, or else is a \emph{model-update iteration}, which finds a pair of unit vectors $u \in \R^m, v \in \R^n$ such that $u^\top [(A)_{\zcenter} - M] v \geq \tau$. In the case of a model update iteration, they update the model as follows
\begin{align*}
    M \gets M + uv^\top \cdot {u^\top [(A)_{\zcenter} - M] v} \, ,
\end{align*}
which can be shown to reduce $\normInline{(A)_{\zcenter} - M}_F^2$ by at least $\tau^2$. It is also straightforward to show that initially $\normInline{(A)_{\zcenter}}_F^2 \leq 1$, and this allows them to bound the total number of model update iterations to $1/\tau^2$, yielding an overall matvec complexity of $\Tilde{O}(\tau/\alpha + 1/\tau^2)$ for solving \eqref{eq:sub-problem-constrained} to high-accuracy. Minimizing over $\tau$ yields the aforementioned $\Tilde{O}(\alpha^{-2/3})$ complexity.

\paragraph{Bisection search procedure.} 

Taking stock, we see that the prox point \emph{outer loop} of \citep{karmarkar2025solvingzerosumgames} requires \textit{kineticness} \eqref{eq:kineticness} while the SUPG mirror prox inner loop requires \emph{stability} \eqref{eq:stability}. In order to obtain the final $\Tilde{O}(\epsilon^{-8/9})$ complexity, \citep{karmarkar2025solvingzerosumgames} use a \emph{bisection search procedure} in each iteration of the outer loop in order to ensure that \emph{both} conditions hold.

More precisely, given the previous iterate $z^{(t-1)}$ and $\beta > 0$, \cite[Alg. 6.1]{karmarkar2025solvingzerosumgames} gives a bisection search procedure which makes $\Tilde{O}(1)$ calls to the SUPG mirror prox inner loop and finds
an $\alpha^{(t)} \in [\beta, {\Theta}(1)]$ and $z^{(t)} \in \zset$ such that either $\alpha^{(t)} = \beta$ or $ \breg{z^{(t-1)}}{z^{(t)}} \geq (\alpha^{(t)})^2$, and $z^{(t)}$ is a high-accuracy solution to \eqref{eq:prox-point-subproblem-new}. This bisection search procedure enables \cite{karmarkar2025solvingzerosumgames} to stitch together the guarantees of the inner and outer loop with at most polylogarithmic overhead, culminating in their $\Tilde{O}(\epsilon^{-8/9})$-matvec algorithm.

At a high level, this bisection search procedure uses calls to the SUPG mirror prox inner loop to find $\alpha^{(t)}, z^{(t)}$ such that the following three conditions hold. First, letting $\zopt \defeq \prox_{z^{(t-1)}}^{\alpha^{(t)}} (\nabla_\pm f_A)$ denote the exact solution of \eqref{eq:prox-point-subproblem-new}, $\alpha^{(t)}$ satisfies
\begin{align}\label{eq:upperbound}
    \breg{z^{(t-1)}}{\zopt} \leq 2.8 (\alpha^{(t)})^2, 
\end{align}
which, as \citep{karmarkar2025solvingzerosumgames} show, guarantees the existence of a simple constraint set $\xsetstable \times \ysetstable =: \zsetstable$ and a point $\zcenter \in \zsetstable$ such that the stability condition \eqref{eq:stability} holds for an appropriate absolute constant $C$. Second, $z^{(t)}$ is a high-accuracy solution to 
\begin{align*}
    \min_{x \in \xsetstable} \max_{y \in \ysetstable} f_A(x, y) + \alpha^{(t)} \xbreg{z^{(t-1)}\x}{x} - \alpha^{(t)} \ybreg{z^{(t-1)}\y}{y} 
\end{align*}
(as computed by the inner loop) which, as \citep{karmarkar2025solvingzerosumgames} show, is also a high-accuracy solution to \eqref{eq:prox-point-subproblem-new}
\begin{align*}
    \min_{x \in \cX} \max_{y \in \cY} f_A(x, y) + \alpha^{(t)} \xbreg{z^{(t-1)}\x}{x} - \alpha^{(t)} \ybreg{z^{(t-1)}\y}{y} 
\end{align*}
(as required by the outer loop). Third, $\alpha^{(t)}$ satisfies the kineticness condition \eqref{eq:kineticness} (as required by the outer loop analysis).

Thus, to recap, by combining their $\Otilde(\beta \epsilon^{-1} + \epsilon^{-2/3})$-iteration outer loop with their $\Tilde{O}((\alpha^{(t)})^{-2/3}) = \Tilde{O}(\beta^{-2/3})$-matvec inner loop with at most polylogarithmic overhead due to the bisection search procedure, \cite{karmarkar2025solvingzerosumgames} achieves an $\Otilde(\epsilon^{-8/9})$-complexity for $\ellOneEllOne$ and $\ellTwoEllOne$ games with $\beta = \epsilon^{1/3}$.

\paragraph{The challenge of improving to $\Otilde(\epsilon^{-2/3})$.} To illustrate the challenge of improving the matvec complexity of \cite{karmarkar2025solvingzerosumgames}, recall that the $T = \Tilde{O}(\epsilon^{-2/3})$-iteration complexity of the prox point outer loop in \citep{karmarkar2025solvingzerosumgames} already matches the matvec complexity lower bound of $\Tilde{\Omega}(\epsilon^{-2/3})$ \citep{kornowski2024oracleupdated}. Thus, in order to improve further, a natural starting point is to ask whether the matvecs used to build the model in the SUPG mirror prox inner loop of \citep{karmarkar2025solvingzerosumgames} can be \textit{reused} across all $\Otilde(T)$ inner loop calls, enabling a tighter amortized analysis to better bound the overall matvec complexity.

One approach towards this is to attempt to reuse the same model $M$ across all $T$ iterations of the prox point outer loop, and indeed, \citep{karmarkar2025solvingzerosumgames} consider this approach. However, because the center $\zcenter$ varies over the $\Otilde(T)$ inner loop calls, it is necessary to argue that $M$ remains a good approximation to $(A)_{\zcenter}$ as $\zcenter$ changes. \citep{karmarkar2025solvingzerosumgames} performed such an analysis, resulting in an improved $\Tilde{O}(\epsilon^{-7/9})$-matvec complexity for $\ellTwoEllOne$ games. However, their approach has two limitations. First, the analysis is specialized to $\ellTwoEllOne$ games and does not seem to readily extend to $\ellOneEllOne$ games (e.g., they use a one-sided projection to update their model), and second, their analysis seems unable to improve beyond a $\Tilde{O}(\epsilon^{-7/9})$-matvec complexity, even for $\ellTwoEllOne$ games. 

Hence, in this paper, we take an alternative approach, which we detail in the following sections. Our main algorithmic contribution is to show how to modify the outer loop of \citep{karmarkar2025solvingzerosumgames} to enable a tighter amortized analysis. Similar to \citep{karmarkar2025solvingzerosumgames}, our method consists of an outer loop, inner loop, and bisection search procedure. As in \citep{karmarkar2025solvingzerosumgames}, the outer loop runs $T = \Tilde{O}(\epsilon^{-2/3})$-iterations; however, unlike \citep{karmarkar2025solvingzerosumgames}, we carefully design and analyze the outer loop to enable our inner loop and bisection search to run in an amortized $\Tilde{O}(1)$ matvec complexity per iteration. In the next three sections, we summarize the key techniques which enable this improvement.

\subsection{Technique 1: Telescoping sums and matrix-approximation paths} 
\label{subsec:technique-1-telescoping-sums}

In this section, we discuss our first key algorithmic innovation, which is to change the \emph{representation} of the matrix $(A)_{\zcenter}$ on iteration $t$ of the outer loop. Letting $(A)_{z^{(-1)}} \defeq 0_{m \times n}$ and $\Delta_{j, j'} \defeq (A)_{z^{(j')}} - (A)_{z^{(j)}}$ for any integers $-1 \leq j \leq j' \leq t-1$, suppose we decompose $(A)_{z^{(t-1)}}$ as a telescoping sum
\begin{align}
    \label{eq:overview-z(t-1)-decomp}
    (A)_{z^{(t-1)}} &= \sum_{\ell \in [L]} \Delta_{j_{\ell-1}, j_{\ell}} \text{ where } -1 \eqqcolon j_0  < j_1 < \cdots < j_L \defeq t-1 \, .
\end{align}
This ensures that
\begin{align}
    \label{eq:overview-telescoping-sums-linear}
    (A)_{\zcenter} &= \sum_{\ell \in [L]} \Delta_{j_{\ell-1}, j_\ell} + (A)_{\zcenter} - (A)_{z^{(t - 1)}} \, .
\end{align}
Now, suppose we maintain a model $\Mtilde$ for $(A)_{\zcenter} - (A)_{z^{(t - 1)}}$ as well as a separate model $M_{j_{\ell-1}, j_{\ell}}$ for the $\Delta_{j_{\ell-1}, j_\ell}$ matrix for each $\ell \in [L]$. Using \eqref{eq:overview-telescoping-sums-linear}, we can combine these models to build a model $M$ for $(A)_{\zcenter}$, namely 
$M = \sum_{\ell \in [L]} M_{j_{\ell-1}, j_{\ell}} + \Mtilde$. In other words, to maintain a model for the target matrix $(A)_{\zcenter}$ (as required in the SUPG mirror prox analysis of \citep{karmarkar2025solvingzerosumgames}) it suffices to maintain a model for $\Mtilde$ for $(A)_{\zcenter} - (A)_{z^{(t - 1)}}$ as well a sequence of ``telescoping models'' $M_{j_{\ell - 1}, j_{\ell}}$ which sum up to a model for $(A)_{z^{(t-1)}}$. 

In order formalize the latter, we introduce the following notion of a \emph{matrix-approximation path} to $z^{(t-1)}$ in \Cref{def:matrix-approx-path}. Conceptually, a matrix-approximation path to $(A)_{z^{(t-1)}}$ is a collection of $L$ 
matrices $\Delta_{\ell}$ (which may be unknown, but for which matvec queries requires few matvecs to $A$) and corresponding models (e.g., low-rank approximations) $M_{\ell}$ for $\ell \in [L]$ such that the $\Delta_\ell$ matrices telescope to $(A)_{z^{(t-1)}}$. Correspondingly, the models $M_{\ell}$ telescope to a model of $(A)_{z^{(t-1)}}$, and we measure the quality of the models by the \emph{size} of the path. We use the terminology ``path'' as the decomposition \eqref{eq:overview-z(t-1)-decomp} can be viewed as a path to $z^{(t - 1)}$ through the prior iterates.

\begin{definition}[Matrix-approximation path]
    \label{def:matrix-approx-path}
For $z \in \zset$ and $L \in \Z_{>0}$, we call $\pathd = \inbraces{\Delta_\ell \in \R^{m \times n}, M_\ell \in \R^{m \times n}}_{\ell \in [L]}$ a \emph{matrix-approximation path to $z$} if: (i) $\sum_{\ell \in [L]} \Delta_\ell = (A)_z$, (ii) a matvec to any $\Delta_\ell$ can be computed in $O(1)$ matvecs to $A$, and (iii) the matrices $M_\ell$ are known explicitly. We refer to $L$ as the \emph{length} of $\pathd$, and additionally define
\begin{align}
    \label{eq:def-of-size-for-path}
    \size(\pathd) \defeq \sum_{\ell \in [L]} \innorm{\Delta_\ell - M_\ell}_F^2.
\end{align}
\end{definition} 

A matrix-approximation path generalizes the modeling approach used in the SUPG mirror prox inner loop of \citep{karmarkar2025solvingzerosumgames} to work with telescoping sums of models. Indeed, \citep{karmarkar2025solvingzerosumgames}'s inner loop directly builds a single model $M$ for the target $(A)_{\zcenter}$ and measures the progress of the model-update iterations by the potential $\normInline{M - (A)_{\zcenter}}_F^2$. That is, in the language of Definition~\ref{def:matrix-approx-path}, \citep{karmarkar2025solvingzerosumgames}'s inner loop always works with a matrix-approximation path to $\zcenter$ of length $L=1$. In our approach, we instead work with matrix approximation paths $\cP = \{(\Delta_{j_{\ell-1}, j_\ell}, M_{j_{\ell-1}, j_\ell})\}_{\ell \in [L]}$ to $z^{(t-1)}$ of length $L\geq 1$ and measure the progress of model-update iterations by the $\size$ of the path. 

It is not difficult to show that one can generalize the SUPG mirror prox inner loop of \citep{karmarkar2025solvingzerosumgames} so that the inner loop maintains and updates a matrix-approximation path to $(A)_{z^{(t-1)}}$ and a model of $(A)_{\zcenter} - (A)_{z^{(t-1)}}$, as opposed to maintaining a single model for $(A)_{\zcenter}$ as was done in the original inner loop of \citep{karmarkar2025solvingzerosumgames} (compare Section~\ref{sec:sug-solver} to Section 5 of \citep{karmarkar2025solvingzerosumgames}). Concretely, one can slightly modify the SUPG mirror prox inner loop of \citep{karmarkar2025solvingzerosumgames} to build a \emph{path-modified} (SUPG mirror prox) inner loop that does the following. It takes as input a matrix-approximation path $\cP$ to $z^{(t-1)}$ of length $L$, $\alpha \in [\beta, \Theta(1)]$, and $\zcenter \in \zsetstable$ such that both $\breg{z^{(t-1)}}{\zcenter} \leq C'\alpha^2$ and \eqref{eq:stability} hold for appropriate absolute constants $C, C' > 0$. It then outputs a high-accuracy solution to the constrained problem
\begin{align*}
    \min_{x \in \xsetstable} \max_{y \in \ysetstable} f_A(x,y) + \alpha \breg{z^{(t-1)}\x}{x} - \alpha \breg{z^{(t-1)}\y}{y}, 
\end{align*}
\emph{along with} an updated matrix-approximation path $\cP'$ to $z^{(t-1)}$ of length $L$, after making at most 
\begin{align}\label{eq:overview-path-modified-upshot}
     \Tilde{O}\paren { L^2 \tau^{-2} ( \size(\cP) - \size(\cP') + \alpha^2 ) + \tau \beta^{-1} }
     \text{ matvec queries to $A$.}
\end{align}

Note that in the path-modified inner loop, the number of matvecs due to progress iterations remains the same as in the original SUPG mirror prox inner loop of \cite{karmarkar2025solvingzerosumgames} at $\Otilde(\tau \beta^{-1})$ (as indeed model updates do not occur in progress iterations, and thus their execution remains the same as before). Importantly, \eqref{eq:overview-path-modified-upshot} indicates that the number of matvec queries that the path-modified inner loop makes due to model-update iterations, namely $\Otilde(L^2 \tau^{-2} ( \size(\cP) - \size(\cP') + \alpha^2 ))$, can be directly be charged to the improvement in approximation quality of $\cP'$ relative to $\cP$, along with an additive $\alpha^2$ term to bound matvecs due to updates to the model $\Mtilde$. Unfortunately, this guarantee degrades quadratically with the length $L$ of the matrix-approximation path. Fortunately, there is another standard technique for controlling $L = \Tilde{O}(1)$ using a dyadic (or binary) decomposition, which we discuss in the next section, which ensures that this degradation is at most polylogarithmic. 

\paragraph{The utility of matrix-approximation paths.} Before discussing the dyadic decomposition technique in further detail, we briefly pause to highlight some important intuition for why matrix-approximation paths are useful in our analysis. Recall that as we discussed above, a key obstacle towards improving \citep{karmarkar2025solvingzerosumgames} with a tighter amortization analysis was that $\zcenter$  changes in each invocation of their SUPG mirror prox inner loop of \citep{karmarkar2025solvingzerosumgames}. Consequently, in general, it is unclear how to argue that a model for the (moving) target $(A)_{\zcenter}$ constructed in one invocation of their inner loop can effectively reused for other invocations. (Indeed, recall that \citep{karmarkar2025solvingzerosumgames} were only able to obtain such an argument for $\ell_2$-$\ell_1$ games and their argument does not lead to the near-optimal matvec complexity.)

In contrast, consider our approach of using matrix-approximation paths to $z^{(t-1)}$ along with the path-modified inner loop. In this case, the outer loop's iterates $z^{(1)}, \dots, z^{(t-1)}$ are fixed by the time we begin the $t$-th iteration and consequently, each $\Delta_{j,j'}$ for $-1 \leq j \leq j' \leq t-1$ is a \textit{fixed} target matrix. In other words, the target matrices in a matrix approximation path $\cP$ remain \emph{fixed} over the entire algorithm, even though $\zcenter$ changes upon each call to the inner loop. Hence, the work done to update each $M_{j_{\ell-1}, j_{\ell}}$ remains potentially useful across \textit{all} outer loop iterations. This essentially allows the corresponding matvecs to be ``reused'' in future iterations of the outer loop, leading to an improved amortization argument. Importantly, this argument holds for \emph{both} $\ell_2$-$\ell_1$ and $\ell_1$-$\ell_1$ games and, as we show, enables near-optimal matvec complexities for both problems.

\subsection{Technique 2: Dyadic decompositions}
\label{subsec:technique-2-dyadic-decompositions}

In this section, we discuss how we use dyadic decompositions to control the matrix-approximation path lengths and correspondingly control the $L^2$ factor in \eqref{eq:overview-path-modified-upshot} to be at most polylogarithmic.  Dyadic decompositions are widespread throughout data structure/algorithm design (e.g., binary index trees) and optimization (e.g., \cite{bachoc2022nearoptimalalgorithmunivariatezerothorder,Axelrod2019NearoptimalAD}). Its use in the paper \cite{carmon2024whole} for a matrix-vector maintenance data structure perhaps most closely resembles our application described below, albeit the specific context (namely, the subproblem it is applied to solve) is still quite different.

To specify how we use dyadic decompositions to control matrix-approximation path lengths, we first describe the overarching set of models which our algorithm maintains at each iteration of the outer loop. Indeed, suppose that at iteration $t$ of the outer loop, we maintain models $M_{j, j'}$ for all pairs $(j, j')$ formed by adjacent partial sums (ordered by descending powers of 2) in the binary decompositions of the integers between 1 and $t - 1$, as well as a model $M_{-1, 0}$. As before, $M_{j, j'}$ is a model for $\Delta_{j, j'} \defeq (A)_{z^{(j')}} - (A)_{z^{(j)}}$ with $(A)_{z^{(-1)}} \defeq 0_{m \times n}$. Then with this particular set of models, we can build a matrix-approximation path of length $O(\log t)$ for $(A)_{z^{(t-1)}}$ using the pairs $(j, j')$ formed by adjacent partial sums in the binary decomposition of $t - 1$. (For an illustration of this path for $t - 1 = 13$, see Figure~\ref{fig:dyadic-example}.)
This ensures that for every $t \in [T]$, every matrix-approximation path $\cP$ for $(A)_{z^{(t-1)}}$ passed into the path-modified inner loop is of length $L = O(\log T)$. 

\begin{figure}
    \centering
    \includegraphics[width=0.8\linewidth]{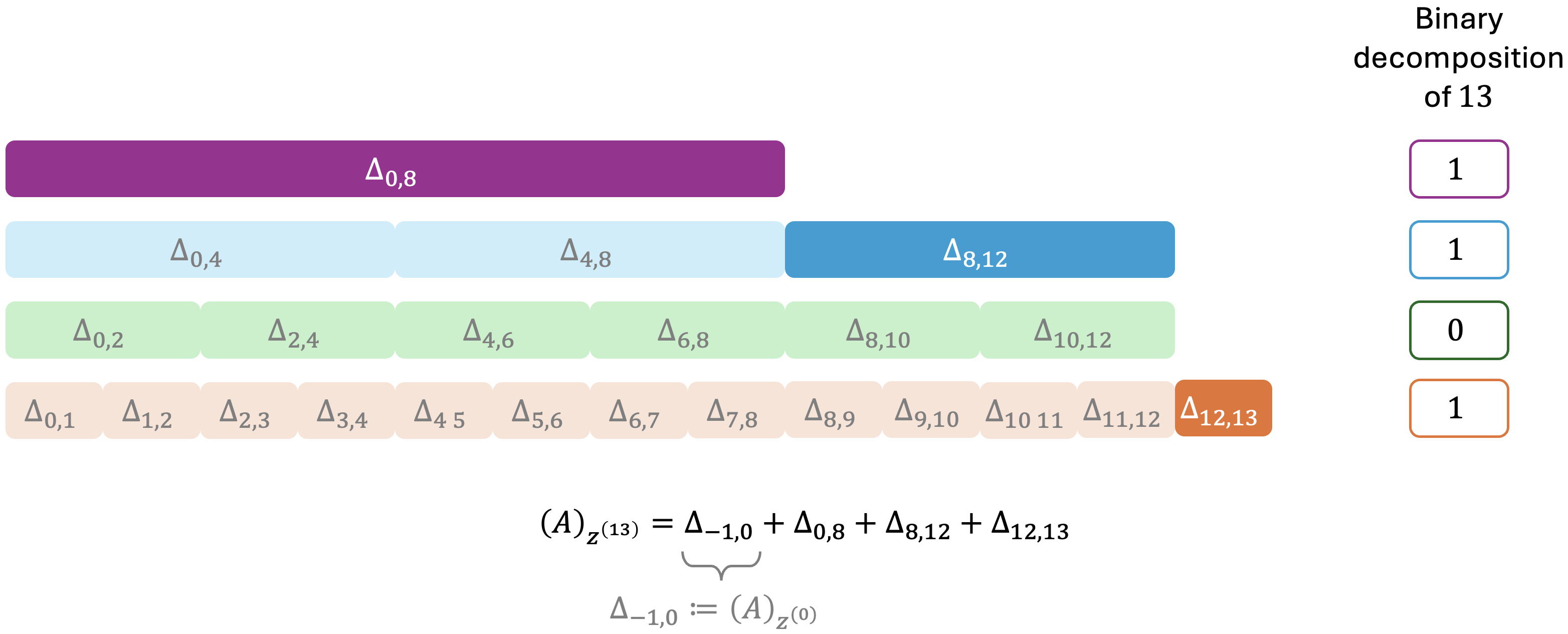}
    \caption{Dyadic decomposition. This figure illustrates how we decompose $(A)_{z^{(13)}}$ into a telescoping sum of $\Delta_{-1,0} = (A)_{z^{(0)}}$ plus at most $\ceil{\log_2(13)}$ distinct $\Delta_{j,j'}$'s for $(j,j') \in \cK$. This \textit{dyadic decomposition} leverages the binary expansion of $13=1101$.}
    \label{fig:dyadic-example}
\end{figure}

We can now compute the overall matvec complexity of the algorithm which results from combining this choice of models and matrix-approximation paths (i.e., dyadic decompositions), the path-modified SUPG mirror prox inner loop, and the outer loop and bisection search procedure of \cite{karmarkar2025solvingzerosumgames}. Let $\binpairs$ denote the set of all pairs $(j, j')$ formed by adjacent partial sums in the binary decompositions of the integers between $1$ and $T$ (ordered by descending powers of 2). Then a telescoping argument (see \Cref{thm:matrix-games-outer-loop-guarantee} and \Cref{thm:main-general-result}) using the matvec complexity \eqref{eq:overview-path-modified-upshot} of the path-modified inner-loop (and the at most polylogarithmic
 overhead of the bisection search procedure) gives a total matvec complexity of
\begin{align*}
   &\Tilde{O} \paren{ \tau^{-2}  \sum_{(j,j') \in \binpairs} \innorm{\Delta_{j, j'}}_F^2  +  \tau^{-2} \sum_{t \in [T]} (\alpha^{(t)})^2  + T \tau \beta^{-1} 
    } \\
     \overeq{(i)} \, & \Otilde \inparen*{
        \tau^{-2} \sum_{(j,j') \in \binpairs}  \breg{z^{(j)}}{z^{(j')}}  +  T (\tau^{-2} \beta^2 +  \tau \beta^{-1} ) + \tau^{-2}
     }.
\end{align*}
Here, $(i)$ uses the fact that $\innorm{\Delta_{j, j'}}_F^2 \le 2 \breg{z^{(j)}}{z^{(j')}}$ (see \Cref{lemma:compatibility}), and $\sum_{t \in [T]} (\alpha^{(t)})^2 = \Otilde(1 + T \beta^2)$ by the kineticness condition \eqref{eq:kineticness} and a standard prox point movement bound (see also \Cref{lem:multiprox-iteration-bound}). 

Then recalling $T = \Otilde(\epsilon^{-2/3})$, if we could show $\sum_{(j,j') \in \binpairs}  \breg{z^{(j)}}{z^{(j')}} = \Otilde(1)$, the choice $\tau = \beta = \epsilon^{1/3}$ would achieve an overall matvec complexity of $\Otilde(\epsilon^{-2/3})$! However, it is not clear how to obtain this movement bound via the standard prox point method. Indeed, the standard prox point method only gives strong control, to our knowledge, of the total movement between consecutive iterates, namely $\sum_{t \in [T]} \breg{z^{(t - 1)}}{z^{(t)}} \le \Otilde(1)$ (e.g., \cite[Lemma 4.2]{karmarkar2025solvingzerosumgames}; we also recap this bound in the next section). 
While it has been shown that the Bregman divergence $\breg{z}{z'}$ satisfies a type of relaxed triangle inequality \cite[Section 5.1]{carmon2024whole} (after the domain has been appropriately truncated), the multiplicative logarithmic factor it picks up in each use means it is unclear how to achieve tight control of $\breg{z^{(j)}}{z^{(j')}}$ when $j' - j$ is large.

To overcome this obstacle, we develop a new general primitive described next, termed the \emph{prox multi-point method}, which may be of independent interest. We note that to our knowledge, while there are many instances of using dyadic decompositions as a purely analytical technique, our use of it as an algorithmic intervention (as described next) is less common.

\subsection{Technique 3: The prox multi-point method}
\label{subsec:technique-3-prox-multi-point}

In this section, we describe our general \emph{prox multi-point method} primitive which enables tighter iterate movement bounds than the standard prox point method, thereby enabling our improved amortized analysis. 
While we ultimately apply our prox multi-point method to compute $\epsilon$-solutions of matrix games, our prox multi-point method (as well as the standard prox point method) solves the more general problem of achieving $\epsilon$-regret with respect to a monotone operator. (See \Cref{lem:regret-bounds-the-gap} for a formal restatement of this reduction.) Thus, we operate in the latter, more general setting in this section (\Cref{subsec:technique-3-prox-multi-point}). Formally, fix a dgf setup $\dgfsetup = (\zset, r)$ per Definition~\ref{def:dgf-setup} with $\Gamma_\dgfsetup \ge \max_{z, z' \in \zset} r(z) - r(z')$ and a continuous monotone operator $g : \zset \to \R^d$. 

Before describing our method, we first briefly recap the movement bound between consecutive iterates achieved by the standard prox point method. The standard prox point method starts with an initial point $z^{(0)} \in \zset$ and iterates $z^{(t)} \gets \prox_{z^{(t - 1)}}^{\alpha^{(t)}} (g)$ for $t = 1, 2, \dots, T$, where $\alpha^{(1)}, \alpha^{(2)}, \dots$ is a sequence of positive regularization parameters which can be chosen dynamically. By Definition~\ref{def:proximal-mappings}, this update is equivalent to $z^{(t)}$ satisfying
\begin{align*}
    \inangle{g(z^{(t)}), z^{(t)} - u} \le \alpha^{(t)} \insquare{\breg{z^{(t - 1)}}{u} - \breg{z^{(t)}}{u} - \breg{z^{(t - 1)}}{z^{(t)}}} \text{ for all $u \in \zset$}.
\end{align*}
Multiplying both sides by $(\alpha^{(t)})^{-1} / S$ for $S \defeq \sum_{t \in [T]} (\alpha^{(t)})^{-1}$ and summing gives the standard regret guarantee
\begin{align}
    \label{eq:standard-prox-point-regret-guarantee}
   \frac{1}{S} \sum_{t \in [T]} (\alpha^{(t)})^{-1} \inangle{g(z^{(t)}), z^{(t)} - u} \le \frac{\breg{z^{(0)}}{u} - \sum_{t \in [T]} \breg{z^{(t - 1)}}{z^{(t)}}}{S} \text{ for all $u \in \zset$}.
\end{align}
In the context of $\ellOneEllOne$ and $\ellTwoEllOne$ games, a standard result (\Cref{lem:regret-bounds-the-gap}) gives that instantiating this guarantee with $g \gets \gm f_A$ results in $\gap(\frac{1}{S} \sum_{t \in [T]} (\alpha^{(t)})^{-1}z^{(t)} )$ (recall \Cref{def:epsilon-solution}) being bounded by the left-hand side of \eqref{eq:standard-prox-point-regret-guarantee}. Furthermore, it is straightforward in this application to pick $z^{(0)}$ such that $\sup_{u \in \zset} \breg{z^{(0)}}{u} = \Otilde(1)$, in which case we obtain the movement bound $\sum_{t \in [T]} \breg{z^{(t - 1)}}{z^{(t)}} = \Otilde(1)$ since regret with respect to a monotone operator is nonnegative (e.g., \cite[Prop. A.1]{karmarkar2025solvingzerosumgames}).

Recall that our goal is to extend this movement bound between consecutive iterates to a movement bound over all pairs of iterates $z^{(j)}$ and $z^{(j')}$ for $(j, j') \in \binpairs$. We achieve this goal by carefully adding additional regularization to iterations $t$ of the standard prox point method about iterates \emph{before} the previous iterate $z^{(t - 1)}$ in order to increase control over gapped pairs of iterates.

Formally, our \emph{prox multi-point method} (Algorithm~\ref{alg:multiprox-method-monotone-op}) maintains $K$ sequences of regularization center points, where the $k$-th sequence for $k \in [K]$ is denoted $w_k^{(0)}, w_k^{(1)}, \dots, w_k^{(T)}$, as well as a sequence of iterates $z^{(0)}, z^{(1)}, \dots, z^{(T)}$. We also maintain a sequence of \emph{active center index sets} $\idset^{(1)}, \idset^{(2)}, \dots, \idset^{(T)}$ where each $\idset^{(t)} \subseteq [K]$. These encode which regularization centers are active at each iteration $t$, and can be chosen dynamically (though this is not strictly necessary for our application to matrix games). We initialize $w_k^{(0)} \gets z^{(0)}$ for all $k \in [K]$. 

Then, at every iteration $t = 1, 2, \dots, T$ of the prox multi-point method, we let $\uset^{(t)} \defeq \inbraces{w_k^{(t - 1)} : k \in \idset^{(t)}}$ denote the multiset containing the centers which are active at that step, and set $z^{(t)} \in \zset$ to be the unique point such that
\begin{align}
    \label{eq:overview-prox-multi-point-update}
   (\alpha^{(t)})^{-1} \inangle{g(z^{(t)}), z^{(t)} - u} \le  \sum_{k \in \idset^{(t)}} \insquare{\breg{w_k^{(t - 1)}}{u} - \breg{z^{(t)}}{u} - \breg{w_k^{(t - 1)}}{z^{(t)}}} \text{ for all $u \in \zset$} .
\end{align}
We introduce the shorthand notation $z^{(t)} \gets \prox_{\uset^{(t)}}^{\alpha^{(t)}}(g ; \zset)$ for this update in \Cref{sec:combined-outer-loop-section}; as discussed further there, $z^{(t)}$ is the solution to a strongly monotone variational inequality with respect to the operator $g + \alpha^{(t)} \sum_{w \in \uset^{(t)}} \grad \breg{w}{\cdot}$.\footnote{In Section~\ref{subsec:prox-multi-general-monotone-ops} we allow for an approximate solution, but we use exact solutions here for simplicity.} Note that for different applications, prior works also use proximal steps with multiple regularization centers  \citep{mehta2025min, mehta2024drago}. 

We then update the regularization centers via
\begin{align}
    \label{eq:w-update-in-overview}
    w^{(t)}_k \gets \begin{cases} 
      z^{(t)},  & \text{for all } k \in \idset^{(t)}, \\
      w^{(t - 1)}_k,  & \text{for all } k \in [K] \setminus \idset^{(t)}.
   \end{cases}
\end{align}
Summing \eqref{eq:overview-prox-multi-point-update} and dividing by $S$, we obtain for all $u \in \zset$:
\begin{align*}
    \frac{1}{S} \sum_{t \in [T]} (\alpha^{(t)})^{-1} \inangle{g(z^{(t)}), z^{(t)} - u} &\le \frac{1}{S} \sum_{t \in [T]} \, \,\sum_{k \in \idset^{(t)}}  \insquare{\breg{w_k^{(t - 1)}}{u} - \breg{z^{(t)}}{u} - \breg{w_k^{(t - 1)}}{z^{(t)}}} \\
    &= \frac{1}{S} \sum_{t \in [T]} \sum_{k \in [K]}  \insquare{\breg{w_k^{(t - 1)}}{u} - \breg{w_k^{(t)}}{u} - \breg{w_k^{(t - 1)}}{w_k^{(t)}}} \\
    &= \frac{1}{S}  \sum_{k \in [K]} \sum_{t \in [T]} \insquare{\breg{w_k^{(t - 1)}}{u} - \breg{w_k^{(t)}}{u} - \breg{w_k^{(t - 1)}}{w_k^{(t)}}} \\
    &\le \frac{1}{S}  \insquare*{ \sum_{k \in [K]} \breg{w_k^{(0)}}{u}  -   \sum_{k \in [K]} \sum_{t \in [T]}  \breg{w_k^{(t - 1)}}{w_k^{(t)}}           }.
\end{align*}
Therefore, in the matrix games application where we can ensure $\sup_{u \in \zset} \breg{w_k^{(0)}}{u} = \Otilde(1)$ for all $k \in [K]$ simultaneously, we obtain the movement bound $\sum_{k \in [K]} \sum_{t \in [T]}  \breg{w_k^{(t - 1)}}{w_k^{(t)}} = \Otilde(K)$. Thus, the prox multi-point method can be viewed as enabling control over $K$ subsequences of iterates at the cost of a $K$ factor in the movement bound. 

In our application to matrix games 
where we aim to control $\sum_{(j, j') \in \binpairs} \breg{z^{(j)}}{z^{(j')}}$, we choose $\idset^{(t)} \gets \inbraces{k \in [K] : \text{ $t$ is divisible by $2^{k - 1}$ }}$ for all $t \in [T]$ and $K = \Theta (\log T)$, and thereby obtain the desired control
\begin{align}
    \label{eq:overview-multiprox-movement-bound}
    \sum_{(j, j') \in \binpairs} \breg{z^{(j)}}{z^{(j')}} \overle{(i)} \sum_{k \in [K]} \sum_{t \in [T]}  \breg{w_k^{(t - 1)}}{w_k^{(t)}}  \le \Otilde(\log T).
\end{align}
Regarding $(i)$, note that for any $k \in [K]$, the update \eqref{eq:w-update-in-overview} along with the choice of $\idset^{(t)}$ implies $w_{k}^{(t)} = z^{(b_{k, t})}$ where $b_{k, t} \defeq t - (t \bmod 2^{k - 1})$; namely, $b_{k, t}$ is the largest multiple of $2^{k - 1}$ which is at most $t$. Thus, 
\begin{align*}
    \sum_{t \in [T]}  \breg{w_k^{(t - 1)}}{w_k^{(t)}} = \sum_{m \ge 0} \breg{z^{(m \cdot 2^{k - 1})}}{z^{((m + 1) \cdot 2^{k - 1})}} \, . 
\end{align*}
Namely, the $k$-th inner summation in the second term in \eqref{eq:overview-multiprox-movement-bound} bounds the divergences between iterates corresponding to consecutive powers of $2^{k - 1}$. In other words, 
\begin{align}
    \label{eq:overview-interpretation-of-movement-bound}
    \sum_{k \in [K]} \sum_{t \in [T]}  \breg{w_k^{(t - 1)}}{w_k^{(t)}} = \sum_{(j, j') \in \biset} \breg{z^{(j)}}{z^{(j')}} 
\end{align}
where $\biset \defeq \inbraces{(2^k m, 2^k (m + 1) ) : k, m \in \Z_{\ge 0} \text{ s.t. } 2^k (m + 1) \le T}$ is the set of all pairs given by consecutive multiples of the same power of two, with the restriction that they are at most $T$. (For example, for $T = 13$, $\biset$ contains precisely the pairs $(j, j')$ for all of the matrices $\Delta_{j, j'}$ shown in \Cref{fig:dyadic-example}.) Then $\binpairs \subseteq \biset$, yielding the inequality $(i)$.

We refer to the prox multi-point method with the specific choice of $\idset^{(t)}$ which we instantiate in our matrix games application (and which yields control between iterates gapped by dyadic intervals) as the \emph{dyadic prox method}. However, we ultimately define it with $\idset^{(t)} \gets \inbraces{k \in [K] : \text{ $t$ is divisible by $2^{K - k}$ }}$ as opposed to the choice $\idset^{(t)} \gets \inbraces{k \in [K] : \text{ $t$ is divisible by $2^{k - 1}$ }}$ made above since the former leads to more concise pseudocode indexing (see Section~\ref{subsec:prox-multi-matrix-games} for details). Note that this corresponds to simply relabeling the sequences $w_k^{(0)}, w_k^{(1)}, \dots, w_k^{(T)}$ (e.g., what was previously $w_1^{(0)}, \dots, w_1^{(T)}$ is now $w_K^{(0)}, \dots, w_K^{(T)}$), so the iterates $z^{(t)}$ remain unchanged.

\subsection{Putting it all together}
\label{subsec:overview-putting-it-all-together}

We have now covered all of our core innovations over the approach of \cite{karmarkar2025solvingzerosumgames}, but some more straightforward technical details remain. In this section, we give an overview of the outer loop, inner loop, and bisection search procedure of our algorithm, and how they relate to those of \cite{karmarkar2025solvingzerosumgames}.

\paragraph{Our outer loop.} Our outer loop is the dyadic prox method covered above on the monotone operator $g \gets \gm f_A$. We set $z^{(t)}$ to a high-accuracy solution of $\prox_{\uset^{(t)}}^{\alpha^{(t)}}(\gm f_A ; \zset)$ for $t = 1, 2, \dots, T$, where $\alpha^{(t)} \ge \beta > 0$ and $\uset^{(t)} \subset \zset$ is the multiset defined above. (Compare this to the subproblem \eqref{eq:prox-point-subproblem-new} in the outer loop of \cite{karmarkar2025solvingzerosumgames}, which is equivalent to $\prox_{z^{(t - 1)}}^{\alpha^{(t)}}(\gm f_A ; \zset)$.) We show that under a \emph{kineticness} requirement which is the natural extension of \eqref{eq:kineticness} to our update (see \Cref{def:DMP}), our outer loop obtains the same $T = \Otilde(\beta \epsilon^{-1} + \epsilon^{-2/3})$ iteration bound as the outer loop of \cite{karmarkar2025solvingzerosumgames}.

\paragraph{Our inner loop.} As in \citep{karmarkar2025solvingzerosumgames}, our inner loop is parameterized by a step size $\tau > 0$ and implements a similar ``smooth-until-proven-guilty'' variant of composite mirror prox. Our inner loop takes as input a matrix-approximation path $\innerpath = \{(\Delta_\ell, M_\ell)\}_{\ell \in [L]}$ to a point $\zcenter \in \cZ$ and solves constrained prox multi-point subproblems of the following form to high accuracy, for $\alpha > 0$ and appropriate stable regions $\xsetstable \times \ysetstable =: \zsetstable \subseteq \cZ$,
\begin{align}\label{eq:ultimate-subproblem}
    \min_{x \in \xsetstable} \max_{y \in \ysetstable} f_A(x,y) + \alpha \sum_{u \in \cU} \bregwr[\rx]{u\x}{x} - \alpha \sum_{u \in \cU} \bregwr[\ry]{u\y}{y} .
\end{align}
Here, $\cU \subset \cZ$ is any finite, non-empty multiset of $\cZ$ (compare this to the subproblem \eqref{eq:sub-problem-constrained} in the inner loop of \citep{karmarkar2025solvingzerosumgames} where $\cU$ is a singleton). Our inner loop is only a slight modification of \citep{karmarkar2025solvingzerosumgames}'s smooth-until-guilty mirror prox inner loop. The key differences, as discussed in Section~\ref{subsec:technique-1-telescoping-sums}, are that (i) our inner loop performs model-update iterations to update the entire matrix approximation path to $\zcenter$ rather than to update just a single fixed model of $(A)_{\zcenter}$, and (ii) our inner loop returns a final matrix-approximation path $\innerpath' = \{\Delta_{\ell}, M'_\ell\}_{\ell \in [L]}$ to $\zcenter$ along with a high-accuracy solution to \eqref{eq:ultimate-subproblem}. The overall matvec complexity of the inner loop is bounded by 
\begin{align*}
    \Tilde{O}\paren{ L^2 \tau^{-2} [\size(\innerpath) - \size(\innerpath')] + \tau \alpha^{-1} } \text{matvecs}.
\end{align*}

\paragraph{Our bisection search.} We use a straightforward extension of the bisection search procedure of \citep{karmarkar2025solvingzerosumgames} to bridge our outer and inner loop. Concretely, given a finite multiset $\cU^{(t)} \ni z^{(t-1)}$ which is a subset of $\cZ$, a matrix approximation path $\cP = \{(\Delta_\ell, M_\ell)\}_{\ell \in [L]}$ to $z^{(t-1)}$, and $\beta > 0$, our bisection search procedure makes $\Tilde{O}(1)$ calls to our inner loop. In each call to our inner loop, the bisection search procedure passes in a matrix approximation path $\innerpath$ to a (carefully selected) center $\zcenter$, along with an appropriate stable region $\zsetstable$ and $\alpha \ge \beta$. Here, $\innerpath$ is obtained from $\cP$ by concatenating $\cP$ with an $(L+1)$-th element $(\Delta_{L+1} = (A)_{\zcenter} - (A)_{z^{(t - 1)}}, M_{L+1} = 0_{m \times n})$. Importantly, in order to bound the growth in size $\size(\innerpath) - \size(\cP)$ caused by concatenating this $(L+1)$-th term, $\zcenter$ is carefully chosen to ensure that $\size(\innerpath) - \size(\cP) \leq O(\alpha^2)$. Correspondingly, each inner loop call runs in 
\begin{align*}
    \Tilde{O}\paren{ L^2 \tau^{-2} ( \size(\cP) - \size(\cP') + \alpha^2) + \tau \alpha^{-1} } \text{ matvecs}.
\end{align*}
After $\Tilde{O}(1)$ calls to the inner loop, the bisection search procedure finds $\alpha^{(t)}, z^{(t)}, \cP'$ such that (i) $\cP'$ is a matrix-approximation path to $z^{(t-1)}$, and (ii) either $\sum_{w \in \uset^{(t)}} \breg{w}{z^{(t)}} \geq C (\alpha^{(t)})^2$ for an absolute constant $C > 0$ or else $\alpha^{(t)} = \beta$ (namely, the natural extension of the kineticness condition \eqref{eq:kineticness} to our outer loop).

Finally, using the dyadic decomposition and amortization argument laid out in Section~\ref{subsec:technique-2-dyadic-decompositions} and the movement bound obtained in Section~\ref{subsec:technique-3-prox-multi-point}, we obtain an overall matvec complexity of 
\begin{align*}
    \Tilde{O}\paren{ \tau^{-2} + (\beta\epsilon^{-1} + \epsilon^{-2/3}) (\tau^{-2} \beta^2 + \tau \beta^{-1}) }, 
\end{align*}
which yields the desired $\Tilde{O}(\epsilon^{-2/3})$-matvec complexity when $\tau = \beta = \epsilon^{1/3}$. 

\subsection{Paper organization}\label{sec:paper-organization}

In this section, having described our algorithmic approach, we give a roadmap of the remaining sections and the key components of our formal proofs. To make it more clear precisely what assumptions each component uses and to hopefully enable future work, we often work in more general setups in these sections (as opposed to only $\ellOneEllOne$/$\ellTwoEllOne$ matrix games), as discussed further below. We ultimately verify that all of the more general assumptions we make in Sections~\ref{sec:combined-outer-loop-section},~\ref{sec:MDMP-implementation},~and~\ref{sec:sug-solver} hold for $\ellOneEllOne$/$\ellTwoEllOne$ matrix games in \Cref{sec:putting-together}, where we prove Theorems~\ref{thm:final-result-l1-l1-aka-zero-sum}~and~\ref{thm:final-result-l2-l1-aka-SVM}. We hope that the technical tools we introduce in this work, particularly our multi-point/dyadic prox methods and accompanying tools for controlling movement over sequences of variational inequalities, may find broader use.

In \Cref{sec:combined-outer-loop-section}, we give the formal statement and analysis of our outer loop. Specifically, in \Cref{subsec:prox-multi-general-monotone-ops} we give our general prox multi-point method for obtaining $\epsilon$-regret with respect to a monotone operator. In particular, we formally prove the movement bound discussed in detail above (\Cref{lem:multiprox-correctness}) and extend the iteration bound of \cite{karmarkar2025solvingzerosumgames} to our outer loop (\Cref{lem:multiprox-iteration-bound}). At the end of \Cref{subsec:prox-multi-general-monotone-ops}, we formally define the dyadic prox method (\Cref{def:dyadic-prox-multi-point-method}). \Cref{subsec:prox-multi-matrix-games} contains the formal statement of our outer loop for $\ellTwoEllOne$/$\ellOneEllOne$ matrix games (\Cref{alg:final-algo-outer-loop}), except stated in a slightly more general setting. In particular, \Cref{alg:final-algo-outer-loop} adds model initialization and updating to the (more general) dyadic prox method, including constructing matrix-approximation paths to pass to the bisection search procedure (although this is abstracted through an oracle). 

In Section~\ref{sec:MDMP-implementation}, we give our bisection search procedure. As mentioned, our bisection search procedure is 
a modification of that of \citep{karmarkar2025solvingzerosumgames} to handle the multiple points $\cU$ selected in the multi-prox outer-loop for regularization. Concretely, our bisection search reduces a single step of our dyadic prox method to (approximately) solving a sequence of $\Tilde{O}(1)$ \emph{constrained prox multi-point problems} (Definition~\ref{def:subproblem}) which are problems of the form \eqref{eq:ultimate-subproblem}, where $\xsetstable \times \ysetstable$ are carefully selected to be Hessian stable regions (Definition~\ref{def:stable-region-hess} and Definition~\ref{def:best-response-stability}). In Section~\ref{sec:sug-solver}, we show that under mild assumptions, a single such \emph{constrained prox multi-point problem} (Definition~\ref{def:subproblem}) can be solved to high-accuracy using our smooth-until proven guilty inner loop. (Recall that this inner loop is only a slight modification of that of \citep{karmarkar2025solvingzerosumgames} in that it handles matrix-approximation paths). 

Theorems~\ref{thm:final-result-l1-l1-aka-zero-sum}~and~\ref{thm:final-result-l2-l1-aka-SVM} are proven in Section~\ref{sec:putting-together}. Because each intermediate section (Sections~\ref{sec:combined-outer-loop-section} through~\ref{sec:sug-solver}) introduces some assumptions about the setup, in Section~\ref{sec:general-analysis}, we summarize our assumptions and results for general setups. Then, in Section~\ref{sec:applications} we verify the relevant assumptions and specify the parameter settings (e.g., of $\tau, \beta$) which yield our final results for $\ellTwoEllOne$ and $\ellOneEllOne$ games.

 \addtocontents{toc}{\protect\setcounter{tocdepth}{2}}

\section{Prox multi-point method outer loop}
\label{sec:combined-outer-loop-section}

In this section, we give the formal statement and analysis of our outer loop, albeit in more general settings which are instantiated for $\ellTwoEllOne$ and $\ellOneEllOne$ matrix games in \Cref{sec:putting-together}. In Section~\ref{subsec:prox-multi-general-monotone-ops}, we give our \emph{prox multi-point method} for general continuous monotone operators, and also provide a particularization of this general method, termed the \emph{dyadic prox method,} which is used in our application to matrix games. In the context of our ultimate outer loop for $\ellOneEllOne$ and $\ellTwoEllOne$ matrix games, \Cref{subsec:prox-multi-general-monotone-ops} can be viewed as proving correctness along with an iteration bound and iterate movement bound. \Cref{subsec:prox-multi-matrix-games} then specializes the dyadic prox method to (general) matrix games.
In particular, Section~\ref{subsec:prox-multi-matrix-games} handles model creation and clearing, as well as passing models to our bisection search procedure (although this is abstracted through an oracle). By carefully applying a movement bound between iterates given in Section~\ref{subsec:prox-multi-general-monotone-ops}, Section~\ref{subsec:prox-multi-matrix-games} forms the backbone of the amortized analysis discussed in Section~\ref{sec:overview-of-approach} (and formally computed in Section~\ref{sec:putting-together})
 to bound the total number of matvecs made over all inner loop model-update steps.

But first, we give some additional preliminaries which will be used in \Cref{sec:combined-outer-loop-section} and throughout the rest of the paper. In particular, we formally define regret with respect to an operator and restate a standard result which reduces computing an $\epsilon$-solution to achieving $\epsilon$-regret. Then, we give some additional notation pertaining to multisets and the natural extension of proximal mappings (\Cref{def:proximal-mappings}) to regularization over multisets.

\paragraph{Reducing minimax optimization to regret minimization.}
As is standard in the literature, our algorithms obtain $\epsilon$-solutions by achieving $\epsilon$-regret with respect to an appropriate monotone operator. We define regret with respect to a general operator $g$ in \Cref{def:regret}, and then we restate a standard result in \Cref{lem:regret-bounds-the-gap} which reduces obtaining an $\epsilon$-solution of a minimax optimization problem in $f$ to regret minimization with respect to $\gm f$. Note that in both cases we allow nonuniform weights $\lambda^{(t)} / \Lambda$; this will be important for our algorithms in \Cref{sec:combined-outer-loop-section}.

\begin{definition}[Regret]
    \label{def:regret}
With $\zset' \subseteq \zset \subset \R^d$, let $g : \zset \to \R^d$; $z^{(1)}, \dots, z^{(T)} \in \zset$; and $\lambda^{(1)}, \dots, \lambda^{(T)} > 0$. With $\Lambda \defeq \sum_{t \in [T]} \lambda^{(t)}$, we define
\begin{align*}
    \regret_g(\inbraces{z^{(t)}, \lambda^{(t)}}_{t \in [T]}; \zset') \defeq \sup_{u \in \zset'} \inbraces*{
        \frac{1}{\Lambda} \sum_{t \in [T]} \lambda^{(t)} \inangle{g(z^{(t)}), z^{(t)} - u},
    }
\end{align*}
where $\regret_g(\inbraces{z^{(t)}, \lambda^{(t)}}_{t \in [T]}; \zset')$ is called the \emph{regret of the sequence $z^{(1)}, \dots, z^{(T)}$ (with respect to the operator $g$, weights $\inbraces{\lambda^{(t)}}_{t \in [T]}$, and set $\zset'$).} We may drop $\zset'$, writing $\regret_g(\inbraces{z^{(t)}, \lambda^{(t)}}_{t \in [T]})$, when $\zset' = \zset$ for brevity.
\end{definition}

\begin{lemma}[Lemma 2.3 of \citep{karmarkar2025solvingzerosumgames}, restated]
    \label{lem:regret-bounds-the-gap}
    Let $f : \xset \times \yset \to \R$ be a differentiable convex-concave function over compact, convex sets $\xset \subset \R^n$ and $\yset \subset \R^m$, with $\zset \defeq \xset \times \yset$. Then for any $z^{(1)}, \dots, z^{(T)} \in \zset $ and $\lambda^{(1)}, \dots, \lambda^{(T)} > 0$, letting $\Lambda \defeq \sum_{t \in [T]} \lambda^{(t)}$ and $\zbar \defeq \frac{1}{\Lambda} \sum_{t \in [T]} \lambda^{(t)} z^{(t)}$, we have $\gap(\zbar) \le \regret_{\gm f}(\inbraces{z^{(t)}, \lambda^{(t)}}_{t \in [T]}; \zset)$.
\end{lemma}

\paragraph{Notation for finite multisets $\cU$.} Recall from the technical overview (\Cref{sec:overview-of-approach}) that a key feature of the prox multi-point method is that it extends the standard prox point method to allow for regularization over finite multisets $\cU \subset \cZ$.
Correspondingly, we frequently work with \emph{sums} of Bregman divergences over a multiset $\cU$. Consequently, for a dgf setup $\dgfsetup = (\zset \subset \R^d, r)$, a finite nonempty multiset $\uset \subset \zset$, and $z' \in \zset$, we may use the notation $\breg{\uset}{z'} \defeq \sum_{w \in \uset} \breg{w}{z'}$ for brevity. Similarly, given a multiset $\cU \subset \cZ$, we let $\cU\x \defeq \{u\x: u \in \cU\}$ and $\cU\y \defeq \{u\y : u \in \cU\}$ where $\cU\x, \cU\y$ are defined as multisets with multiplicity so that $|\cU|=|\cU\x|=|\cU\y|$. (Note that $\cU\x, \cU\y$ may have multiplicities even if $\uset$ has no repeated elements. For example, if there exist $u, v \in \cU$ such that $u\x = v\x$, then $\cU\x$ will contain both $u\x$ and $v\x$.) Additionally, departing from \citep{karmarkar2025solvingzerosumgames}, for notational convenience when describing and analyzing our prox multi-point method, we use the following additional notation. For a finite nonempty multiset $\cU \subset \cZ$ and $\zset' \subseteq \zset$, we write $\prox_{\cU}^\lambda(g; \cZ')$ to denote 
the unique $z' \in \cZ'$ such that 
\begin{align}\label{eq:set-prox-mapping} 
    \inangle{g(z'), z' - u}
\leq \lambda \sum_{z \in \cU} \insquare{\breg{z}{u} - \breg{z'}{u} - \breg{z}{z'}} \text{ for all $u \in \cZ'$}.
\end{align}

\subsection{Prox multi-point method for general monotone operators}
\label{subsec:prox-multi-general-monotone-ops}

In this section, we provide our \emph{prox multi-point method} (Algorithm~\ref{alg:multiprox-method-monotone-op}) as well as a specialization termed the \emph{dyadic prox method} (Definition~\ref{def:dyadic-prox-multi-point-method}). For $\epsilon > 0$, the prox multi-point method obtains $\epsilon$-regret with respect to a general monotone operator $g : \zset \to \R^d$; namely, it obtains sequences $z^{(1)}, \dots, z^{(T)} \in \zset$ and $\lambda^{(1)}, \dots, \lambda^{(T)} > 0$ such that $\regret_g(\inbraces{z^{(t)}, \lambda^{(t)}}_{t \in [T]}) \le \epsilon$ (recall Definition~\ref{def:regret}). In the next section (Section~\ref{subsec:prox-multi-matrix-games}), we  use this regret bound to bound the gap in our matrix games applications via Lemma~\ref{lem:regret-bounds-the-gap}.

\begin{assumptions}
In this section (Section~\ref{subsec:prox-multi-general-monotone-ops}), we fix a dgf setup $\dgfsetup = (\zset, r)$ per Definition~\ref{def:dgf-setup} with $\Gamma_\dgfsetup \ge \max_{z, z' \in \zset} r(z) - r(z')$ and a continuous monotone operator $g : \zset \to \R^d$. (We note that the results of \Cref{subsec:prox-multi-general-monotone-ops} only require $r$ to be differentiable as opposed to twice differentiable, as we do not use local norms in this section.)
\end{assumptions}

First, in the following Definition~\ref{def:DMP} we define a key oracle to which our algorithm will assume access. A $\DMP$ oracle approximately solves a strongly monotone variational inequality \eqref{eq:DMP-def-variational-property} with respect to the operator $g + \alpha \grad \breg{\uset}{\cdot}$ (sometimes referred to as an \emph{$\epsprim$-strong solution} \cite{lin2025perseus,monteiro2010complexity}); in particular, it approximates $\prox_{\cU}^\alpha(g ; \zset)$. (Note that $\grad \breg{\uset}{z'} = \sum_{w \in \uset} \grad \breg{w}{z'}$ for all $z' \in \zset$ by linearity of the gradient.)
Additionally, we say a $\DMP$ oracle is \emph{kinetic} if it either uses a default level of regularization $\beta$, or else certifies progress by lower bounding the movement of the output. Definition~\ref{def:DMP} can be viewed as the natural extension of \cite[Def. 4.1]{karmarkar2025solvingzerosumgames} to regularization about multiple points.

\begin{definition}[$\epsilon'$-$\DMP$] \label{def:DMP}
    For $\epsprim > 0$, we call $\ODMP(\cdot)$ an \emph{$\epsprim$-dynamic multiprox oracle} or $\epsprim$-$\DMP$ (with respect to the operator $g$ and setup $\dgfsetup$) if given a finite, nonempty multiset $\uset \subset \zset$ as input, it returns $(z' \in \zset, \alpha > 0)$
    such that
\begin{align}
    \label{eq:DMP-def-variational-property}
    \inangle{g(z') + \alpha \grad \breg{\uset}{z'}, z' - u} \le \epsilon' \text{~for all $u \in \zset$}.
\end{align}
For $\beta, \gamma, \rho > 0$, we say an $\epsilon'$-$\DMP$ oracle is \emph{$(\beta, \gamma, \rho)$-kinetic} if additionally the output always satisfies at least one of 
\begin{enumerate*}[series = tobecont, itemjoin =, label=(\alph*)]
    \item $\alpha = \beta$ 
    or \label{item:DMP-beta-cond}
    \item $\breg{\cU}{z'} \ge \gamma \alpha^\rho$. \label{item:DMP-movement-cond}
\end{enumerate*}
\end{definition}

Next, we present the prox multi-point method in Algorithm~\ref{alg:multiprox-method-monotone-op}. Algorithm~\ref{alg:multiprox-method-monotone-op} maintains a sequence of iterates $z^{(1)}, z^{(2)}, \dots$ outputted by the $\DMP$ oracle, as well as $K$ sequences of regularization center points, where the $k$-th sequence for $k \in [K]$ is denoted $w_k^{(0)}, w_k^{(1)}, \dots$. At each iteration $t$, a set of \emph{active centers} is dynamically chosen via the index set $\idset^{(t)}$ in Line~\ref{line:multiprox-choose-active-centers}. The corresponding centers $\inbraces{w^{(t - 1)}_k : k \in \idset^{(t)}}$ are passed to the $\DMP$ oracle in Line~\ref{line:multiprox-DMP} to obtain $z^{(t)}$ and $\alpha^{(t)}$; note that here $\inbraces{w^{(t - 1)}_k : k \in \idset^{(t)}}$ is a multiset (i.e., multiplicity is preserved). Then, the regularization centers for the next iteration $w_k^{(t)}$ are set in Line~\ref{line:multiprox-w^{(t)}_k-unified-def}. Centers which were active in the current iteration are updated to $z^{(t)}$; otherwise they retain their previous value. This ensures appropriate telescoping occurs in the analysis. 
Finally, note that the conditional in Line~\ref{line:multiprox-while-loop} evaluates to $\true$ when $t = 0$ due to the convention from Section~\ref{sec:prelims} that a summation over an empty index set is 0.

\RestyleAlgo{ruled}
\DontPrintSemicolon
\SetKwComment{Comment}{/* }{ */}
\begin{algorithm2e}[h!]
\caption{Prox multi-point method}
\label{alg:multiprox-method-monotone-op}
\KwInput{Precision $\epsilon > 0$, max centers per step $K \in \Z_{> 0}$, 
 $\epsilon$-$\DMP$ oracle $\ODMP$
}

$z^{(0)} \gets \argmin_{z \in \zset} r(z)$ ~and~ $t \gets 0$\;

$w^{(0)}_k \gets z^{(0)}$ for all $k \in [K]$ \;

\While(\tcp*[f]{Recall $\Range_\dgfsetup \ge \max_{z, z' \in \zset} r(z) - r(z')$}){
    $\sum_{j \in [t]} (\alpha^{(j)})^{-1} < K \Range_\dgfsetup \epsilon^{-1}$ 
}{ \label{line:multiprox-while-loop}

    $t \gets t + 1$ \;

    Choose nonempty $\idset^{(t)} \subseteq [K]$ \tcp*{Dynamically select the active center indices at the $t$-th step} \label{line:multiprox-choose-active-centers}

    $(z^{(t)}, \alpha^{(t)}) \gets \ODMP(\inbraces{w^{(t - 1)}_k : k \in \idset^{(t)}})$ \label{line:multiprox-DMP}

    $w^{(t)}_k \gets \begin{cases} 
      z^{(t)},  & \text{for all } k \in \idset^{(t)} \\
      w^{(t - 1)}_k,  & \text{for all } k \in [K] \setminus \idset^{(t)}
   \end{cases}$ \label{line:multiprox-w^{(t)}_k-unified-def}

}

\Return{$\inbraces{z^{(j)}, \alpha^{(j)}}_{j \in [T]}$ where $T \defeq t$ \tcp*{$T$ is used in the analysis to refer to the final iteration count}
}\label{line:multiprox-return}

\end{algorithm2e}

Next, we give our correctness guarantee as well as a movement bound over the sequences of regularization centers in Lemma~\ref{lem:multiprox-correctness}. Recall from the discussion in Section~\ref{sec:overview-of-approach} that this movement bound can be viewed as enabling control over $K$ subsequences of the iterates $z^{(0)}, z^{(1)}, \dots$ as opposed to the movement bound over the single sequence $z^{(0)}, z^{(1)}, \dots$ obtained by the standard prox point method, albeit at the cost of an additional $K$ factor in the bound. 

Note that unlike in \Cref{sec:overview-of-approach}, the movement bound in \Cref{lem:multiprox-correctness} when the algorithm terminates is over $t \in [T - 1]$ instead of $t \in [T]$. (More generally, we are always able to obtain a movement bound over all iterates other than the final iterate; this unites the two cases of \Cref{lem:multiprox-correctness}.) As is evident from the proof (see \eqref{eq:key-eq-dist-regret-bound}), this is due to the fact that in this section each step taken in \Cref{line:multiprox-DMP} involves solving an approximate variational inequality per the discussion above, whereas in \Cref{sec:overview-of-approach} each step solved the corresponding exact variational inequality for simplicity. However, as discussed further in \Cref{subsec:prox-multi-matrix-games}, the movement bound over $t \in [T - 1]$ suffices for our purposes.

\begin{lemma}
    \label{lem:multiprox-correctness}
    If \Cref{alg:multiprox-method-monotone-op} terminates, then $\regret_g(\inbraces{z^{(t)}, (\alpha^{(t)})^{-1}}_{t \in [T]})  \le 2 \epsilon$ and \\ $\sum_{t \in [T - 1]} \sum_{k \in [K]} \breg{w^{(t - 1)}_k}{w^{(t)}_k} \le 2 K \Gamma_\dgfsetup$. Otherwise, $\sum_{t \ge 1} \sum_{k \in [K]} \breg{w^{(t - 1)}_k}{w^{(t)}_k} \le 2 K \Gamma_\dgfsetup$.
\end{lemma}

\begin{proof}
The definition of $z^{(t)}$ in \Cref{line:multiprox-DMP}, along with \Cref{def:DMP} and \eqref{eq:Bregman-three-point-equality}, imply that for all $t \ge 1$ and $u \in \zset$:
\begin{align*}
    (\alpha^{(t)})^{-1} \inangle{g(z^{(t)}), z^{(t)} - u} &\le \sum_{k \in \idset^{(t)}} [\breg{w_k^{(t - 1)}}{u} - \breg{z^{(t)}}{u} - \breg{w_k^{(t - 1)}}{z^{(t)}}] + (\alpha^{(t)})^{-1} \epsilon \\
    &= \sum_{k \in [K]} [\breg{w_k^{(t - 1)}}{u} - \breg{w_k^{(t)}}{u} - \breg{w_k^{(t - 1)}}{w_k^{(t)}}] + ( \alpha^{(t)})^{-1} \epsilon,
\end{align*}
where the equality followed from the definition of $w_k^{(t)}$ in Line~\ref{line:multiprox-w^{(t)}_k-unified-def}; note in particular 
that the expression within the final summation is zero for $k \in [K] \setminus \idset^{(t)}$. Letting $t' \ge 1$, summing both sides over $t \in [t']$, dividing by $S_{t'} \defeq \sum_{t \in [t']} (\alpha^{(t)})^{-1}$ (defined as a function of $t'$), and using the nonnegativity of Bregman divergences yields
\begin{align}
    \label{eq:key-eq-dist-regret-bound}
    \frac{ K \Gamma_\dgfsetup - \sum_{t \in [t']} \sum_{k \in [K]}  \breg{w_k^{(t - 1)}}{w_k^{(t)}} }{S_{t'}} + \epsilon \ge
        \sup_{u \in \zset} \inbraces*{\frac{1}{S_{t'}} \sum_{t \in [t']} (\alpha^{(t)})^{-1} \inangle{ g(z^{(t)}), z^{(t)} - u }} \overge{(i)} 0,
\end{align}
where $(i)$ follows since regret with respect to a monotone operator is nonnegative (e.g., \cite[Proposition A.1]{karmarkar2025solvingzerosumgames}). Here, we also used the fact that $w_k^{(0)} = z^{(0)}$ for all $k \in [K]$ and $\breg{z^{(0)}}{u} \le \Gamma_\dgfsetup$ for all $u \in \zset$ since $z^{(0)} \in \argmin_{z \in \zset} r(z)$.

Then supposing \Cref{alg:multiprox-method-monotone-op} terminates, the first claim follows by instantiating $t' \gets T$ in \eqref{eq:key-eq-dist-regret-bound} and noting $S_T \ge K \Gamma_\dgfsetup \epsilon^{-1}$ due to the termination condition in Line~\ref{line:multiprox-while-loop}. As for the second claim, the case $T = 1$ is trivial. Otherwise, taking $t' = T - 1$ and using the termination condition in \Cref{line:multiprox-while-loop}, we get
\begin{align*}
    \sum_{t \in [T - 1]} \sum_{k \in [K]} \breg{w^{(t - 1)}_k}{w^{(t)}_k} \le K \Gamma_\dgfsetup + \epsilon \cdot K \Gamma_\dgfsetup \epsilon^{-1} = 2 K \Gamma_\dgfsetup .
\end{align*}

Similarly, supposing \Cref{alg:multiprox-method-monotone-op} does not terminate, \eqref{eq:key-eq-dist-regret-bound} as well as the termination condition in \Cref{line:multiprox-while-loop} imply that for all $t' \ge 1$,
\begin{align*}
    \sum_{t \in [t']} \sum_{k \in [K]} \breg{w^{(t - 1)}_k}{w^{(t)}_k} \le K \Gamma_\dgfsetup + \epsilon \cdot K \Gamma_\dgfsetup \epsilon^{-1} = 2 K \Gamma_\dgfsetup .
\end{align*}
\end{proof}

Next, when the $\DMP$ is kinetic, we bound the number of iterations $T$ as well as the sum of the regularization levels $\alpha^{(t)}$ raised to the $\rho$ power. Recall from Sections~\ref{subsec:technique-1-telescoping-sums}~and~\ref{subsec:technique-2-dyadic-decompositions} that the latter is ultimately needed since $(\alpha^{(t)})^\rho$ with our ultimate choice of $\rho \gets 2$ appears in the complexity of each inner loop call during the $t$-th iteration of our outer loop for $\ellTwoEllOne$ and $\ellOneEllOne$ matrix games. Note that \Cref{lem:multiprox-iteration-bound} does not bound $\alpha^{(T)}$; we bound $\alpha^{(T)}$ separately when applying the results of \Cref{sec:combined-outer-loop-section} in \Cref{sec:putting-together}, where it is guaranteed to be at most a polylogarithmic factor. 

\begin{lemma}
    \label{lem:multiprox-iteration-bound}
If the $\epsilon$-$\DMP$ oracle given as input to Algorithm~\ref{alg:multiprox-method-monotone-op} is $(\beta, \gamma, \rho)$-kinetic, then the algorithm terminates with
\begin{align*}
   T \leq K \Gamma_\dgfsetup (\beta \epsilon^{-1} + 2^{\frac{1}{\rho + 1}} \gamma^{- \frac{1}{\rho+1}} \epsilon^{- \frac{\rho}{\rho + 1}}) + 2 ~~\text{and}~~  \sum_{t \in [T - 1]} (\alpha^{(t)})^\rho \le 2 K \Gamma_{\dgfsetup} \gamma^{-1} + (T - 1) \beta^\rho  .
\end{align*}
\end{lemma}

\begin{proof}
Let $J_a \defeq \inbraces{t \ge 1 : \alpha^{(t)} = \beta}$ and $J_b \defeq \inbraces{t \ge 1 : \alpha^{(t)} \ne \beta}$, where we restrict to values of $t$ such that $\alpha^{(t)}$ is well-defined. (In particular, note that the $\DMP$ oracle call in Line~\ref{line:multiprox-DMP} during an iteration $t$ such that $t \in J_b$ must satisfy condition \ref{item:DMP-movement-cond} in Definition~\ref{def:DMP}.)

We first prove termination. Note $|J_a| \le \beta K \Gamma_\dgfsetup \epsilon^{-1} + 1$ due to the termination condition in Line~\ref{line:multiprox-while-loop}, and thus it suffices to show that $|J_b|$ is finite. Supposing $|J_b|$ is infinite for the sake of contradiction, we have then that for every $t' \ge 1$,
\begin{align}
    \label{eq:does-not-terminate-alpha-bound}
    \begin{split}
    \sum_{t \in J_b \cap [t']} \gamma (\alpha^{(t)})^\rho \le \sum_{t \in J_b \cap [t']}  \sum_{k \in \idset^{(t)}} \breg{w_k^{(t - 1)}}{z^{(t)}} 
    &\overeq{(i)} \sum_{t \in J_b \cap [t']}  \sum_{k \in [K]} \breg{w_k^{(t - 1)}}{w_k^{(t)}} \\ 
    &\overle{(ii)} \sum_{t \in [t']}  \sum_{k \in [K]} \breg{w_k^{(t - 1)}}{w_k^{(t)}} \overle{(iii)} 2 K \Gamma_\dgfsetup,
    \end{split}
\end{align}
where we used $(i)$ the definition of $w_k^{(t)}$ in Line~\ref{line:multiprox-w^{(t)}_k-unified-def},
$(ii)$ the nonnegativity of Bregman divergences, and $(iii)$ Lemma~\ref{lem:multiprox-correctness}. Also, since we have assumed $|J_b|$ is infinite and therefore \Cref{alg:multiprox-method-monotone-op} does not terminate, for every $t' \ge 1$, we have $\sum_{t \in [t']} (\alpha^{(t)})^{-1} < K \Gamma_\dgfsetup \epsilon^{-1}$ by the termination condition in \Cref{line:multiprox-while-loop}. But this contradicts \eqref{eq:does-not-terminate-alpha-bound}, as the latter implies $\lim_{t \to \infty, \, t \in J_b} \alpha^{(t)} = 0$.

Thus, having shown \Cref{alg:multiprox-method-monotone-op} terminates, we now focus on proving the bounds on $T$ and $\sum_{t \in [T - 1]} (\alpha^{(t)})^\rho$. 
Toward bounding $|J_b|$, let $J_b' \defeq J_b \setminus \inbraces{T}$, and note
\begin{align}
    \label{eq:J_b'-movement-bound-after-terminate}
    \sum_{t \in J_b'} \gamma (\alpha^{(t)})^\rho \le \sum_{t \in J_b'} \sum_{k \in \idset^{(t)}} \breg{w_k^{(t - 1)}}{z^{(t)}} = \sum_{t \in J_b'} \sum_{k \in [K]} \breg{w_k^{(t - 1)}}{w_k^{(t)}} \le 2 K \Gamma_\dgfsetup
\end{align}
by \Cref{lem:multiprox-correctness}. Then
\begin{align*}
    |J_b'| = \sum_{t \in J_b'} (\alpha^{(t)})^{\frac{\rho}{\rho + 1}} (\alpha^{(t)})^{- \frac{\rho}{\rho + 1}} 
    &\overle{(i)} 
    \inparen*{
        \sum_{t \in J_b'} (\alpha^{(t)})^\rho
    }^{\frac{1}{\rho + 1}}
     \inparen*{
        \sum_{t \in J_b'} (\alpha^{(t)})^{-1} 
     }^{\frac{\rho}{\rho + 1}} \\
     &=   \gamma^{- \frac{1}{\rho+1}}  \inparen*{
        \sum_{t \in J_b'} \gamma (\alpha^{(t)})^\rho
    }^{\frac{1}{\rho + 1}}
     \inparen*{
        \sum_{t \in J_b'} (\alpha^{(t)})^{-1} 
     }^{\frac{\rho}{\rho + 1}}
     \\
     &\overle{(ii)} \gamma^{- \frac{1}{\rho+1}} (2 K \Gamma_\dgfsetup)^{\frac{1}{\rho + 1}} (K \Gamma_\dgfsetup \epsilon^{-1})^{\frac{\rho}{\rho + 1}} \\
     &= 2^{\frac{1}{\rho + 1}} \gamma^{- \frac{1}{\rho+1}} K \Gamma_\dgfsetup  \epsilon^{- \frac{\rho}{\rho + 1}},
\end{align*}
by $(i)$ \Holder's inequality and $(ii)$ \eqref{eq:J_b'-movement-bound-after-terminate} as well as the fact that $\sum_{t \in J_b'} (\alpha^{(t)})^{-1} \le \sum_{t \in [T - 1]} (\alpha^{(t)})^{-1} < K \Gamma_\dgfsetup \epsilon^{-1}$ by the termination condition in Line~\ref{line:multiprox-while-loop}. To obtain the 
desired upper bound on $T$, note
\begin{align*}
    T = |J_a| + |J_b| \le |J_a| + |J_b'| + 1 \le K \Gamma_\dgfsetup (\beta \epsilon^{-1} + 2^{\frac{1}{\rho + 1}} \gamma^{- \frac{1}{\rho+1}} \epsilon^{- \frac{\rho}{\rho + 1}}) + 2.
\end{align*}
As for the bound on the sum of $(\alpha^{(t)})^\rho$, note that by \eqref{eq:J_b'-movement-bound-after-terminate},
\begin{align*}
\sum_{t \in [T - 1]} (\alpha^{(t)})^\rho \le (T - 1) \beta^\rho + \sum_{t \in J_b'} (\alpha^{(t)})^\rho  \le (T - 1) \beta^\rho + 2 K \Gamma_\dgfsetup \gamma^{-1} .
\end{align*}
\end{proof}

Next, we formally define the \emph{dyadic prox method} discussed in Section~\ref{sec:overview-of-approach}. The dyadic prox method fixes the choice $\idset^{(t)} \gets \inbraces{k \in [K] : \text{$t$ is divisible by $2^{K - k}$}}$, which results in the movement bound in Lemma~\ref{lem:multiprox-correctness} controlling pairs of iterates of the form $z^{(m \cdot 2^\ell)}, z^{((m + 1) \cdot 2^\ell)}$;
recall \eqref{eq:overview-interpretation-of-movement-bound} from Section~\ref{sec:overview-of-approach} and the surrounding discussion. We note that whenever we use Definition~\ref{def:dyadic-prox-multi-point-method} in this paper, we also ensure (either by assumption or by making an explicit choice of $K$) that $K \ge \log_2 T + 5$ (recall $T$ is the final iteration count per Line~\ref{line:multiprox-return}); this ensures that we are indeed controlling every such pair up to $T$.

\begin{definition}[Dyadic prox method]
    \label{def:dyadic-prox-multi-point-method}
We refer to Algorithm~\ref{alg:multiprox-method-monotone-op} with the choice $\idset^{(t)} \gets \inbraces{k \in [K] : \text{$t$ is divisible by $2^{K - k}$}}$ in Line~\ref{line:multiprox-choose-active-centers} for all $t \ge 1$ as the \emph{dyadic prox method}.
\end{definition}

In the following lemma, for the dyadic prox method, we connect the specific (in fact, weaker) movement bound used in the next section (Section~\ref{subsec:prox-multi-matrix-games}) for our amortized analysis to the movement bound given in Lemma~\ref{lem:multiprox-correctness}. In particular, the first inequality in \eqref{eq:dyadic-prox-multi-point-method-movement-bound} is precisely the first inequality $(i)$ in \eqref{eq:overview-multiprox-movement-bound}, except stated in a different form which will be useful in Section~\ref{subsec:prox-multi-matrix-games}.

\begin{lemma}
    \label{lem:dyadic-prox-movement-bound}
    Supposing $K \ge \log_2 T + 5$ (where $T$ is the final iteration count per Line~\ref{line:multiprox-return}), the iterates of the dyadic prox method (Definition~\ref{def:dyadic-prox-multi-point-method}) satisfy
\begin{align}
    \label{eq:dyadic-prox-multi-point-method-movement-bound}
  \sum_{t \in [T - 1]} \sum_{k \in \idset^{(t)}} \breg{w^{(t)}_{k - 1}}{w^{(t)}_{k}} \le  \sum_{t \in [T - 1]} \sum_{k \in [K]} \breg{w^{(t - 1)}_k}{w^{(t)}_k} \le 2 K \Range_\dgfsetup .
\end{align}
Furthermore, if $(w_{k - 1}^{(t - 1)}, w_{k}^{(t - 1)}) \ne (w_{k - 1}^{(t)}, w_{k}^{(t)})$ for some $t \ge 1$ and $2 \le k \le K$, then $k \in \idset^{(t)}$.
\end{lemma}

\begin{proof}
Note that for any $t \in [T]_0$ and $k \in [K]$, we have that $w_{k}^{(t)} = z^{(a_{k, t})}$ where $a_{k, t} \defeq t - (t \bmod 2^{K - k})$; namely, $a_{k, t}$ is the largest multiple of $2^{K - k}$ which is at most $t$. Indeed, this follows from the fact that by the choice of $\idset^{(t)}$ in Definition~\ref{def:dyadic-prox-multi-point-method}, the subsequence of $\idset^{(1)}, \idset^{(2)}, \idset^{(3)}, \dots$ such that each set in the subsequence contains $k$ is precisely $\idset^{2^{K - k}}, \idset^{2(2^{K - k})}, \idset^{3(2^{K - k})}, \dots$. Hence, by Line~\ref{line:multiprox-w^{(t)}_k-unified-def} of Algorithm~\ref{alg:multiprox-method-monotone-op}, $w_{k}^{(t)}$ is updated to $z^{(t)}$ in iterations $t$ such that $t$ is a multiple of $2^{K - k}$, and otherwise retains its previous value $w_{k}^{(t - 1)}$.

Let us now examine the leftmost summation in \eqref{eq:dyadic-prox-multi-point-method-movement-bound}; in particular, fix an arbitrary $t \in [T - 1]$ and consider $\sum_{k \in \idset^{(t)}} \breg{w^{(t)}_{k - 1}}{w^{(t)}_{k}}$. Note that the latter is well-defined since $K \ge \log_2 T + 5$ implies $1 \notin \idset^{(t)}$. (In other words, the fact that $w_0^{(t)}$ is not defined does not pose an issue.) Then, letting $\ktstar \defeq \min_{k \in \idset^{(t)}} k$ (note $\idset^{(t)}$ is nonempty since $K \in \idset^{(t)}$), we claim
\begin{align}
    \label{eq:evaluating-the-idset-sum}
    \sum_{k \in \idset^{(t)}} \breg{w^{(t)}_{k - 1}}{w^{(t)}_{k}} = \breg{w^{(t)}_{\ktstar - 1}}{w^{(t)}_{\ktstar}}.
\end{align}
This follows since for any $k \in [K - 1]$, we have that $k \in \idset^{(t)}$ implies $k + 1 \in \idset^{(t)}$ by the choice of $\idset^{(t)}$ (if $t$ is divisible by $2^{K - k}$ then it is also divisible by $2^{K - k - 1}$). Thus, $w_{k - 1}^{(t)} = w_{k}^{(t)} = z^{(t)}$ for all $\ktstar + 1 \le k \le K$ by Line~\ref{line:multiprox-w^{(t)}_k-unified-def} of Algorithm~\ref{alg:multiprox-method-monotone-op}. 

Then to prove \eqref{eq:dyadic-prox-multi-point-method-movement-bound}, it suffices to show $w^{(t)}_{\ktstar - 1} = w^{(t - 1)}_{\ktstar}$, as combining the latter with \eqref{eq:evaluating-the-idset-sum} yields
\begin{align*}
    \sum_{k \in \idset^{(t)}} \breg{w^{(t)}_{k - 1}}{w^{(t)}_{k}} = \breg{w^{(t)}_{\ktstar - 1}}{w^{(t)}_{\ktstar}} = \breg{w^{(t - 1)}_{\ktstar}}{w^{(t)}_{\ktstar}} \le \sum_{k \in [K]} \breg{w^{(t - 1)}_k}{w^{(t)}_k}
\end{align*}
and recall $t \in [T - 1]$ was set arbitrarily. (The second inequality in \eqref{eq:dyadic-prox-multi-point-method-movement-bound} is immediate from Lemma~\ref{lem:multiprox-correctness}.) As for proving $w^{(t)}_{\ktstar - 1} = w^{(t - 1)}_{\ktstar}$, it suffices to show $a_{\ktstar - 1, t} = a_{\ktstar, t - 1}$ by the above general characterization of $w_k^{(t)}$. In other words, we need to show that the largest multiple of $2^{K - \ktstar + 1}$ which is at most $t$ is equal to the largest multiple of $2^{K - \ktstar}$ which is at most $t - 1$. This follows because $t$ is a multiple of $2^{K - \ktstar}$ due to the fact that $\ktstar \in \idset^{(t)}$, but $t$ is not a multiple of $2^{K - \ktstar + 1}$ by the definition of $\ktstar$ (in particular $\ktstar - 1 \notin \idset^{(t)}$). Therefore, the multiple of $2^{K - \ktstar}$ before $t$ must coincide with the multiple of $2^{K - \ktstar + 1}$ before $t$.

Finally, to prove that $(w_{k - 1}^{(t - 1)}, w_{k}^{(t - 1)}) \ne (w_{k - 1}^{(t)}, w_{k}^{(t)})$ for some $t \ge 1$ and $2 \le k \le K$ implies $k \in \idset^{(t)}$, we will prove the contrapositive; namely, $k \in [K] \setminus \idset^{(t)}$ implies $(w_{k - 1}^{(t - 1)}, w_{k}^{(t - 1)}) = (w_{k - 1}^{(t)}, w_{k}^{(t)})$. Note that for any $2 \le k \le K$, we have $k \in [K] \setminus \idset^{(t)}$ implies $k - 1 \in [K] \setminus \idset^{(t)}$ due to the choice of $\idset^{(t)}$ (if $t$ isn't divisible by $2^{K - k}$, it also isn't divisible by $2^{K - k + 1}$). Thus, $(w_{k - 1}^{(t - 1)}, w_{k}^{(t - 1)}) = (w_{k - 1}^{(t)}, w_{k}^{(t)})$ for all $k$ such that $k \in [K] \setminus \idset^{(t)}$ and $k \ge 2$ by Line~\ref{line:multiprox-w^{(t)}_k-unified-def}.
\end{proof}

\subsection{Prox multi-point method outer loop for matrix games}
\label{subsec:prox-multi-matrix-games}

In this section, we provide and analyze the outer loop of our ultimate algorithm for obtaining Theorems \ref{thm:final-result-l1-l1-aka-zero-sum} and \ref{thm:final-result-l2-l1-aka-SVM} (albeit, stated in a more general setting which is instantiated for $\ellOneEllOne$ and $\ellTwoEllOne$ matrix games in \Cref{sec:putting-together}). In particular, the main guarantee of this section, Theorem~\ref{thm:matrix-games-outer-loop-guarantee}, provides key bounds for our amortized analysis of the total number of matvecs made over all model-update steps within inner loop calls.

\begin{assumptions}
In this section (Section~\ref{subsec:prox-multi-matrix-games}), we fix arbitrary dgf setups $\dgfsetup\x = (\xset \subset \R^n, \rx)$ and $\dgfsetup\y = (\yset \subset \R^m, \ry)$ with $\dgfsetup = (\zset \subset \R^d, r) \defeq \prodsetup(\dgfsetup\x, \dgfsetup\y)$ (recall Definition~\ref{def:product-dgf-setups}) and $\Gamma_\dgfsetup \ge \max_{z, z' \in \zset} r(z) - r(z')$.
For a given $A \in \R^{m \times n}$, our goal in this section is to obtain an $\epsilon$-solution of the general matrix game \eqref{eq:intro-general-matrix-game}.
Moreover, we assume throughout that $\dgfsetup$ is $\zeta$-compatible with respect to $A$ in the sense of Definition~\ref{def:zeta-compatible-mapping} given below. 
\end{assumptions}

Definition~\ref{def:zeta-compatible-mapping} abstracts a property we use to bound the total number of matvecs made within all inner loop model-update steps by the divergences between iterates (to use the notation of that section, recall the key inequality $\innorm{\Delta_{j, j'}}_F^2 \le 2 \breg{z^{(j)}}{z^{(j')}}$ from \Cref{subsec:technique-2-dyadic-decompositions}). Indeed, we later show that Definition~\ref{def:zeta-compatible-mapping} is satisfied in the context of Theorems \ref{thm:final-result-l1-l1-aka-zero-sum} and \ref{thm:final-result-l2-l1-aka-SVM} with $\zeta = 2$ in Lemma~\ref{lemma:compatibility}.

\begin{definition}[$\zeta$-compatible]
    \label{def:zeta-compatible-mapping}
For $\zeta > 0$, we say the dgf setup $\dgfsetup$ is \emph{$\zeta$-compatible with respect to a matrix $B \in \R^{m \times n}$} if $\innorm{(B)_{z'} - (B)_{z}}_F^2 \le \zeta \breg{z}{z'}$ for all $z, z' \in \zset$.
\end{definition}

We now extend Definition~\ref{def:DMP} from Section~\ref{subsec:prox-multi-general-monotone-ops} in Definition~\ref{def:MDMP} below. Note in particular that an $\epsilon'$-$\MDMP$ is an $\epsprim$-$\DMP$ with respect to the dgf setup $\dgfsetup$ and monotone operator $\gm f_A$. The only difference (or rather extension) is that an $\epsilon'$-$\MDMP$ also takes in a matrix-approximation path $\pathd = \inbraces{\Delta_\ell, M_\ell}_{\ell \in [L]}$ to some $z \in \uset$ and outputs another matrix-approximation path $\pathd' = \inbraces{\Delta_\ell, M_\ell' \in \R^{m \times n}}_{\ell \in [L]}$ to $z$ where only the models may have changed. 
As will be discussed further below, we use an $\epsilon'$-$\MDMP$ to abstract our bisection search procedure (Section~\ref{sec:MDMP-implementation}) and inner loop subproblem solver (Section~\ref{sec:sug-solver}). Regarding the requirement $z \in \uset$: in Algorithm~\ref{alg:final-algo-outer-loop} (covered next), we only pass matrix-approximation paths to the previous iterate $z^{(t - 1)}$ to an $\MDMP$ as in the overview \Cref{sec:overview-of-approach}. However, when we implement an $\MDMP$ in \Cref{sec:MDMP-implementation}, we only need the weaker assumption $z \in \uset$ to obtain our guarantee, hence its presence here.

\begin{definition}[$\epsilon'$-$\MDMP$] \label{def:MDMP}
    For $\epsprim > 0$, we call $\OMDMP(\cdot, \cdot)$ an \emph{$\epsprim$-matrix-games dynamic multiprox oracle} ($\epsprim$-$\MDMP$) if on input $(\uset \subset \zset, \pathd = \inbraces{\Delta_\ell, M_\ell}_{\ell \in [L]})$ where $\uset$ is a finite nonempty multiset and $\pathd$ is a matrix-approximation path to some $z \in \uset$, it returns $(z' \in \zset, \alpha > 0, \pathd' = \inbraces{\Delta_\ell, M_\ell' \in \R^{m \times n}}_{\ell \in [L]})$ such that: (i) the outputs $z', \alpha$ satisfy the property of Definition~\ref{def:DMP} with respect to $\gm f_A$ and $\dgfsetup$, i.e., \eqref{eq:DMP-def-variational-property} holds with $g \gets \gm f_A$, and (ii) $\pathd'$ is also a matrix-approximation path to $z$. For $\beta, \gamma, \rho > 0$, we say an $\epsilon'$-$\MDMP$ is \emph{$(\beta, \gamma, \rho)$-kinetic} if additionally the output always satisfies at least one of the conditions \ref{item:DMP-beta-cond} or \ref{item:DMP-movement-cond} from Definition~\ref{def:DMP}.
\end{definition}

Next, Algorithm~\ref{alg:final-algo-outer-loop} gives the outer loop of our ultimate algorithm for obtaining Theorems \ref{thm:final-result-l1-l1-aka-zero-sum} and \ref{thm:final-result-l2-l1-aka-SVM}. In particular, Algorithm~\ref{alg:final-algo-outer-loop} can be viewed as an instantiation of Algorithm~\ref{alg:multiprox-method-monotone-op}
for the dgf setup $\dgfsetup$ and monotone operator $\gm f_A$, with extensions to handle model creation, clearing, and passing. Indeed, note that the initialization and updates to the iterates $w_k^{(t)}$ and $z^{(t)}$ in Algorithm~\ref{alg:final-algo-outer-loop} are identical to those in Algorithm~\ref{alg:multiprox-method-monotone-op}. As for differences between the two algorithms, we specifically instantiate $\idset^{(t)}$ in Line~\ref{line:final-algo-idset} of Algorithm~\ref{alg:final-algo-outer-loop} to the choice made in Definition~\ref{def:dyadic-prox-multi-point-method}, thereby making Algorithm~\ref{alg:final-algo-outer-loop} a dyadic prox method.

The other key difference is of course the model iterates: Algorithm~\ref{alg:final-algo-outer-loop} maintains $K$ sequences of said iterates, where the $k$-th sequence for $k \in [K]$ is denoted $M_k^{(0)}, M_k^{(1)}, \dots$. In particular, $M_{k}^{(t - 1)}$ (for $t \in [T + 1]$) is a model for the difference $\Delta^{(t - 1)}_{k} \defeq (A)_{w^{(t - 1)}_{k}}  -  (A)_{w^{(t - 1)}_{k - 1}}$, where we overload notation and define $(A)_{w^{(t - 1)}_{0}} \defeq 0_{m \times n}$ for brevity (note that $\Delta^{(t - 1)}_{k}$ is defined the same way in Line~\ref{line:final-algo-set-P^(t)}).
Note that Algorithm~\ref{alg:final-algo-outer-loop} maintains the invariant
\begin{align}
    \label{eq:expl-sec4_2-for-path}
    (A)_{z^{(t - 1)}} = (A)_{w_K^{(t - 1)}} = \sum_{k \in [K]} \Delta_k^{(t - 1)} 
\end{align}
for all $t \in [T + 1]$. In other words, $\pathd^{(t)}$ (defined in Line~\ref{line:final-algo-set-P^(t)}) is a matrix-approximation path to $z^{(t - 1)}$, as in the sketch given in Section~\ref{sec:overview-of-approach}. The first equality in \eqref{eq:expl-sec4_2-for-path} follows because in fact $w_K^{(t)} = z^{(t)}$ for all $t \ge 0$ due to the fact that $K \in \idset^{(t)}$ for all $t \ge 1$ and the update rule of Line~\ref{line:final-algo-wtk-unified-def}. 

More broadly, we have $w_{k}^{(t)} = z^{(a_{k, t})}$ for all $t \ge 0$ and $k \in [K]$ where $a_{k, t} \defeq t - (t \bmod 2^{K - k})$; namely, $a_{k, t}$ is the largest multiple of $2^{K - k}$ which is at most $t$.
 Again, this is due to the choice of $\idset^{(t)}$ in Line~\ref{line:final-algo-idset} and the update rule of Line~\ref{line:final-algo-wtk-unified-def}, and it implies the terms in the rightmost summation in \eqref{eq:expl-sec4_2-for-path} are in fact tracking the differences in the binary decomposition of $t - 1$ (as long as $K$ is sufficiently large, e.g., $K \ge \log_2 T + 5$ as in Lemma~\ref{lem:dyadic-prox-movement-bound}). For example, if $K = 20$ and we are on iteration $t = 10$, then one can verify $w_{20}^{(t - 1)} = z^{(9)}$, $w_{19}^{(t - 1)} = z^{(8)}$, $w_{18}^{(t - 1)} = z^{(8)}$, $w_{17}^{(t - 1)} = z^{(8)}$, and $w_j^{(t - 1)} = z^{(0)}$ for all $1 \le j \le 16$. In other words (assuming the iterates $z^{(j)}$ are unique for simplicity), if $w_{k}^{(t - 1)} = z^{(j)}$ and $w_{k - 1}^{(t - 1)} = z^{(j')}$ for some $j \ne j'$, then the jump from $j'$ to $j$ appears in the binary decomposition of $t - 1$ (which is 9 in the above example).\footnote{We note that this is the reason mentioned in Section~\ref{sec:overview-of-approach} for why we set $\idset^{(t)} \gets \inbraces{k \in [K] : \text{$t$ is divisible by $2^{K - k}$}}$ in Definition~\ref{def:dyadic-prox-multi-point-method} instead of $\idset^{(t)} \gets \inbraces{k \in [K] : \text{$t$ is divisible by $2^{k - 1}$}}$. In the latter case, $w_1^{(t)}$ would be the head of the path instead of $w_K^{(t)}$, i.e., the directions of the paths $\pathd^{(t)}$ in Line~\ref{line:final-algo-set-P^(t)} and $\pathd'^{(t)}$ in Line~\ref{line:final-algo-MDMP} would need to be reversed, resulting in perhaps less concise indexing.}

Let us now discuss the logic of the updates to the models $M_k^{(t)}$ in Algorithm~\ref{alg:final-algo-outer-loop}. All models are initialized to $0_{m \times n}$ in Line~\ref{line:final-algo-w-M-init}. As mentioned above, the $\MDMP$ oracle call in Line~\ref{line:final-algo-MDMP} abstracts our bisection search procedure (Section~\ref{sec:MDMP-implementation}) and inner loop subproblem solver $\SUG$ (Section~\ref{sec:sug-solver}). As discussed in Section~\ref{sec:overview-of-approach}, $\SUG$ will perform updates to the models $M_k^{(t - 1)}$ within the path $\pathd^{(t)}$ in model-update steps, and thus $M_k'^{(t)}$ are the results of all these updates (potentially over many calls to $\SUG$ within the bisection search procedure). In Line~\ref{line:final-algo-wtk-unified-def}, we set the new model iterates $M_k^{(t)}$. If $k \in \idset^{(t)}$, then potentially $w_k^{(t)} \ne w_k^{(t - 1)}$ (again due to the update logic for $w_k^{(t)}$ in Line~\ref{line:final-algo-wtk-unified-def}), and therefore potentially $\Delta_k^{(t - 1)} \ne \Delta_k^{(t)}$. Since $M_k^{(t - 1)}$ is a model for $\Delta_k^{(t - 1)}$ and $M_k^{(t)}$ is a model for $\Delta_k^{(t)}$, we therefore reset $M_k^{(t)} \gets 0_{m \times n}$ in Line~\ref{line:final-algo-wtk-unified-def}. If $k \in [K] \setminus \idset^{(t)}$, then we are guaranteed $\Delta_k^{(t - 1)} = \Delta_k^{(t)}$ due to the contrapositive of the final statement in Lemma~\ref{lem:dyadic-prox-movement-bound}.
Thus, the term modeled by $M_k$ has not changed, and we update $M_k^{(t)}$ in Line~\ref{line:final-algo-wtk-unified-def} to the corresponding output of the $\MDMP$ call in Line~\ref{line:final-algo-MDMP}.

\RestyleAlgo{ruled}
\DontPrintSemicolon
\SetKwComment{Comment}{/* }{ */}
\begin{algorithm2e}[h!]
\caption{Prox multi-point method for matrix games}
\label{alg:final-algo-outer-loop}
\KwInput{Precision $\epsilon > 0$, max centers per step $K$, $\epsilon$-$\MDMP$ oracle $\OMDMP$
}

$z^{(0)} \gets \argmin_{z \in \zset} r(z)$ ~and~ $t \gets 0$\;

$(w^{(0)}_k, M_{k}^{(0)} )  \gets (z^{(0)}, 0_{m \times n} )$ for all $k \in [K]$ \label{line:final-algo-w-M-init}

\While(\tcp*[f]{Recall $\Range_\dgfsetup \ge \max_{z, z' \in \zset} r(z) - r(z')$}){
    $\sum_{j \in [t]} (\alpha^{(j)})^{-1} < K \Range_\dgfsetup \epsilon^{-1}$ 
}{ \label{line:final-algo-while-loop}

    $t \gets t + 1$ \;

    \tcp{Here, we overload notation and let $(A)_{w_{0}^{(t - 1)}} \defeq 0_{m \times n}$}

    $\pathd^{(t)} \gets \inbraces{\Delta^{(t - 1)}_{k}  ,  M^{(t - 1)}_{k}    }_{k \in [K]}$ ,  where $\Delta^{(t - 1)}_{k} \defeq (A)_{w^{(t - 1)}_{k}}  -  (A)_{w^{(t - 1)}_{k - 1}}$ for $k \in [K]$ \label{line:final-algo-set-P^(t)}

    $\idset^{(t)} \gets \inbraces{k \in [K] : \text{$t$ is divisible by $2^{K - k}$}}$ \label{line:final-algo-idset}

    $\inparen{z^{(t)}, \alpha^{(t)}, \pathd'^{(t)} \defeq \inbraces{\Delta^{(t - 1)}_{k}  ,  M'^{(t)}_{k}    }_{k \in [K]  }} \gets \OMDMP(\inbraces{w^{(t - 1)}_k : k \in \idset^{(t)}}, \pathd^{(t)})$           \label{line:final-algo-MDMP}

    $(w^{(t)}_k, M_{k}^{(t)}) \gets \begin{cases} 
     ( z^{(t)},  0_{m \times n} ) \, ,  & \text{for all } k \in \idset^{(t)} \\
     ( w^{(t - 1)}_k, M'^{(t)}_{k} ) \, ,  & \text{for all } k \in [K] \setminus \idset^{(t)} 
   \end{cases}$ \label{line:final-algo-wtk-unified-def}

}

\tcp{$T$ is used in the analysis to refer to the final iteration count}

\Return{$\zbar \defeq \frac{1}{S} \sum_{j \in [T]} (\alpha^{(j)})^{-1} z^{(j)}$, where $T \defeq t$ and $S \defeq \sum_{j \in [T]} (\alpha^{(j)})^{-1}$} \label{line:final-algo-return}

\end{algorithm2e}

We give our guarantee for Algorithm~\ref{alg:final-algo-outer-loop} in Theorem~\ref{thm:matrix-games-outer-loop-guarantee}. The latter chooses $K$ so as to satisfy the lower bound requirement of Lemma~\ref{lem:dyadic-prox-movement-bound}, while also ensuring $K = \Otilde(1)$, which will be useful when we instantiate Theorem~\ref{thm:matrix-games-outer-loop-guarantee} in Section~\ref{sec:putting-together} in the context of Theorems \ref{thm:final-result-l1-l1-aka-zero-sum} and \ref{thm:final-result-l2-l1-aka-SVM} (specifically, to ultimately obtain $T = \Otilde(\epsilon^{-2/3})$ and so that all of the matrix-approximation paths passed to our implementation of the $\MDMP$ oracle have length $\Otilde(1)$). Besides guaranteeing correctness, Theorem~\ref{thm:matrix-games-outer-loop-guarantee} provides several bounds which will enable our ultimate matvec bounds in Section~\ref{sec:putting-together}. The iteration bound on $T$ and the bound on $\sum_{t \in [T - 1]} (\alpha^{(t)})^\rho$ are immediate from Lemma~\ref{lem:multiprox-iteration-bound} and repeated here for ease of reference in Section~\ref{sec:putting-together}.

The bound on $\sum_{t \in [T]} \insquare{\size(\pathd^{(t)}) - \size(\pathd'^{(t)})}$ in Theorem~\ref{thm:matrix-games-outer-loop-guarantee} is new; namely, it uses the additional machinery of this section as opposed to only that of Section~\ref{subsec:prox-multi-general-monotone-ops}.
It is used in Section~\ref{sec:putting-together} to bound the total number of matvecs made over all model-update steps within calls to the subproblem solver $\SUG$ (Section~\ref{sec:sug-solver}) within our implementation of the $\MDMP$ oracle. We note that the proof of this bound is where we use the ``alternate movement bound'' for the dyadic prox method given in Lemma~\ref{lem:dyadic-prox-movement-bound}. 
More specifically, we use that bound over the first $T - 1$ outer-loop iterations, which is sufficient since the proof telescopes against $\pathd'^{(T)}$ rather than introducing $\pathd^{(T + 1)}$; see \eqref{eq:telescoping-against-path-prime}.
Note also that the terms $\breg{w_{k - 1}^{(t)}}{w_k^{(t)}}$ for $k \in \idset^{(t)}$ in the leftmost summation in \eqref{eq:dyadic-prox-multi-point-method-movement-bound} correspond to models $M_k^{(t)}$ for $k \in \idset^{(t)}$ which are reset to $0_{m \times n}$ in Line~\ref{line:final-algo-wtk-unified-def} of Algorithm~\ref{alg:final-algo-outer-loop}. Using the assumption that $\dgfsetup$ is $\zeta$-compatible with respect to $A$, we are able to bound $\innorm{\Delta_k^{(t)} - M_k^{(t)}}_F^2 = \innorm{\Delta_k^{(t)}}_F^2 \le \zeta \breg{w_{k - 1}^{(t)}}{w_k^{(t)}}$ for $k \in \idset^{(t)}$ in the proof.

\begin{theorem}[Algorithm~\ref{alg:final-algo-outer-loop} guarantee]
    \label{thm:matrix-games-outer-loop-guarantee}
    Suppose the $\MDMP$ oracle given as input to Algorithm~\ref{alg:final-algo-outer-loop} is $(\beta, \gamma, \rho)$-kinetic (Def.~\ref{def:MDMP}) and we choose 
    \begin{align*}
        K \gets \inceil{5 \log_2 \inparen{\Gamma_\dgfsetup (\beta \epsilon^{-1} + 2^{\frac{1}{\rho + 1}} \gamma^{- \frac{1}{\rho+1}} \epsilon^{- \frac{\rho}{\rho + 1}}) + 2} } + 5
    \end{align*}
    Then Algorithm~\ref{alg:final-algo-outer-loop} terminates with 
    \begin{align}
        \label{eq:multiprox-final-iter-bound}
        T \le K \Gamma_\dgfsetup (\beta \epsilon^{-1} + 2^{\frac{1}{\rho + 1}} \gamma^{- \frac{1}{\rho+1}} \epsilon^{- \frac{\rho}{\rho + 1}}) + 2
    \end{align}
    and the output $\zbar$ is a $2 \epsilon$-solution of \eqref{eq:intro-general-matrix-game}. Furthermore, the length of $\pathd^{(t)}$ is $K$ for all $t \in [T]$,
    \begin{align}
        \sum_{t \in [T - 1]} (\alpha^{(t)})^\rho &\le 2 K \Gamma_{\dgfsetup} \gamma^{-1} + (T - 1) \beta^\rho, 
        \text{ and} \label{eq:sum-alpha-bound-takeaway} \\
         \sum_{t \in [T]} \insquare{\size(\pathd^{(t)}) - \size(\pathd'^{(t)})} &\le \innorm{(A)_{z^{(0)}}}_F^2 + 2 \zeta K \Gamma_{\dgfsetup} . \label{eq:final-paths-diff-bound-takeaway}
    \end{align}
\end{theorem}

\begin{proof}
First, we verify that the input $\pathd^{(t)}$ to the $\MDMP$ oracle in Line~\ref{line:final-algo-MDMP} is indeed a matrix-approximation path to some $z \in \inbraces{w^{(t - 1)}_k : k \in \idset^{(t)}}$ (note that the latter is the multiset passed into the $\MDMP$ oracle in Line~\ref{line:final-algo-MDMP}), thereby satisfying the stipulations of Definition~\ref{def:MDMP}. Indeed, note
\begin{align*}
    \sum_{k \in [K]} \Delta_k^{(t - 1)} = (A)_{w_K^{(t - 1)}} - (A)_{w_{0}^{(t - 1)}} = (A)_{w_K^{(t - 1)}}
\end{align*}
since $(A)_{w_{0}^{(t - 1)}} = 0$ by definition. Note $w_K^{(t - 1)} \in \inbraces{w^{(t - 1)}_k : k \in \idset^{(t)}}$ as required since $K \in \idset^{(t)}$. Furthermore, the matrices $M_{k}^{(t - 1)}$ are known explicitly for all $t \in [T]$ since the matrices $M_k'^{(t)}$ are known explicitly for all $t \in [T]$ by Definition~\ref{def:MDMP}. Finally, it is clear that matvecs with any $\Delta_k^{(t - 1)}$ can be computed with $O(1)$ matvecs to $A$ by Definition~\ref{def:product-dgf-setups}.

Note then that Algorithm~\ref{alg:final-algo-outer-loop} is an instantiation of Algorithm~\ref{alg:multiprox-method-monotone-op} for the dgf setup $\dgfsetup$ and monotone operator $\gm f_A$. Then by Lemmas \ref{lem:regret-bounds-the-gap}, \ref{lem:multiprox-correctness}, and \ref{lem:multiprox-iteration-bound}, Algorithm~\ref{alg:final-algo-outer-loop} has the iteration bound \eqref{eq:multiprox-final-iter-bound} and $\zbar$ is a $2 \epsilon$-solution for \eqref{eq:intro-general-matrix-game}. \eqref{eq:sum-alpha-bound-takeaway} is immediate from Lemma~\ref{lem:multiprox-iteration-bound}, and thus we focus on proving \eqref{eq:final-paths-diff-bound-takeaway} for the remainder of the proof.

Toward this goal, for any $t \in [T - 1]$, we have
\begin{align*}
    \size(\pathd^{(t + 1)}) &= \sum_{k \in \idset^{(t)}} \innorm{\Delta_k^{(t)} - M_k^{(t)}}_F^2  + \sum_{k \in [K] \setminus \idset^{(t)}} \innorm{\Delta_k^{(t)} - M_k^{(t)}}_F^2 \\
    &\overeq{(i)} \sum_{k \in \idset^{(t)}} \innorm{\Delta_k^{(t)} }_F^2  + \sum_{k \in [K] \setminus \idset^{(t)}} \innorm{\Delta_k^{(t - 1)} - M_k'^{(t)}}_F^2 \\
    &\le \sum_{k \in \idset^{(t)}} \innorm{\Delta_k^{(t)} }_F^2  + \sum_{k \in [K]} \innorm{\Delta_k^{(t - 1)} - M_k'^{(t)}}_F^2 \\
    &= \sum_{k \in \idset^{(t)}} \innorm{\Delta_k^{(t)} }_F^2 + \size(\pathd'^{(t)}) \, .
\end{align*}
Here, $(i)$ follows because $M_k^{(t)} = 0$ for all $k \in \idset^{(t)}$ by Line~\ref{line:final-algo-wtk-unified-def}. Moreover, we claim $\Delta_k^{(t)} = \Delta_k^{(t - 1)}$ and $M_k^{(t)} = M_k'^{(t)}$ for all $k \in [K] \setminus \idset^{(t)}$. The latter is immediate from Line~\ref{line:final-algo-wtk-unified-def} of Algorithm~\ref{alg:final-algo-outer-loop}. As for the former, the case where $2 \le k \le K$ follows from the final claim of Lemma~\ref{lem:dyadic-prox-movement-bound} (it is the contrapositive). As for the case $\Delta_1^{(t)} = \Delta_1^{(t - 1)}$, recall $(A)_{w_{0}^{(t)}} = (A)_{w_{0}^{(t - 1)}} = 0$ by definition. 

Then using the above and the fact that $\size(\pathd'^{(T)}) \ge 0$, we obtain
\begin{align}
    \label{eq:telescoping-against-path-prime}
    \begin{split}
    \sum_{t \in [T]} (\size(\pathd^{(t)}) - \size(\pathd'^{(t)})) &= \size(\pathd^{(1)}) - \size(\pathd'^{(T)}) + \sum_{t \in [T - 1]} (\size(\pathd^{(t + 1)}) - \size(\pathd'^{(t)})) \\
    &\le \size(\pathd^{(1)}) + \sum_{t \in [T - 1]} \sum_{k \in \idset^{(t)}} \innorm{\Delta_k^{(t)}}_F^2 .
    \end{split}
\end{align}
To conclude the proof of \eqref{eq:final-paths-diff-bound-takeaway}, note that by the choice of $w_k^{(0)}$ and $M_k^{(0)}$ in Line~\ref{line:final-algo-w-M-init} as well as the fact that $(A)_{w_0^{(0)}} = 0$ by definition, we have $\size(\pathd^{(1)}) = \innorm{(A)_{z^{(0)}}}_F^2$. And finally,
\begin{align*}
    \sum_{t \in [T - 1]} \sum_{k \in \idset^{(t)}} \innorm{\Delta_k^{(t)}}_F^2 \le \zeta \sum_{t \in [T - 1]} \sum_{k \in \idset^{(t)}}  \breg{w_{k - 1}^{(t)}}{w_k^{(t)}} \le 2 \zeta K \Gamma_{\dgfsetup}
\end{align*}
by Definition~\ref{def:zeta-compatible-mapping} and Lemma~\ref{lem:dyadic-prox-movement-bound}, noting Algorithm~\ref{alg:final-algo-outer-loop} is indeed an instantiation of the dyadic prox method (Definition~\ref{def:dyadic-prox-multi-point-method}) by the choice of $\idset^{(t)}$ in Line~\ref{line:final-algo-idset}, and $K \ge \log_2 T + 5$ by \eqref{eq:multiprox-final-iter-bound} and a straightforward argument.
\end{proof}

\section{$\MDMP$ implementation for matrix games}
\label{sec:MDMP-implementation}

In this section, we provide and analyze our implementation of a dynamic $\epsilon$-$\MDMP$ oracle (Definition~\ref{def:MDMP}) for matrix-games using the bisection search procedure discussed in Section~\ref{sec:overview-of-approach}. The pseudocode of our implementation is described in Algorithm~\ref{alg:DMP-implementation-matrix-games} and we derive and analyze it in several steps. In particular, we reduce implementing an $\epsilon$-$\MDMP$ to solving a sequence of what we call \emph{constrained prox multi-point problems}. In Section~\ref{sec:wrapper}, we define these problems and related solution concepts. In Section~\ref{sec:binary_search_intro} we introduce a crucial bisection search subroutine (Algorithm~\ref{alg:cautious}) which enables our method, as discussed in Section~\ref{sec:overview-of-approach}. In Section~\ref{sec:mgdamo-sub} we show how to leverage these preliminaries to implement an $\epsilon$-$\MDMP$ oracle $\MDMPImp$ (Algorithm~\ref{alg:DMP-implementation-matrix-games}). Finally, in Section~\ref{sec:analysis-of-implementation} we analyze the implementation.

As in \citep{karmarkar2025solvingzerosumgames}, our algorithm leverages the notion of Hessian stability \citep{carmon2020acceleration, karimireddy2018globallinearconvergencenewtons} in order to implement the \emph{inner loop} discussed in Sections~\ref{sec:intro} and~\ref{sec:overview-of-approach}. In order to formalize this, in the remainder of the paper, we use the following notions of a $c$-\textit{stable ball} and \textit{stability} with respect to a fixed (but arbitrary) mapping. First, the following Definition~\ref{def:stable-region-hess} defines a notion of a stable ball, generalizing Definition 5.1 of \cite{karmarkar2025solvingzerosumgames} to general dgf setups. 

\begin{definition}[$c$-stable ball]\label{def:stable-region-hess} For a dgf setup $\cS = (\zset \subset \R^d, r)$, $z \in \cZ$ and $c > 1$, we define the \emph{$c$-stable ball about $z$} as $\cB_{c, z}^{\cS} \defeq \{z' \in \cZ : c^{-1} \cdot \hess r(z) \preceq \hess r(z') \preceq c \cdot \hess r(z)\}.$
\end{definition}

Next, we define a notion of stability with respect to a mapping. The following definition generalizes the notion of stability introduced in Section 6.2 of \cite{karmarkar2025solvingzerosumgames} to arbitrary dgf setups.

\begin{definition}[Stability]\label{def:best-response-stability} We say that a dgf setup $\cS = (\zset \subset \R^d, r)$ is $(\iota, \rho)$-\emph{stable with respect to a mapping}\footnote{Here, as usual, we allow $\cU$ to be a multiset of $\cZ$ when we write $\cU \subseteq \cZ$.}  $(\alpha > 0, \cU \subseteq \cZ) \mapsto \bestresponse(\alpha, \cU) \in \R^d$ if $\iota: \R_{>1} \to \R_{>1}$ is a strictly increasing, $\rho > 0$, and $z_\alpha^\star \in \cB^{\cS}_{\iota(c), \bestresponse(\alpha, \cU)}$ for any $\alpha > 0$ and $c > 1$ with $z_\alpha^\star \defeq \prox_\cU^\alpha(\nabla_\pm f_A; \cZ)$ and $\breg{\cU}{z_\alpha^\star} \leq c\alpha^\rho$.
\end{definition}

\paragraph{Assumptions.} %
In the remainder of this section, we fix an arbitrary dgf setup, $\cS = (\zset \subset \R^d, r)$, which is $(\iota, \rho)$-stable with respect to a fixed but arbitrary mapping $(\alpha, \cU) \mapsto \bestresponse(\alpha, \cU)$ (Definition~\ref{def:best-response-stability}) for some strictly increasing function $\iota: \R_{>1} \to \R_{>1}$ and $\rho > 0$. In addition, we assume that for any $z \in \cZ$ and $c > 1$, $\cB_{c, z}^\cS$ (Definition~\ref{def:stable-region-hess}) is closed and convex. In Section~\ref{sec:applications} we verify this assumption, specify $\cS, \rho,$ and $\bestresponse$, and bound $\iota$ for our particular applications in Theorems~\ref{thm:final-result-l1-l1-aka-zero-sum} and \ref{thm:final-result-l2-l1-aka-SVM}. 

\subsection{Constrained prox multi-point problems}\label{sec:wrapper}

Here we introduce what we call a constrained prox multi-point problem, corresponding to implementing a constrained variant of the proximal step $\prox_{\cU}^\alpha(\nabla_\pm f_{A}; \cZ)$ (recall Definition~\ref{def:proximal-mappings}). We will implement our $\epsilon$-$\MDMP$ ($\MDMPImp$, Algorithm~\ref{alg:DMP-implementation-matrix-games}) by carefully iteratively solving constrained prox multi-point problems and processing their solutions.

\begin{definition}[Constrained prox multi-point problem]\label{def:subproblem} In the $(\cU, c, \alpha, z, \cS)$-\emph{constrained prox multi-point problem}, we are given a finite, non-empty multiset $\cU \subseteq \cZ$, $c > 1$, $\alpha >0$, and $z \in \cZ$ and define the \emph{solution to the problem} as $z^\star \defeq \prox_{\cU}^\alpha(\nabla_\pm f_A; \cB^{\cS}_{c, z})$ (recall the notation in \Cref{def:proximal-mappings}).
\end{definition}

More precisely we reduce implementing an $\epsilon$-DMDP to computing a type of approximate solution to constrained prox multi-point problems. The following definition introduces this notion of an \emph{$(\epsilon, \delta, \rho)$-approximate solution} to a constrained prox multi-point problem. 

\begin{definition}[Approximate solution]\label{def:approx-solution} 

For $\epsilon, \delta \geq 0$ and $\rho > 0$, letting $z^\star$ be the solution to the $(\cU, c, \alpha, z, \cS)$-constrained prox multi-point problem (Definition~\ref{def:subproblem}), we say that a point $z' \in \cB^{\cS}_{c, z}$ is an $(\epsilon, \delta, \rho)$-\emph{approximate solution} to the problem if, 
\begin{itemize}
    \item $\abs{\breg{\cU}{z'} - \breg{\cU}{z^\star}} < \alpha^\rho/10$, 
    \item $z' \in \cB^{\cS}_{1+\delta, z^\star}$, and
    \item if $\prox_\cU^\alpha(\nabla_\pm f_A; \cZ) \in \cB_{c, z}^{\cS}$ then $\inangle*{\nabla_{\pm}f_A(z') + \alpha \grad \breg{\cU}{z'}, z' -u} \leq \epsilon, \text{ for all } u \in \cZ$. 
\end{itemize}
\end{definition}

Correspondingly, we define an oracle for this approximate solution concept as follows. 

\begin{definition}\label{def:oracle} For any $\rho > 0$, a $\rho$-\emph{approximate solution oracle} $\OAPPROX$
(for $\cS$) takes in a finite non-empty multiset $\cU \subseteq \cZ$, $c > 1$, $\alpha >0$, $z \in \cZ$, a matrix approximation path $\cP = \{\Delta_\ell, M_\ell\}_{\ell \in [L]}$ to $z$ (Definition~\ref{def:matrix-approx-path}), and $\epsilon, \delta \geq 0$ and returns $(z', \cP' = \{\Delta_\ell, M'_\ell\}_{\ell \in [L]})$, where $z'$ is an $(\epsilon, \delta, \rho)$-approximate solution to the $(\cU, c, \alpha, z, \cS)$-constrained prox multi-point problem (Definition~\ref{def:approx-solution}) and $\cP'$ is a matrix approximation path to $z$. 
\end{definition}

In the next sections, we show how, for any $\beta > 0$, we can leverage a $\rho$-approximate solution oracle for $\cS$ to implement an $\epsilon$-$\MDMP$ oracle which is $(\beta, \Theta(1), \rho)$-kinetic (Definition~\ref{def:MDMP}).

\subsection{Cautious Bisection Search}\label{sec:binary_search_intro}

In this section, we introduce a general routine $\cautiousSearch$ (Algorithm~\ref{alg:cautious}) which is our main bisection search procedure to reduce implementing an $\epsilon$-MDMP to implementing a $\rho$-approximate solution oracle. $\cautiousSearch$ is a key subroutine of our $\MDMP$ oracle implementation (the pseudocode of which is in Algorithm~\ref{alg:DMP-implementation-matrix-games}). $\cautiousSearch$ takes a tolerance $\epsilon \geq 0$, a range $[\theta_\ell, \theta_r] \subset \R_{\geq 0}$, and an oracle $\oracle(\cdot)$ that when queried at any $\alpha \in [\theta_\ell,\theta_r]$ outputs a point $z \in \zset \cup \{0\}$ and either $\oracleSuccess$ or $\oracleFailure$. In our application, $\theta_\ell$ corresponds to the $\beta$ parameter of a $\MDMP$ oracle (Definition~\ref{def:MDMP}) and $\theta_r$ corresponds to a value of $\alpha$ for which we are \emph{guaranteed} that $\breg{\cU}{\prox_{\cU}^\alpha(\nabla_\pm f_A; \cZ)} \leq C \alpha^\rho$ for appropriate $C$ (motivated by Definition~\ref{def:best-response-stability}). Correspondingly, the procedure assumes that $\oracle$ outputs $\oracleSuccess$ at $\alpha = \theta_r$ and then, either the oracle outputs $\oracleSuccess$ at $\alpha = \theta_\ell$ or else finds a pair of query values that are $\epsilon$-close where for the larger the $\oracle$ outputs $\oracleSuccess$ and for the smaller the $\oracle$ outputs $\oracleFailure$. 

It is possible to compute the desired $\alpha$ with a logarithmic number of queries using bisection search. However, in our application, querying the oracle for larger $\alpha$ may require more matvecs. Consequently, $\cautiousSearch$ instead queries the oracle at geometrically increasing values starting from $\alpha = \theta_\ell$ searching for the oracle to either output $\oracleFailure$ or $\oracleSuccess$ at $\alpha$. In the case it finds an $\oracleSuccess$ at $\alpha = \theta_\ell$, the procedure returns $\theta_\ell$. Otherwise, the oracle finds $\oracleSuccess$ for some $\alpha \in (\theta_\ell, \theta_r]$. In this case, the procedure performs a bisection search for an $\alpha_u$ such that the oracle outputs $\oracleSuccess$ at $\alpha_u$ and $\oracleFailure$ at an $\alpha_l \geq \alpha_u - \epsilon$. Ultimately, this ensures that the $\cautiousSearch$ both finds the requisite value of $\alpha$ with only a logarithmic number of queries to the oracle and that the oracle does not query the oracle with a value of $\alpha$ much higher than the value of $\alpha$ it ultimately outputs. In our application, this helps ensure that the bisection search only induces polylogarithmic factors of overhead in terms of the number of matvecs made in our applications.

The main guarantees of $\cautiousSearch$ are given below in \Cref{lemma:cautious_search}. A similar procedure was used in the $\lambda$-bisection procedure in Algorithm 1 of \citep{carmon2021thinking} for similar reasons. Our algorithm uses the same general ideas as in that procedure.

 \begin{algorithm2e}[ht]
 	\caption{$\cautiousSearch(\epsilon, \theta_\ell, \theta_r, \oracle(\cdot))$}
 	\label{alg:cautious}
 	\KwInput{Range lower bound $\theta_\ell > 0$, range upper bound $\theta_r > \theta_\ell$, error threshold $\epsilon \in \R_{\geq 0}$, oracle $\oracle : (0, \theta_r] \rightarrow (\cZ, \{\oracleSuccess, \oracleFailure\})$.}
    \BlankLine
    \tcp{If $\oracleSuccess$ at $\alpha = \theta_\ell$, output $\theta_\ell$.}
    \lIf{$\flag = \oracleSuccess$ when $(z, \flag) \gets \oracle(\theta_\ell)$}{\Return{$(z, \theta_\ell)$}}

    \tcp{Repeatedly double $L^{(i)}$ looking for $\oracleSuccess$ at $L^{(i)} \geq \theta_\ell$.}
 	$i = 1$ and $L^{(1)} = \theta_\ell$\;
 	\While{$\flag = \oracleFailure$ when $(z, \flag) \gets \oracle(L^{(i)})$ \label{line:cautious:start}}{
        $L^{(i + 1)} \gets \min\{2L^{(i)}, \theta_r\}$ and then $i \gets i + 1$ \label{line:cautious:end}
    }
    \BlankLine
    \tcp{Bisection search between $\oracleSuccess$ at $L^{(i_*)}$ and $\oracleFailure$ at $L^{(i_*-1)}$ for close $\oracleSuccess$ and $\oracleFailure$ output}
    $\alpha_u^{(1)} \gets L^{(i_*)}$, and $\alpha_\ell^{(1)} \gets L^{(i_*-1)}$ where $i_* = i$\;
    \For{$j \in [j_*]$ where $j_* \defeq \max\{1,\ceil{\log_2((\alpha_u^{(1)} - \alpha_{\ell}^{(1)})/\epsilon)}\}$\label{line:for:start}}{
 			$\alpha_m^{(j)} \gets \frac{\alpha_u^{(j)} + \alpha_{\ell}^{(j)}}{2}$ and 
            $(z', \flag) \gets \oracle(\alpha_m^{(j)})$\; 
 			\lIf{$\flag = \oracleSuccess$}{
 				$(z, \alpha_{u}^{(j +1)}) \gets (z', \alpha_{m}^{(j)})$ and $\alpha_{\ell}^{(j +1)} \gets \alpha_{\ell}^{(j)}$
 				}
 			\lElse{
 					$\alpha_{u}^{(j +1)} \gets \alpha_{u}^{(j)}$ and $\alpha_{\ell}^{(j +1)} \gets \alpha_{m}^{(j)}$
                    \label{line:for:end}
 				}
	}
 	\Return{$(z, \alpha_{u}^{(j_*+1)})$}
 \end{algorithm2e}
  
 \begin{lemma}[\cautiousSearch~guarantee]\label{lemma:cautious_search} Let 
 \[(z_*, \alpha_*) = \cautiousSearch(\epsilon, \theta_\ell, \theta_r, \oracle(\cdot))\] (\Cref{alg:cautious}) where  $\epsilon, \theta_\ell, \theta_r \in \R_{> 0}$ with $0 < \theta_\ell < \theta_r$ and $\oracle : [\theta_\ell, \theta_r] \rightarrow \zset \cup \{0\} \times \{\oracleSuccess, \oracleFailure\}$ is a deterministic oracle satisfying $\oracle(\theta_r) = (\cdot,\oracleSuccess)$.\footnote{In \Cref{lemma:cautious_search} and  \Cref{alg:cautious}, $\zset$ can be any non-empty set.}\footnote{Here and throughout $(a,b) = (\cdot,c)$ denotes that $b = c$.} Then $\alpha_* \in [\theta_l, \theta_r]$, $\oracle(\alpha_*) = (z_*,\oracleSuccess)$ and either \[
 \alpha_* = \theta_\ell
 \text{ or } 
 \oracle(\varsigma) = (\cdot,\oracleFailure)\text{ for some }
 \varsigma \in [\max\{\alpha_* - \epsilon, \theta_\ell\}, \alpha_*)\,.
 \]
 Furthermore, $\cautiousSearch(\cdot)$ makes at most $O(\log(\theta_r/\min\{\epsilon,\theta_\ell\}))$ queries to $\oracle(\alpha)$ and in each query $\alpha \in [\theta_\ell, \min\{2\alpha_*, \theta_r\}]$.
 \end{lemma}

 \begin{proof}
Every iteration of the $\code{while}$ loop, Line~\ref{line:cautious:start} to Line~\ref{line:cautious:end} that does not terminate either increases $L^{(i)}$ by a factor of 2 or has $L^{(i)} = \theta_r$. In the latter case, the while loop terminates at the next iteration, due to the guarantee that $\oracle(\theta_r) = (\cdot, \oracleSuccess)$. Consequently, the $\code{while}$ loop terminates with $O(\log(\theta_r/\theta_\ell))$ queries and ends the loop with $i_* > 1$, $(\cdot, \oracleFailure) = \oracle(L^{(i_*-1)})$ and $(\cdot, \oracleSuccess) = \oracle(L^{(i_*)})$ for $\theta_\ell < L^{(i_*-1)} < L^{(i_*)} = \min\{2L^{(i_*-1)} \theta_r\}$. 

We now check that the algorithm has the desired properties. Note that the $\code{for}$ loop (Line~\ref{line:for:start} to Line~\ref{line:for:end}) simply performs a bisection search between $\alpha_u^{(1)} = L^{(i_*)}$ and $\alpha_\ell^{(1)} = L^{(i_* - 1)}$ maintaining the invariant that $(z,\oracleSuccess) = \oracle(\alpha_u^{(j)})$ and $(\cdot,\oracleFailure) = \oracle(\alpha_\ell^{(j)})$. In addition, note that the returned value of $z_*$ corresponds to the first argument of $\oracle(\alpha_*)$, as desired. 

Also note that $j_*$ is designed so that when the algorithm terminates $\alpha_u^{(j_*+1)} - \alpha_\ell^{(j_*+1)} \leq \epsilon$. Furthermore, all calls to $\oracle(\varsigma)$ made by the algorithm satisfy $\varsigma \in [\theta_\ell, L^{(i_*)}]$ and $\alpha^{(j_*+1)} \in [L^{(i_*-1)}, L^{(i_*)}] \subseteq [L^{(i_*)}/2, L^{(i_*)}]$. Thus, all calls to $\oracle(\varsigma)$ made by the algorithm satisfy $\varsigma \in [\theta_\ell, 2\alpha_*]$ as desired. Finally, the number of queries made is $O(\log(\theta_r/\theta_\ell) + \max\{1,O(\log((L^{(i_*)} - L^{(i_* - 1)}) / \epsilon))\}$. Since $L^{(i_*)} - L^{(i_* - 1)} \leq 2 L^{(i_*-1)}\leq 2\theta_r$, the overall query bound holds. 
 \end{proof}

\subsection{$\MDMP$ implementation}\label{sec:mgdamo-sub}

Here, we introduce our $\epsilon$-$\MDMP$ implementation $\MDMPImp$ (Algorithm~\ref{alg:DMP-implementation-matrix-games}) which is parameterized by $\epsilon > 0$ (other parameters are discussed in the subsequent paragraph). Note that $\MDMPImp$ essentially reduces the implementation of a $\epsilon$-$\MDMP$ to computing approximate solutions of a sequence of constrained prox multi-point problems (recall Definition~\ref{def:subproblem}).

$\MDMPImp$ essentially has two major components. The first component is a subroutine $\Validate$ which is designed to (approximately) identify whether an inputted value of $\alpha$ satisfies the movement guarantee $\breg{\cU}{\prox_\cU^\alpha(\nabla_\pm f_A; \cB^{\cS}_{\iota(5), \zcenter})} = \Theta(\alpha^\rho)$ (where $\zcenter$ is the point defined in Line~\ref{line:zground}) by invoking the $\rho$-approximate solution oracle $\OAPPROX$. $\Validate$ also accesses and updates the global matrix approximation path $\cP$ passed as input to $\MDMPImp$. 

The second component of $\MDMPImp$ is an invocation of $\cautiousSearch$ to compute the desired output for Definition~\ref{def:MDMP}. In particular, note that $\MDMPImp$ is parameterized by $\beta, \rho > 0$. Here, $\beta$ and $\rho$ control the \emph{kineticness} of the resulting $\MDMP$ is (recall Definition~\ref{def:MDMP}). Furthermore, $\MDMPImp$ instantiates $\cautiousSearch$, passing the subroutine $\Validate(\cdot)$ as the underlying $\oracle$ and range lower and upper bounds $\beta$ and $\theta_r$ respectively. 

In the next sections, we discuss and analyze the implementation of $\Validate$ and $\MDMPImp$ in further detail.

\RestyleAlgo{ruled}
\DontPrintSemicolon
\SetKwComment{Comment}{/* }{ */}
\begin{algorithm2e}[h!]
\caption{$\MDMP$ for matrix games implementation $\MDMPImp(\cU, \cP)$}
\label{alg:DMP-implementation-matrix-games}
\KwInput{Finite nonempty multiset $\uset \subseteq \zset$ and a matrix-approximation path $\pathd = \{\Delta_\ell, M_\ell\}_{\ell \in [L]}$ to $z \in \cU$ (Definition~\ref{def:matrix-approx-path})}
\KwParameter{ $\epsilon > 0$, $0 < \beta < \theta_r$, and a $\rho$-approximate solution (AS) oracle $\OAPPROX$ for $\cS$ (Definition~\ref{def:oracle})}
\BlankLine
Define $\epsilon'$ as in \eqref{eq:epsilonprime}\; 
$(z_*, \alpha_*) \gets \cautiousSearch(\epsilon', \beta, \theta_r, \Validate(\cdot)) \label{line:bsearch}$ 
\tcp*{Algorithm~\ref{alg:cautious}}
\Return{$(z_*, \alpha_*, \cP)$}
\BlankLine
\Function{$\Validate(\alpha)$}{
     \tcp{Such a $\delta$ always exists because $\iota$ is strictly increasing over $\R_{>1}$}
     Set $\delta > 0$ so that $(1+\delta)^2 \cdot \iota(3) < (1+\delta) \cdot \iota(4) < \iota(5)$ and either $(1+\delta) \cdot \iota(3) = (\iota(3) + \iota(4))/2$ or else $(1+\delta) \cdot \iota(4) = (\iota(4) + \iota(5))/2$ \label{line:delta2}\; 
     Compute $\tilde{z} \gets \bestresponse(\alpha, \cU) \label{line:bestresponse}$,
     $\zcenter \gets \argmin_{z \in \cB^{\cS}_{\iota(5), \tilde{z}}} \breg{\cU}{z}$ \label{line:zground}\; 
     \lIf{$\breg{\cU}{\zcenter} > 3\alpha^\rho$}{\Return{$(0, \oracleFailure) \label{line:safety-check}$}}
     \tcp{Complete $\cP$ into a matrix-approximation path to $\zcenter$ by adding an additional term with a null model}
    $\Delta_{L+1} \gets \ground{A}{\zcenter} - \ground{A}{z}$ and $M_{L+1} \gets 0_{m \times n}$ \label{line:null-model}\;  
     $(z', \{\Delta_{\ell}, M'_{\ell}\}_{\ell \in [L+1]}) \gets \OAPPROX(\cU, \iota(5)^2, \alpha, \zcenter, \{\Delta_{\ell}, M_{\ell}\}_{\ell \in [L+1]}, \epsilon, \delta)$ \label{line:wrap}\; 
     $\cP \gets \{\Delta_\ell, M'_{\ell}\}_{\ell \in [L]}$ \label{line:inplace} \tcp*{Update $\cP$ in-place}
     \lIf{$z' \notin \cB^{\cS}_{\iota(4) \cdot \iota(5), \zcenter}$}{\Return{$(0, \oracleFailure)$ \label{line:too-big-2}}} 
     \lIf{$\breg{\cU}{z'} > 2.5\alpha^\rho$}{\Return{$(0, \oracleFailure)$ \label{line:too-big-3}}}
     \lElse{\Return{$(z', \oracleSuccess)$ \label{line:too-small}}}
}
\end{algorithm2e}

\subsection{$\MDMP$ analysis}\label{sec:analysis-of-implementation}

Here, we analyze $\MDMPImp$ (Algorithm~\ref{alg:DMP-implementation-matrix-games}). First, we analyze the  $\Validate$ subroutine Algorithm~\ref{alg:DMP-implementation-matrix-games}. $\Validate$ accepts an $\alpha > 0$ (and, implicitly accesses the global variables $\cU, \cP$). As we will show, this subroutine returns $(\cdot, \oracleSuccess)$ only if $\breg{\cU}{\prox_\cU^\alpha(\nabla_\pm f_A; \cZ)} > 2.4\alpha^\rho$ and $(\cdot, \oracleFailure)$ only if $\breg{\cU}{\prox_\cU^\alpha(\nabla_\pm f_A; \cZ)} \leq 2.6\alpha^\rho$ (as we prove in Lemma~\ref{lemma:validate-correctness}.) Note that by Lemma~\ref{lemma:cautious_search}, this ensures that the $\alpha_*$ returned by $\MDMPImp$ satisfies $\breg{\cU}{\prox_\cU^{\alpha_*}(\nabla_\pm f_A; \cZ)} \leq 2.6 \alpha_*^2$, as described in Section~\ref{sec:overview-of-approach}.

In Line~\ref{line:bestresponse}, the algorithm first computes the mapping $\bestresponse(\alpha, \cU)$ and then selects a point $\zcenter \in \cB^{\cS}_{C, \tilde{z}}$ in Line~\ref{line:zground} which minimizes the sum of divergences from points in $\cU$. In the case that this sum of divergences is too large, then the algorithm returns $(0, \oracleFailure)$ in Line~\ref{line:safety-check}. This check is included for two reasons. First, if $\breg{\cU}{\zcenter} > 3\alpha^\rho$, then, as we prove in the following Lemma~\ref{lemma:validate-correctness}, this immediately implies $\breg{\cU}{\prox_\cU^{\alpha}(\nabla_\pm f_A; \cZ)} > 2.4 \alpha^\rho$ (consequently, $\Validate(\alpha)$ must return $(\cdot, \oracleFailure)$. Second, in our eventual application, invoking the approximate solution oracle $\OAPPROX$ on such a $\zcenter$ as in Line~\ref{line:wrap} might require many matvecs. To avoid needlessly exceeding the matvec budget in this case, Line~\ref{line:safety-check} returns ``early'' without ever invoking $\OAPPROX$, ensuring that the algorithm will never call $\OAPPROX$ on a point $\zcenter$ which is too far (in the sense of divergence) from $\cU$. 

On the other hand, if $\breg{\cU}{\zcenter} \leq 3\alpha^\rho$, the algorithm completes $\cP$ (which is a matrix approximation path to $z$) into a matrix-approximation path to the selected $\zcenter$ and computes an $(\epsilon, \delta, \rho)$-approximate solution of a constrained prox multi-point problem (Definition~\ref{def:subproblem}) in Line~\ref{line:wrap} using the oracle $\OAPPROX$. Next, the algorithm updates the matrix approximation path $\cP$ in Line~\ref{line:inplace} in-place. Finally, the algorithm returns depending on a variety of conditions on $z'$ in Lines~\ref{line:too-big-2},~\ref{line:too-big-3} or~\ref{line:too-small}. These return conditions are tailored to enable the following correctness guarantee. 

\begin{lemma}[$\Validate$~correctness guarantee]\label{lemma:validate-correctness}For any $\alpha, \epsilon > 0$, finite nonempty multiset $\cU \subseteq \cZ$, matrix-approximation path $\cP$ to $z \in \cU$, and $\beta > 0$, letting $z^\star_\alpha \defeq \prox_{\cU}^{\alpha}(\gm f_A; \cZ)$ and $(z', \flag) \gets \Validate(\alpha)$, we have that 
\begin{align*}
        \flag = \begin{cases}
        \oracleFailure, & \text{only if } \breg{\cU}{z^\star_\alpha} > 2.4\alpha^\rho, \\
        \oracleSuccess, & \text{only if } \breg{\cU}{z^\star_\alpha} \leq 2.6 \alpha^\rho. 
    \end{cases}
\end{align*}
Furthermore, if $\flag = \oracleSuccess$, then $z^\star_\alpha \in \cB^{\cS}_{\iota(5)^2, \zcenter}$, $z' \in \cZ$, and
\begin{align}\label{eq:dmp-closeness-guarante}
    \inangle*{ \nabla_\pm f_A(z') + \alpha \nabla \breg{\cU}{z'}, z'-u } \leq \epsilon, \text{ for all } u \in \cZ. 
\end{align}
\end{lemma}
\begin{proof} The proof considers all of the return conditions in $\Validate(\alpha)$ and reasoning about the containment of $z^\star_\alpha \in \cB_{\iota(5)^2, \zcenter}^\cS$ in the case that $\flag=\oracleSuccess$. 
\\\\
\emph{Line~\ref{line:safety-check}:} First, suppose that $\Validate(\alpha)$ returns in Line~\ref{line:safety-check}. Then, suppose, for the sake of contradiction, that $\breg{\cU}{z^\star_\alpha} \leq 2.4\alpha^\rho$. Then, by Definition~\ref{def:best-response-stability}, we have that $z^\star_\alpha \in \cB^{\cS}_{\iota(2.4), \tilde{z}} \subseteq \cB^{\cS}_{\iota(5), \tilde{z}}$ because $\iota(5) > \iota(2.4)$ (recall that $\iota$ is strictly increasing). Consequently, we have that $3\alpha^\rho < \breg{\cU}{\zcenter} = \min_{z \in \cB^{\cS}_{\iota(5), \tilde{z}}} \breg{\cU}{z} \leq \breg{\cU}{z^\star_\alpha} \leq 2.4 \alpha^\rho$, which is a contradiction. Thus, the lemma holds. 
\\\\
\emph{Line~\ref{line:too-big-2}:} Suppose that $\Validate(\alpha)$ returns in Line~\ref{line:too-big-2}. For notational convenience, let $z_\ell \defeq \prox_{\cU}(\nabla_\pm f_A; \cB^{\cS}_{C^2, \zcenter})$ (where we follow the convention of \citep{karmarkar2025solvingzerosumgames} and use $\ell$ to denote ``local''). By the properties of $\OAPPROX$ (Definition~\ref{def:oracle} and~\ref{def:approx-solution}), we have that $z' \in \cB^{\cS}_{1+\delta, z_\ell}$. Now, by the fact that $z' \notin \cB^{\cS}_{\iota(4)\cdot \iota(5), \zcenter}$ and Definition~\ref{def:stable-region-hess}, we must have that
\begin{align*}
    \hess r(z') \prec \frac{1}{\iota(4) \iota(5)} \hess r(\zcenter) \text{~~or~~}  \hess r(z') \succ \iota(4)\iota(5) \hess r(\zcenter).
\end{align*}
However, since $z' \in \cB^{\cS}_{1 + \delta, z_\ell}$, we must also have that 
\begin{align*}
    \frac{1}{(1 + \delta)} \hess r(z_\ell) \preceq \hess r(z') \preceq (1 + \delta) \hess r(z_\ell). 
\end{align*}
Consequently, we must have that 
\begin{align*}
    \frac{1}{(1 + \delta)} \hess r(z_\ell) \prec\frac{1}{\iota(4)\iota(5)}\hess r(\zcenter) \text{~~or~~} {\iota(4)\iota(5)} \hess r(\zcenter) \prec (1 + \delta) \hess r(z_\ell). 
\end{align*}
Rearranging the above display, we must have that
\begin{align*}
    \hess r(z_\ell) \prec\frac{(1+\delta)}{\iota(4)\iota(5)}\hess r(\zcenter) \text{~~or~~} \frac{\iota(4)\iota(5)}{(1 + \delta)} \hess r(\zcenter) \prec \hess r(z_\ell). 
\end{align*}
Now, by the choice of $\delta$ in Line~\ref{line:delta2}, we have that $\iota(3) < \frac{\iota(4)}{(1 + \delta)}.$ 
Thus, $z_\ell \notin \cB^{\cS}_{\iota(3)\iota(5), \zcenter}$. We claim that this implies $\breg{\cU}{z^\star_\alpha} > 3\alpha^\rho$. 

Indeed, suppose for the sake of contradiction that $\breg{\cU}{z^\star_\alpha} \leq 3\alpha^\rho$. Then, by Definition~\ref{def:best-response-stability}, we would have that $z^\star_\alpha \in \cB^{\cS}_{\iota(3), \tilde{z}}$. However, note that by construction (Line~\ref{line:zground}) we have that $\zcenter \in \cB^\cS_{\iota(5), \tilde{z}} = \cB^\cS_{\iota(5), \tilde{z}}$ which also implies that $\tilde{z} \in\cB^\cS_{\iota(5), \zcenter}$ by Definition~\ref{def:best-response-stability}. Thus, by Definition~\ref{def:best-response-stability}, we would have that 
\begin{align*}
    z^\star_\alpha \in \cB^{\cS}_{\iota(3), \tilde{z}} \in \cB^{\cS}_{\iota(3)\iota(5), \zcenter} \subseteq \cB^{\cS}_{\iota(5)^2, \zcenter}, 
\end{align*}
where the $\subseteq$ follows because $\iota$ is strictly increasing and hence $\iota(3) < \iota(5)$. But this would imply that $z^\star_\alpha = z_\ell$, which contradicts that $z_\ell \notin  \cB^{\cS}_{\iota(3)\iota(5), \zcenter}$. Consequently, we must have that $\breg{\cU}{z^\star_\alpha} > 3\alpha^\rho$ and hence the lemma holds. 
\\\\
\emph{Containment of $z^\star_\alpha \in \cB_{\iota(5)^2, \zcenter}^{\cS}$:} Next, we prove that if the algorithm reaches Line~\ref{line:too-big-3}, then $z^\star_\alpha \in \cB^{\cS}_{\iota(5)^2, \zcenter}$. Indeed, the algorithm does not return in Line~\ref{line:too-big-2}, then it must be the case that $z' \in \cB^{\cS}_{\iota(4)\cdot \iota(5), \zcenter}$. By Definition~\ref{def:stable-region-hess}, we must have that 
\begin{align*}
    \frac{1}{\iota(4)\iota(5)} \hess r(\zcenter) \preceq \hess r(z') \preceq \iota(4)\iota(5) \hess r(\zcenter). 
\end{align*}
Next, recall that, taking $z_\ell$ as defined above, we have $z' \in \cB^{\cS}_{1 + \delta, z_\ell}$ (due to the the properties of $\OAPPROX$ from Definition~\ref{def:oracle}). Thus,
\begin{align}\label{eq:notbinding}
    \frac{1}{\iota(4)\iota(5) (1+\delta)} \hess r(\zcenter) \preceq \hess r(z_\ell) \preceq \iota(4)\iota(5) \cdot (1+ \delta) \hess r(\zcenter). 
\end{align}
Now, by the choice of $\delta$ in Line~\ref{line:delta2}, we have that $\iota(4)(1+\delta) < (\iota(4) + \iota(5))/2$. Thus, by \eqref{eq:notbinding}, the constraint to $\cB^{\cS}_{C^2, \zcenter}$ in the definition of $z_\ell$ is not binding. Hence, $z^\star_\alpha = z^\ell \in \cB_{C^2, \zcenter}^{\cS}$.
\\\\
\emph{Lines~\ref{line:too-big-3}:} Next, suppose that the algorithm returns in Line~\ref{line:too-big-3}. Then, by the properties of $\OAPPROX$ (Definition~\ref{def:oracle}) and the fact that $z^\star_\alpha = z^\ell$, we have that $\breg{\cU}{z^\star_\alpha} > 2.5 \alpha^\rho - \alpha^\rho/10 = 2.4 \alpha^\rho$ as required, and the lemma holds.
\\\\
\emph{Lines~\ref{line:too-small}:} Suppose that $\Validate$ returns in Line~\ref{line:too-small}. Then, by the properties of $\OAPPROX$ (Definition~\ref{def:oracle}) and the fact that $z^\star_\alpha = z^\ell$, we must have that $\breg{\cU}{z^\star_\alpha} \leq 2.5\alpha^\rho + \alpha^\rho/10 = 2.6\alpha^\rho$ as required, and the lemma holds.
\\\\
The final claim now follows by the properties of $\OAPPROX$ (Definition~\ref{def:oracle}).
\end{proof}

Next, we combine the analysis of $\Validate$ and $\cautiousSearch$ to analyze Algorithm~\ref{alg:DMP-implementation-matrix-games} and prove that it meets the conditions of Definition~\ref{def:MDMP} under appropriate assumptions. The following theorem shows how to instantiate Algorithm~\ref{alg:DMP-implementation-matrix-games} to implement a kinetic $\epsilon$-$\MDMP$ (Definition~\ref{def:MDMP}).

\begin{theorem}\label{thm:dmp-basic} 
For any finite nonempty multiset $\cU \subseteq \cZ$ and $\alpha > 0$, let $z_\alpha^\star \defeq \prox_\cU^\alpha(\nabla_\pm f_A; \cZ)$ and $h(\alpha) \defeq \breg{\cU}{z^\star_\alpha}$. Suppose that $h$ is $M$-Lipchitz over $[\beta, \theta_r]$, and set 
\begin{align}\label{eq:epsilonprime}
     \epsilon' = \min\left\{\left(1 - \paren{\frac{14}{15}}^{{1}/{\rho}}\right)\beta,\frac{\beta^\rho}{15M}\right\}.  
\end{align}
Suppose further that $h(R)\leq 2.4 \cdot \theta_r^\rho$. Then, $\MDMPImp$ (Algorithm~\ref{alg:DMP-implementation-matrix-games}) is a $(\beta, 2, \rho)$-kinetic $\epsilon$-$\MDMP$ (Definition~\ref{def:MDMP}). Moreover, for any finite nonempty multiset $\cU \subseteq \cZ$ and matrix approximation path $\cP$, letting $(z_*, \alpha_*, \cP') \defeq \MDMPImp(\cU, \cP)$, the algorithm makes at most $O(\log(\theta_r/\min\{\epsilon,\beta\}))$ queries to $\Validate(\alpha)$ where in each query, $\alpha \in [\beta, \min\{2\alpha_*, \theta_r\}]$.
\end{theorem}

We prove this theorem using the following natural monotonicity property \Cref{lem:monotonicity}, which is a generalization of Lemma B.9 of \citep{karmarkar2025solvingzerosumgames}. The proof follows very similarly to the proof of Lemma B.9 of \citep{karmarkar2025solvingzerosumgames} (only mildly modified to handle sums over $\cU$ rather than divergence from a single point).

\begin{lemma}[Generalization of Lemma B.9 of \citep{karmarkar2025solvingzerosumgames}]\label{lem:monotonicity}
Let $(\zset, r)$ denote a dgf setup (Definition~\ref{def:dgf-setup}) with $g : \zset \to \R$ a continuous monotone operator, some $\alpha > \beta > 0$, and let $\cU \subseteq \zset$ be a finite nonempty multiset. Then $\walpha \defeq \prox^\alpha_{\cU}(g)$ and $\wbeta \defeq \prox^{\beta}_{\cU}(g)$ satisfy $\breg{\cU}{\walpha} \le \breg{\cU}{\wbeta}$.
\end{lemma}

\begin{proof}
Applying Definition~\ref{def:proximal-mappings}, we have for all $u, u' \in \zset$:
\begin{align*}
    \inangle*{g(\walpha), (\walpha - u)} &\le \alpha \sum_{v \in \cU}\insquare{\breg{v}{u} - \breg{\walpha}{u} - \breg{v}{\walpha}}, \\
    \inangle*{g(\wbeta), (\wbeta - u')} &\le \beta \sum_{v \in \cU} \insquare{\breg{v}{u'} - \breg{\wbeta}{u'} - \breg{v}{\wbeta}}.
\end{align*}
Setting $u \gets \wbeta$, $u' \gets \walpha$, and using the monotonicity of $g$, yields that
\begin{align*}
    0 &\le \alpha \sum_{v \in \cU} \insquare{\breg{v}{\wbeta} - \breg{\walpha}{\wbeta} - \breg{v}{\walpha}} + \beta \sum_{q \in \cU} \insquare{\breg{v}{\walpha} - \breg{\wbeta}{\walpha} - \breg{v}{\wbeta}} \\
    &= \sum_{v \in \cU} (\alpha - \beta) \breg{v}{\wbeta} + (\beta - \alpha) \breg{v}{\walpha} - \alpha \breg{\walpha}{\wbeta} - \beta \breg{\wbeta}{\walpha}.
\end{align*}
Rearranging the above display,
\begin{align*}
    (\alpha - \beta) \sum_{v \in \cU} \breg{v}{\walpha} \le \sum_{v \in \cU} (\alpha - \beta) \breg{v}{\wbeta} - \alpha \breg{\walpha}{\wbeta} - \beta \breg{\wbeta}{\walpha} \le \sum_{v \in \cU}(\alpha - \beta) \breg{q}{\wbeta}.
\end{align*}
The result follows by dividing through by $(\alpha - \beta) > 0$. 
\end{proof}

With this lemma, we now prove the following theorem.

\begin{proof}[Proof of Theorem~\ref{thm:dmp-basic}] By Lemma~\ref{lemma:cautious_search} together with Lemma~\ref{lemma:validate-correctness}, we have that $\Validate(\alpha_*) = (z_*, \oracleSuccess)$ and $z_*$ satisfies \eqref{eq:dmp-closeness-guarante}. Next we prove that either $\alpha =\beta$ or else $\breg{\cU}{z'} > 2 \alpha^\rho.$

Without loss of generality, suppose that $\alpha > \beta$. Then, by Lemma~\ref{lemma:cautious_search}, $\Validate(\varsigma) = \oracleFailure$ for some $\varsigma \in [\max\{\alpha -\epsilon', \beta\}, \alpha)$. Consequently, by Lemma~\ref{lemma:validate-correctness}, $\breg{\cU}{z^\star_\varsigma} > 2.4 \varsigma^\rho.$ Because $h$ is $M$-Lipschitz, by Lemma~\ref{lem:monotonicity}, we have that
\begin{align*}
    h(\varsigma) - h(\alpha) \leq M (\alpha - \varsigma), 
\end{align*}
and consequently, 
\begin{align*}
    h(\alpha) &\geq h(\varsigma) - M(\alpha - \varsigma)  \geq 2.4 (\alpha - \epsilon')^\rho - M\epsilon' \\
    &\geq 2.4 \paren{\paren{\frac{14}{15}}^{1/\rho} \alpha}^\rho - M\epsilon' \\
    &\geq 2.4 \paren{\frac{14}{15} \alpha^\rho} - \frac{\alpha^\rho}{15} \geq 2.1 \alpha^\rho,
\end{align*}
where we used that that 
\begin{align*}
     \epsilon' = \min\left\{\left(1 - \paren{\frac{14}{15}}^{{1}/{\rho}}\right)\beta,\frac{\beta^\rho}{15M}\right\} \leq \min\left\{\left(1 - \paren{\frac{14}{15}}^{{1}/{\rho}}\right)\alpha,\frac{\alpha^\rho}{15M}\right\}. 
\end{align*}
Finally, by the properties of $\OAPPROX$ (Definition~\ref{def:oracle}), we have that 
\begin{align*}
    \breg{\cU}{z'} > 2.1\alpha^\rho - \alpha^\rho/10 = 2 \alpha^\rho. 
\end{align*}
Lastly, the query complexity bounds follow immediately from Lemma~\ref{lemma:cautious_search}. 
\end{proof}

\section{Smooth-until-proven-guilty solver}\label{sec:sug-solver}

Recall that the $\MDMPImp$ algorithm presented in Section~\ref{sec:MDMP-implementation} assumed access to a $\rho$-approximate solution oracle (Definition~\ref{def:oracle}) $\OAPPROX$. In this section, we adapt the framework from \citep{karmarkar2025solvingzerosumgames} towards implementing a $\rho$-approximate solution oracle for constrained prox multi-point problems (Definition~\ref{def:subproblem}) for our applications to $\ell_1$-$\ell_1$ and $\ell_2$-$\ell_2$ matrix games. 

In fact, in this section, we show a somewhat more general result. In all of our applications (due to geometric properties of our specific setups for $\ell_1$-$\ell_1$ and $\ell_2$-$\ell_1$ matrix games), in order to compute an approximate solution (Definition~\ref{def:approx-solution}) it suffices to compute a (possibly) weaker solution concept, which we term a \emph{divergence-bounded solution.}

\begin{definition}[Divergence-bounded solution]\label{def:divergence-bounded-solution}\label{def:divergence-bounded}
For $\epsilon \geq 0$, letting $z^\star$ be the solution to the $(\cU, c, \alpha, z, \cS)$-constrained prox multi-point problem (Definition~\ref{def:subproblem}), we say that a point $z' \in \cB^{\cS}_{c, z}$ is an \emph{$\epsilon$-divergence-bounded solution} to the problem if $\breg{z^\star}{z'} \leq \epsilon$. 
\end{definition}

Correspondingly, we define a divergence-bounded solution oracle. 

\begin{definition}\label{def:divergence-bounded-oracle} A \emph{divergence-bounded solution oracle} $\ODB$ (for a dgf setup $\cS = (\cZ \subset \R^d, r)$) takes in a finite non-empty multiset $\cU \subseteq \cZ$, $c > 1$, $\alpha >0$, $z \in \cZ$, a matrix approximation path $\cP = \{\Delta_\ell, M_\ell\}_{\ell \in [L]}$ to $z$ (Definition~\ref{def:matrix-approx-path}), and $\epsilon \geq 0$ and returns $(z', \cP' = \{\Delta_\ell, M'_\ell\}_{\ell \in [L]})$, where $z'$ is an $\epsilon$-divergence-bounded solution to the $(\cU, c, \alpha, z, \cS)$-constrained prox multi-point problem (Definition~\ref{def:divergence-bounded}) and $\cP'$ is a matrix approximation path to $z$.
\end{definition}

Under the following definition of \emph{robustness}, access to a divergence-bounded solution oracle is sufficient to implement an approximate-solution oracle (Definition~\ref{def:oracle}) and consequently is sufficient to instantiate an $\MDMP$ (Definition~\ref{def:MDMP}) as described in Section~\ref{sec:MDMP-implementation} (Algorithm~\ref{alg:DMP-implementation-matrix-games}).

\begin{definition}[Robustness]\label{def:robustness} For $\epsilon, \delta, \kappa \geq 0$ and $\rho > 0$, a dgf setup $\mathcal{S} = (\mathcal{Z} \subset \R^d, r)$ is $(\epsilon, \delta, \rho, \kappa)$-\emph{robust} if for every constrained prox multi-point problem $(\cU, c, \alpha, z, \cS)$ (Definition~\ref{def:subproblem}), every $\kappa$-divergence-bounded solution (Definition~\ref{def:divergence-bounded}) is also an $(\epsilon, \delta, \rho)$-approximate solution (Definition~\ref{def:approx-solution}). 
\end{definition}

In particular, the following condition is sufficient to ensure robustness. 
\begin{lemma}[Sufficient conditions for robustness]
    \label{lem:sufficient-cond-for-robustness}
    Suppose that for every $(\cU, c, \alpha, z, \cS = (\cZ \subset \R^d, r))$-constrained prox multi-point problem (Definition~\ref{def:subproblem}) letting $z^\star$ denote its solution, any $z' \in \cB_{c, z}^{\cS}$ with $\breg{z^\star}{z'} \leq \kappa$ satisfies 
\begin{itemize}
    \item $\abs{\breg{\cU}{z'} - \breg{\cU}{z^\star}} < \alpha^\rho/10$, 
    \item $z' \in \cB^{\cS}_{1+\delta, z^\star}$, and
    \item if $\prox_\cU^\alpha(\nabla_\pm f_A; \cZ) \in \cB_{c, z}^{\cS}$ then $\inangle*{\nabla_{\pm}f_A(z') + \alpha \grad \breg{\cU}{z'}, z' -u} \leq \epsilon, \text{ for all } u \in \cZ$. 
\end{itemize}
Then, $\cS$ is $(\epsilon, \delta, \rho, \kappa)$-robust.
\end{lemma}
\begin{proof} The proof is immediate from Definitions~\ref{def:approx-solution} and~\ref{def:divergence-bounded}. 
\end{proof}

In Section~\ref{sec:applications}, we show that for $\kappa$ scaling polynomially in $1/\epsilon, 1/\delta$ and the problem parameters (namely, $m, n$), the preconditions of Lemma~\ref{lem:sufficient-cond-for-robustness} and consequently the robustness condition (Definition~\ref{def:robustness}) is met in our applications. Thus, for our applications, the methods in this section suffice to implement an approximate solution oracle as required in Section~\ref{sec:MDMP-implementation}. Consequently, in this section, we discuss how to implement a divergence-bounded solution oracle (Definition~\ref{def:divergence-bounded-oracle}).

\paragraph{Assumptions.} %
In the remainder of this section, we fix arbitrary dgf setups $\dgfsetup\x = (\xset \subset \R^n, \rx)$ and $\dgfsetup\y = (\yset \subset \R^m, \ry)$ with $\dgfsetup = (\zset \subset \R^d, r) \defeq \prodsetup(\dgfsetup\x, \dgfsetup\y)$ (recall Definition~\ref{def:product-dgf-setups}). Moreover, we assume that $\cZ$ is $\pi$-locally bounded in the sense of the following definition.

\begin{definition}[$\pi$-locally bounded]\label{def:consistency} We say the dgf setup $(\zset \subset \R^d, r)$ is \emph{$\pi$-locally bounded} for $\pi : \R_{>1} \to \R_{>0}$ if for any $z \in \cZ$, $z', z'' \in \cB_{c, z}^{\cS}$, and $c > 1$ we have that $\breg{z'}{z''} \geq \pi(c) \normInline{z' - z''}_{z}$.
\end{definition}

Additionally, we let $\Gamma_\cS$ denote an upper bound on the range of $r$ so that $\sup_{z \in \cZ} \breg{z'}{z} \leq \Gamma_\cS$ for $z' \defeq \argmin_{z\in\cZ} r(z)$. 

In the remainder of this Section~\ref{sec:SUG-solver-prelims}, we discuss a simple linear algebraic sub-routine, which we term a $\judge$ as in \citep{karmarkar2025solvingzerosumgames}. This $\judge$ subroutine shows how we update the matrix approximation path $\cP$ and enables complexity analysis as a function of $\size(\cP) -\size(\cP')$. In Section~\ref{sec:step} we generalize the smooth until proven guilty mirror prox steps from \citep{karmarkar2025solvingzerosumgames}, which enables our implementation of a divergence-bounded solution oracle (Definition~\ref{def:divergence-bounded-oracle}). Finally, in Section~\ref{sec:solver} we describe our approximate solution oracle. This section is largely motivated by Section 5 of \citep{karmarkar2025solvingzerosumgames} and leverages similar techniques to their prior work; however, to handle our general setups and use of matrix approximation paths, we require several slight modifications.

\subsection{The smooth-guilty judge}\label{sec:SUG-solver-prelims} 

Here we describe our notion of a $\judge$ subroutine, which is inspired by the $\judge$ subroutine in \citep{karmarkar2025solvingzerosumgames} but is appropriately adapted to our setting of working with matrix approximation paths. The input to $\judge$ is a center point $\zcenter \in \cZ$, a matrix-approximation path $\cP$ to $\zcenter$, a parameter $\tau > 0$ (which we call a \emph{smoothness threshold} as in \citep{karmarkar2025solvingzerosumgames}), and two vectors $z, z' \in \cZ$. The $\judge$ subroutine ``judges'' whether the vector $z$ or $z'$ reveals a $\tau$-large singular direction along the matrix-approximation path $\cP$. This is formalized in the following pseudocode (Algorithm~\ref{alg:judge}), where we use $\normalize(z): z \mapsto z/\normInline{z}_2$ to be the mapping which takes any vector $z \in \R^d$ to a unit vector in the direction of $z$.   

\RestyleAlgo{ruled}\label{alg:judge}
\SetKwComment{Comment}{/* }{ */}
\begin{algorithm2e}[ht]
\caption{$\judge(\cP, \tau, z, z')$}
\KwInput{Matrix-approximation path $\cP = \{\Delta_\ell, M_\ell\}_{\ell \in [L]}$, smoothness threshold $\tau > 0$, $z, z' \in \cZ$.}
\For{$\bar{z} \in \{z, z'\}$}{
    \tcp{If we find a $\tau$-large singular direction, return a $\guilty$ verdict and update $\cP$}
    \If{
    $\inangle*{\bar{z}\y, \sum_{t \in [L]}(\Delta_\ell - M_\ell) \bar{z}\x} > \tau \normInline{\bar{z}\y}_2\normInline{\bar{z}\x}_2$ \label{line:guilty}
    }{
    $v \gets \normalize(\bar{z}\y)$ and  $u \gets \normalize(\bar{z}\x)$ \label{line:normalize}\; 
    \lFor(\tcp*[f]{Update the $\ell$-th model in $\cP$}){$\ell \in [L]$}{   
       $M_\ell \gets M_\ell + \inangle*{v, (\Delta_\ell - M_\ell) u} \cdot vu^\top$\label{line:update}}
    \Return{$(\guilty, \cP)$}
    }
}
\Return{$(\smooth, \cP)$}
\end{algorithm2e}

To analyze, $\judge$ (Algorithm~\ref{alg:judge}), we use the following property of the Frobenius norm.

\begin{lemma}[Lemma C.1 of 
\citep{karmarkar2025solvingzerosumgames}%
]\label{lemma:frobenius}  $\normInline{A-B}_F^2 \leq \normInline{A}_F^2 - \inangle*{v, A u}^2$ for any $A \in \R^{m \times n}$, unit vectors $u \in \R^n$ and $v \in \R^m$ (i.e., $\normInline{u}_2 = \normInline{v}_2 = 1$), and $B =  \inangle*{v, A u} \cdot vu^\top$.
\end{lemma}

With \Cref{lemma:frobenius}, we can analyze the $\judge$ subroutine (Algorithm~\ref{alg:judge}). 

\begin{lemma}\label{lemma:judge-guarantee} Let $\cP = \{\Delta_\ell, M_\ell\}_{\ell \in [L]}$ be a matrix-approximation path, $\tau >0$ be a smoothness threshold, and $z, z' \in \cZ$. Then $(\verdict, \cP') \gets \judge(\cP, \tau, z, z')$ can be implemented with $O(L)$ matvecs to $A$ and satisfies the following: \begin{itemize}
    \item If $\verdict = \smooth$ then $\cP' = \cP$, 
\begin{align*}
    \inangle*{z\y, \sum_{\ell \in [L]} (\Delta_\ell - M_\ell) z\x} \leq \tau \normInline{z\x}_2\normInline{z\y}_2, ~~\text{ and }~~ \inangle*{z'\y, \sum_{\ell \in [L]} (\Delta_\ell - M_\ell) z'\x} \leq \tau \normInline{z'\x}_2 \normInline{z'\y}_2. 
\end{align*} 
\item If $\verdict = \guilty$ then $\cP' = \{(\Delta_\ell, M'_\ell)\}_{\ell \in [L]}$ is a matrix-approximation path to $\zcenter$ such that $\size(\cP') \leq \size(\cP) - \tau^2/L$. 
\end{itemize}
\end{lemma}

\begin{proof} From the pseudocode, it is easy to verify that if $\verdict = \smooth$, then the first bullet holds. Thus, it remains to prove the second bullet. 

First, observe that $\cP'$ is a matrix-approximation path to $\zcenter$ because each $M'_\ell$ is known explicitly after each update in Line~\ref{line:update}. Next, to analyze $\size(\cP')$, note that $\verdict = \guilty$ ensures that the if statement in Line~\ref{line:guilty} executes for a $\bar{z} \in \{z, z'\}$ and for this value of $\bar{z}$,
\begin{align}
\label{eq:iftrue}
    \inangle*{\bar{z}\y, \sum_{\ell \in [L]} (\Delta_\ell - M_\ell) \bar{z}\x} > \tau \normInline{\bar{z}\y}_2 \normInline{\bar{z}\x}_2\,.
\end{align}
Consequently, $\bar{z}\x \neq 0_n$ and $\bar{z}\y \neq 0_m$. Rescaling \eqref{eq:iftrue} and using the definition of $u$ and $v$ then yields
\begin{align*}
    \tau < 
    \sum_{\ell \in [L]}\inangle*{\frac{\bar{z}\y}{\normInline{\bar{z}\y}_2}, (\Delta_\ell - M_\ell) \frac{\bar{z}\x}{\normInline{\bar{z}\x}_2}} =  
    \sum_{\ell \in [L]}\inangle*{v, (\Delta_\ell - M_\ell) u}. 
\end{align*}
Applying the Cauchy-Schwarz inequality then yields that 
\begin{align}\label{eq:cauchy-schwarz}
    \frac{\tau^2}{L} <
    \left(\frac{1}{L} \sum_{\ell \in [L]}\inangle*{v, (\Delta_\ell - M_\ell) u}\right)^2
    \leq
    \sum_{\ell \in [L]}\inangle*{v, (\Delta_\ell - M_\ell) u}^2
\end{align}
Now, using Lemma~\ref{lemma:frobenius} to reason about the updates in Line~\ref{line:update}, we can conclude that for each $\ell \in [L]$, 
\begin{align*}
    \normInline{\Delta_\ell - M'_\ell}_F^2 \leq \normInline{\Delta_\ell - M_\ell}_F^2 - \inangle*{v, (\Delta_\ell - M_\ell) u}^2. 
\end{align*}
By \eqref{eq:cauchy-schwarz}, it follows that 
\begin{align*}
    \size(\cP') &= \sum_{\ell \in [L]} \normInline{\Delta_\ell - M'_\ell}_F^2 \leq \sum_{\ell \in[L]} \paren{\normInline{\Delta_\ell - M_\ell}_F^2 - \inangle*{v, (\Delta_\ell - M_\ell) u}^2} 
    \leq \size(\cP) - \frac{\tau^2}{L}. 
\end{align*}
Finally, to justify the query complexity, note that the if statement in Line~\ref{line:guilty} requires $O(L)$ matvecs to $A$ while each iteration of Line~\ref{line:update} requires $O(1)$ matvecs to $A$. 
\end{proof}

In some cases, there are alternative implementations of the $\judge$ routine which satisfy the guarantees of Lemma~\ref{lemma:judge-guarantee} (see Section 6.5.2 and Appendix C of \citep{karmarkar2025solvingzerosumgames}); however, we focus on this implementation, as it is particularly simple.

\subsection{Smooth until proven guilty composite mirror prox}\label{sec:step}

Here we adapt the smooth until proven guilty composite mirror prox algorithm of \citep{karmarkar2025solvingzerosumgames} to our framework with path approximations. This adaptation (Algorithm~\ref{alg:mirror-prox-step}) enables us to implement a divergence-bounded solution oracle (Definition~\ref{def:oracle}). The reader might also find it helpful to refer to Definition~\ref{def:product-dgf-setups} for  a reminder of the notation used in Lines~\ref{line:unground} and~\ref{line:unground2}. The following \Cref{lemma:step-guarantee} 
provides the main guarantee of $\Step$. 

\RestyleAlgo{ruled}\label{alg:mirror-prox-step}
\SetKwComment{Comment}{/* }{ */}
\begin{algorithm2e}[ht]
\caption{$\Step(\cU, \cP, c, \alpha, \zcenter, z)$}
\KwInput{$\cU, c, \alpha, \zcenter$ as in Definition~\ref{def:subproblem}, a matrix approximation path $\cP = \{\Delta_\ell, M_\ell\}_{\ell \in [L]}$ to $\zcenter$ and a $z \in \cB_{c, \zcenter}^{\cS}$.}
\KwParameter{ A smoothness threshold $\tau > 0$}
$B \gets \paren{\sum_{\ell \in [L]} (\Delta_\ell - M_\ell)}_{\zcenter, *}$ \label{line:unground} \tcp*{$B$ is the unknown portion of $A$ (the subtraction is done implicitly)}
$C \gets \paren{\sum_{\ell \in [L]} M_\ell}_{\zcenter, *}$ \label{line:unground2}\tcp*{$C$ is the explicitly known portion of $A$} 
$\psi \gets \alpha \nabla \breg{\cU}{\cdot} + \nabla_\pm f_C$ \; 
$w \gets \prox_{z}^\tau(\nabla_\pm f_{B}(z) + \nabla_\pm \psi; \cB_{c, \zcenter}^{\cS})$ \label{line:first-step-stronglymonotone}\;
$z' \gets \prox_{z, w}^{\tau, \alpha}\paren{ (\nabla_\pm f_{B}  + \psi)(w); \cB_{c, \zcenter}^{\cS}}$ \label{line:second-step-stronglymonotone}\;
$z_1 \gets (w\x - z'\x, w\y - z\y)$, $z_2 \gets (z\x -w\x, w\y - z'\y)$ \; 
$z_{(1)} \gets (z_1)_{\zcenter}$, $z_{(2)} \gets (z_2)_{\zcenter}$\; 
$(\verdict, \cP') \gets \judge(\cP, 2\pi(c) \cdot \tau, z_{(1)}, z_{(2)})$ \tcp*{$\pi(c)$ as defined in Definition~\ref{def:consistency}}
\Return{$(z', \verdict, \cP')$}
\end{algorithm2e}

\begin{lemma}\label{lemma:step-guarantee} Let $z^\star \defeq \prox_\cU^\alpha(\nabla_\pm f_A; \cB_{c, \zcenter}^{\cS})$, $\tau > 0$, $c > 1$, $z \in \cB_{c, \zcenter}^{\cS}$, and $(z', \verdict, \cP') \gets \Step(\cU, \cP, c, \alpha, \zcenter, z)$. Then, $\cP'$ is a matrix-approximation path to $\zcenter$, $\size(\cP') \leq \size(\cP)$, and either 
\begin{itemize}
    \item $\verdict = \smooth$ and $\breg{z^\star}{z'} \leq \paren{1 + \frac{\alpha}{\tau}}^{-1} \breg{z^\star}{z}$, or else 
    \item $\verdict = \guilty$ and $\size(\cP') \leq \size(\cP) - (2\pi(c) \cdot \tau)^2/L$. 
\end{itemize}
The algorithm can be implemented with $O(L)$ matvecs to $A$.  
\end{lemma}

Our proof of Lemma~\ref{lemma:step-guarantee} uses the following technical lemma. 

\begin{lemma}\label{lemma:standalone-guarantee} If $c, \zcenter, z, w, z'$, $B$, $z_1, z_2, z_{(1)}$, $z_{(2)}$ are as in the pseudocode of Algorithm~\ref{alg:mirror-prox-step} and 
\begin{align*}
    \inangle*{w\y -z\y, B (w\x - z'\x)} + \inangle*{w\y - z'\y, B (z\x - w\y)} > \tau (\breg{w}{z'} + \breg{z}{w})\,,
\end{align*}
then, $\verdict = \guilty$. 
\end{lemma}
\begin{proof} By Fact~\ref{lemma:ungrounding}, we have
\begin{align*}
    &\inangle*{w\y -z\y, B (w\x - z'\x)} + \inangle*{w\y - z'\y, B (z\x - w\y)} \\
    =& \inangle*{{z_{(1)}}\y, \ground{B}{\zcenter} {{z_{(1)}}}\x} + \inangle*{{z_{(2)}}\y, \ground{B}{\zcenter} {{z_{(2)}}}\x} \\
    =& \inangle*{{z_{(1)}}\y, \sum_{\ell \in [L]} (\Delta_\ell - M_\ell) {{z_{(1)}}}\x} + \inangle*{{z_{(2)}}\y, \sum_{\ell \in [L]} (\Delta_\ell - M_\ell) {{z_{(2)}}}\x}. 
\end{align*}
Now, by Definition~\ref{def:best-response-stability} and Lemma~\ref{lemma:ungrounding}, we have 
\begin{align*}
   (\breg{w}{z'} + \breg{z}{w}) &\geq \pi(c) \cdot \paren{ \normInline{w-z'}_{\zcenter}^2 + \normInline{w-z}_{\zcenter}^2 } \\
    &= \pi(c) \cdot \paren{ \norm{\ground{w-z'}{\zcenter}}_2^2 + \norm{\ground{w-z}{\zcenter}}_2^2 }. 
\end{align*}
Thus, by splitting into the components in $\cX$ and $\cY$, we have 
\begin{align*}
    &(\breg{w}{z'} + \breg{z}{w}) \\
    \geq& \pi(c) \cdot \paren{ \norm{\ground{w-z'}{\zcenter}}_2^2 + \norm{\ground{w-z}{\zcenter}}_2^2 } \\
   =& \pi(c) \cdot \paren{ \norm{ {\ground{w-z'}{\zcenter}}\x }_2^2 + \norm{ {\ground{w-z}{\zcenter}}\x}_2^2 + \norm{ {\ground{w-z'}{\zcenter}}\y }_2^2 + \norm{ {\ground{w-z}{\zcenter}}\y}_2^2} \\
   \geq& 2\pi(c) \cdot \paren{ \norm{ {\ground{w-z'}{\zcenter}}\x }_2 \norm{ {\ground{w-z}{\zcenter}}\y}_2 + \norm{ {\ground{w-z}{\zcenter}}\x}_2 \norm{ {\ground{w-z'}{\zcenter}}\y }_2 } \\
   \geq& 2\pi(c) \cdot \paren{\norm{{z_{(1)}}\x}_2 \norm{{z_{(1)}}\y}_2 + \norm{{z_{(2)}}\x}_2 \norm{{z_{(2)}}\y}_2}, 
\end{align*}
where the second-to-last step used that for any $a, b \geq 0$ we have $a^2 + b^2 \geq 2ab$. Thus, we must have that either 
\begin{align*}
    \inangle*{{z_{(1)}}\y, \sum_{\ell \in [L]} (\Delta_\ell - M_\ell) {{z_{(1)}}}\x}  > 2\pi(c) \cdot \tau \cdot \norm{{z_{(1)}}\x}_2 \norm{{z_{(1)}}\y}_2, 
\end{align*}
or else 
\begin{align*}
    \inangle*{{z_{(2)}}\y, \sum_{\ell \in [L]} (\Delta_\ell - M_\ell) {{z_{(2)}}}\x} > 2\pi(c) \cdot \tau \cdot \norm{{z_{(2)}}\x}_2 \norm{{z_{(2)}}\y}_2. 
\end{align*}
Consequently, by Lemma~\ref{lemma:judge-guarantee}, we must have that $\verdict = \guilty$. 
\end{proof}

We now prove Lemma~\ref{lemma:step-guarantee}. The proof is very similar to the proof of Lemma 5.6 in \citep{karmarkar2025solvingzerosumgames} (and perhaps other well-known proofs of strongly monotone mirror prox). The main difference relative to the proof of Lemma 5.6 of \citep{karmarkar2025solvingzerosumgames} is that our version needs to handle prox steps with regularization to each $u \in \cU$, whereas the version in \citep{karmarkar2025solvingzerosumgames} considered only regularization with respect to a single point. 

\begin{proof}[Proof of Lemma~\ref{lemma:step-guarantee}, adapted from Proof of Lemma 5.6 of \citep{karmarkar2025solvingzerosumgames}] If $\verdict = \guilty$, then the second bullet holds due to Lemma~\ref{lemma:judge-guarantee}. Thus, suppose that $\verdict = \smooth$ and observe that by Lemma~\ref{lemma:standalone-guarantee} we have that $A = B + C$. Consequently, by Lemma~\ref{lemma:standalone-guarantee}, we have that 
\begin{align}
\begin{split}\label{eq:lipschitz-condition-holds-case}
    \inangle*{\nabla_\pm \bilinear{B}(w) - \nabla_\pm \bilinear{B}(z), w - z'} &= \inangle*{w\y -z\y, B w\x - z'\x} + \inangle*{w\y - z'\y, B z\x - w\y} \\
    &\leq \tau (\breg{w}{z'} + \breg{z}{w}). 
\end{split}
\end{align}

Next, we apply the optimality conditions from Definition~\ref{def:proximal-mappings} to each of the composite proximal steps (Lines~\ref{line:first-step-stronglymonotone} and~\ref{line:second-step-stronglymonotone}). We have that for all $u, u' \in \cB_{c, \zcenter}^{\cS}$,  
\begin{align*}
\begin{split}
    \inangle*{\nabla_\pm f_B(z), w-u'} + \inangle*{\psi(w), w-u'} &\leq \tau \paren{\breg{z}{u'} - \breg{w}{u'} - \breg{z}{w}}\,\\ 
    \inangle*{\nabla_\pm f_B(w), z'-u} + \inangle*{\psi(w), z' - u} &\leq \tau \paren{\breg{z}{u} - \breg{z'}{u} - \breg{z}{z'}} + \alpha \paren{ \breg{w}{u} + \breg{z'}{u}}, 
\end{split}
\end{align*}
(the second line used non-negativity of the Bregman divergence). Setting $u' = z'$ in the above display and summing both equations, we have, for all $u \in \cB_{c, \zcenter}^{\cS}$,
\begin{align*}
    &\inangle*{\nabla_\pm f_B(z), w-z'} + \inangle*{\nabla_\pm f_B(w), z'-u} +  \inangle*{\psi(w), w-u}  \\
    &\leq \tau \paren{\breg{z}{u} - \breg{z'}{u} - ( V_{w}^r(z') + \breg{z}{w} )} + \alpha \paren{ \breg{w}{u} + \breg{z'}{u}}. 
\end{align*}
Dividing through the above display and \eqref{eq:lipschitz-condition-holds-case} by $\tau$, we find that for all $u \in \cB_{c, \zcenter}^{\cS}$,
\begin{align*}
    &\frac{1}{\tau} \Brac{ \inangle*{\nabla_\pm f_B(z), w-z'} + \inangle*{\nabla_\pm f_B(w), z'-u} + \inangle*{\psi(w), w-u}} \\
    \leq& \breg{z}{u} - \breg{z'}{u} - ( V_{w}^r{z'} + \breg{z}{w} ) + \frac{\alpha}{\tau} \paren{ \breg{w}{u} + \breg{z'}{u}}. 
    \\
    \leq& \breg{z}{u} - \breg{z'}{u} -\frac{1}{\tau} \inangle*{\nabla_\pm f_B(w) - \nabla_\pm f_B(z), w - z'} + \frac{\alpha}{\tau} \paren{ \breg{w}{u} - \breg{z'}{u}}. 
\end{align*}
Rearranging terms, we have that for all $u \in \cB_{c, \zcenter}^{\cS}$, 
\begin{align*}
       \frac{1}{\tau} \inangle*{\nabla_\pm \bilinear{B}(w) +\psi(w), w-u} &= \frac{1}{\tau} \Brac{\inangle*{\nabla_\pm f_B(z), w-z'} + \inangle*{\nabla_\pm f_B(w), z'-u} + \inangle*{\psi(w), w-u}}  \\ 
       &~~+  \frac{1}{\tau} \inangle*{\nabla_\pm f_B(w) - \nabla_\pm f_B(z), w - z'} + \frac{\alpha}{\tau} \paren{ \breg{w}{u} - \breg{z'}{u}} \\
       &\leq \breg{z}{u} - \breg{z'}{u} + \frac{\alpha}{\tau} \paren{ \breg{w}{u} - \breg{z'}{u}}. 
\end{align*}
Thus, for $u = z^\star$, we have 
\begin{align*}
   \frac{1}{\tau} {\inangle*{\nabla_\pm \bilinear{B}(w)+\psi(w), w-z^\star}} \leq  \breg{z}{z^\star} - \breg{z'}{z^\star} + \frac{\alpha}{\tau} \paren{ \breg{w}{z^\star} - \breg{z'}{z^\star}}. 
\end{align*}
Consequently, subtracting $\frac{\alpha}{\tau} \breg{w}{z^\star}$ from both sides,
\begin{align}\label{eq:right-side-boundedness}
    \frac{1}{\tau} \inangle*{\nabla_\pm \bilinear{B}(w)+\psi(w), w-z^\star} - \frac{\alpha}{\tau} \breg{w}{z^\star} \leq \breg{z}{z^\star} - \paren{1 + \frac{\alpha}{\tau}} \breg{z'}{z^\star}. 
\end{align}
To complete the proof, it suffices to lower bound the left hand side of \eqref{eq:right-side-boundedness} by 0. To this end, note that, by the definition of $z^\star$, we have that $\frac{1}{\tau} \inangle*{\nabla_\pm \bilinear{B}(z^\star)+\psi(z^\star), z^\star-w} \leq 0$. Consequently,
\begin{align*}
     &\frac{1}{\tau} \inangle*{\nabla_\pm \bilinear{B}(w)+\psi(w),  w-z^\star} - \frac{\alpha}{\tau} \breg{w}{z^\star} \\
     \geq& \frac{1}{\tau} \inangle*{(\nabla_\pm \bilinear{B}(w)+\psi(w)) - (\nabla_\pm \bilinear{B}(z^\star)+\psi(z^\star)), w-z^\star} - \frac{\alpha}{\tau} \breg{w}{z^\star}. 
\end{align*}
Now, by $\alpha$-strong monotonicity of the operator $\nabla_\pm f_B + \nabla \psi$, 
\begin{align*}
        \frac{1}{\tau} \inangle*{(\nabla_\pm \bilinear{B}(w)+\psi(w)) - (\nabla_\pm \bilinear{B}(z^\star)+\psi(z^\star)), w-z^\star} \geq \frac{\alpha}{\tau} \breg{w}{z^\star} 
\end{align*}
and hence from the preceding two displays we can conclude that 
\begin{align*}
     \frac{1}{\tau} \inangle*{\nabla_\pm \bilinear{B}(w)+\psi(w),  w-z^\star} - \frac{\alpha}{\tau} \breg{w}{z^\star} \geq 0.
\end{align*}
Thus, taking \eqref{eq:right-side-boundedness} and dividing through by $\paren{1 + \frac{\alpha}{\tau}}$ we obtain $\breg{z'}{z^\star} \leq \paren{1 + \frac{\alpha}{\tau}}^{-1} \breg{z}{z^\star}.$ Finally, the matvec complexity is evident from Lemma~\ref{lemma:judge-guarantee}. 
\end{proof}

\subsection{Implementing an divergence-bounded solution oracle}\label{sec:solver}

Here we discuss Algorithm~\ref{alg:subsolver}, which is our ultimate smooth-until-proven-guilty mirror prox algorithm for implementing a divergence-bounded solution oracle (Definition~\ref{def:oracle}). 

\RestyleAlgo{ruled}\label{alg:subsolver}
\SetKwComment{Comment}{/* }{ */}
\begin{algorithm2e}[ht]
\caption{Smooth-until-proven-guilty solver $\SUG(\cU, c, \alpha, \zcenter, \cP, \epsilon)$}
\KwInput{A constrained prox multi-point problem $(\cU, c, \alpha, \zcenter)$ (Definition~\ref{def:subproblem}), a matrix approximation path $\cP$ to $\zcenter$, and a target accuracy $\epsilon>0$.}
\KwParameter{A smoothness threshold $\tau > 0$}
$\cP^{(0)} = \{\Delta_\ell, M_\ell^{(0)}\}_{\ell \in [L]} \gets \cP$\;
$z^{(0)} \gets \argmin_{z \in \cB_{c, \zcenter}^{\cS}} r(z)$\; 
$j \gets 0, k \gets 0$\;
\While{$j \leq J$ where $J = \ceil{(1 + \tau/\alpha) \log(\Gamma_\cS/\epsilon)}$ }{
    $(z^{(j+1)}, \verdict, \cP^{(j)}) \gets \Step(\cU, \cP^{(j)}, c, \alpha, \zcenter, z^{(j)})$\;
    \lIf{$\verdict = \guilty$}{ $k \gets k+1$ \label{line:path-update-iter} }
    \lElse{$\cP^{(j+1)} \gets \cP^{(j)}$ and then $j \gets j+1$ \label{line:progress-iter}} 
}
\Return{$(z', \cP^{(J)})$}
\end{algorithm2e}

\begin{theorem}\label{thm:sug-main} For any $\tau > 0$, $\SUG$
(Algorithm~\ref{alg:subsolver}) is a divergence-bounded solution oracle for $\cS$ (Definition~\ref{def:oracle}). Moreover, for any constrained prox multi-point problem $(\cU, c, \alpha, \zcenter, \cS)$ (Definition~\ref{def:subproblem}) and matrix approximation path $\cP$ to $\zcenter$, letting $(z', \cP') \defeq \SUG(\cU, c, \alpha, \zcenter, \cP, \epsilon)$, the algorithm makes at most 
\begin{align*}
    L \left\lceil{1 + \frac{\tau}{\alpha}}  \log\paren{\frac{\Gamma_\cS}{\epsilon}} \right\rceil + \frac{L^2}{(2\pi(c) \cdot \tau)^2}[\size(\cP) - \size(\cP')] \text{ matvecs to } A. 
\end{align*}
\end{theorem}

\begin{proof} First, note that by Lemma~\ref{lemma:step-guarantee}, the algorithm maintains the invariant that $\cP^{(j)}$ is always a matrix-approximation path to $\zcenter.$ Now, on every iteration of the while loop of Algorithm~\ref{alg:subsolver}, we have that either $k$ or $j$ is incremented. We refer to iterations wherein $k$ is iterated as \emph{path update steps} and refer to iterations where $j$ is updated as \emph{convergence progress steps}. 

First, we analyze the convergence progress steps. By Lemma~\ref{lemma:step-guarantee} we have that letting $z^\star \defeq \prox_{\cU}^\alpha(\nabla_\pm f_A; \cB_{c, \zcenter}^{\cS})$, for each $j \geq 0$,
\begin{align*}
    \breg{z^\star}{z^{(j+1)}} \leq \paren{1 + \frac{\alpha}{\tau}}^{-1} \breg{z^\star}{z^{(j)}}. 
\end{align*}
Consequently, by induction, 
\begin{align*}
    \breg{z^\star}{z^{(J)}} \leq  \paren{1 + \frac{\alpha}{\tau}}^{-J} \breg{z^\star}{z^{(0)}} \leq \paren{1 + \frac{\alpha}{\tau}}^{-J} \Gamma_\cS \leq \epsilon. 
\end{align*}

To bound the matvec complexity, recall from Lemma~\ref{lemma:step-guarantee} that each call to $\Step$ runs in $O(L)$ matvecs to $A$. The total number of convergence progress steps is $J$ and each convergence progress step $j \geq 0$ maintains $\size(\cP^{(j+1)}) \leq \size(\cP^{(j)})$ (by Lemma~\ref{lemma:step-guarantee}).

Meanwhile, for each path update step, Lemma~\ref{lemma:step-guarantee} guarantees that $\size(\cP') \leq \size(\cP) - (2\pi(c) \cdot \tau)^2/L$. Thus, by induction, letting $K$ denote the total number of path update iterations, we have 
\begin{align*}
    \size(\cP^{(j)}) \leq \size(\cP^{(0)}) - \frac{K (2\pi(c) \cdot \tau)^2}{L}, 
\end{align*}
Consequently, rearranging the above expression yields the result as
\begin{align*}
    K \leq \frac{L}{(2\pi(c) \cdot \tau)^2}[\size(\cP^{(0)}) - \size(\cP^{(j)})]\,. 
\end{align*}
\end{proof}

\section{Main results}\label{sec:putting-together}

In this section, we show how to apply the machinery developed in the previous sections to obtain our main results. In Section~\ref{sec:general-analysis}, we describe a general result for matrix games under  several assumptions introduced in the previous sections. Then, in
Section~\ref{sec:applications}, we verify these assumptions for the setups associated with $\ellOneEllOne$ and $\ellTwoEllOne$ matrix games and prove Theorem~\ref{thm:final-result-l1-l1-aka-zero-sum} and Theorem~\ref{thm:final-result-l2-l1-aka-SVM}.

\subsection{Complexity analysis of general framework}\label{sec:general-analysis}

Here, we discuss how to combine the results from the previous sections to obtain a general algorithm for solving matrix games (recall \eqref{eq:intro-general-matrix-game}) under appropriate assumptions on the setup and bound its matvec complexity. 

\begin{assumptions}
In the remainder of Section~\ref{sec:general-analysis}, we fix arbitrary $\tau, \beta, \epsilon, \rho > 0$, $A \in \R^{m \times n}$, and dgf setups $\dgfsetup\x = (\xset \subset \R^n, \rx)$ and $\dgfsetup\y = (\yset \subset \R^m, \ry)$ with $\dgfsetup = (\zset \subset \R^d,r) \defeq \prodsetup(\dgfsetup\x, \dgfsetup\y)$ (recall Definition~\ref{def:product-dgf-setups}). We assume that $\Gamma_\cS$ is an upper bound on the range of $r$ so that $\sup_{z \in \cZ} \breg{z'}{z} \leq \Gamma_\cS$ for $z' \defeq \argmin_{z\in\cZ} r(z)$. We also assume that for any $z \in \cZ$ and $c > 1$, $\cB^\cS_{c, z}$ is closed and convex. Moreover, we assume that $\cS$ is $\pi$-locally bounded (recall Definition~\ref{def:consistency}) and $(\iota, \rho)$-stable with respect to a mapping $(\alpha > 0, \cU \subseteq \cZ) \mapsto \bestresponse(\alpha, \cU)$ (recall Definition~\ref{def:best-response-stability}) and that for any $\alpha > 0$ and finite nonempty $\cU \subseteq \cZ$, $\bestresponse(\alpha, \cU)$ can be computed with $O(1)$ matvecs to $A$. Furthermore, we assume that $\cS$ is $\zeta$-compatible with respect to $A$ (recall Definition~\ref{def:zeta-compatible-mapping}), and that for every $\epsilon, \delta > 0$, $\cS$ is $(\epsilon, \delta, \rho, \kappa(\epsilon, \delta))$-robust for some function $\kappa: \R_{>0} \times \R_{>0} \to \R_{>0}$ (recall Definition~\ref{def:robustness}). Finally, (as in Theorem~\ref{thm:dmp-basic}) we assume that $M$ and $\theta_r > \beta$ are fixed finite values such that for any finite nonempty multiset $\cU \subseteq \cZ$ and $\alpha > 0$, letting $z_\alpha^\star \defeq \prox_\cU^\alpha(\nabla_\pm f_A; \cZ)$ and $h(\alpha) \defeq \breg{\cU}{z^\star_\alpha}$, $h$ is $M$-Lipchitz over $[\beta, \theta_r]$ with $h(\theta_r) < 2.4 \theta_r^\rho$.
\end{assumptions}

As our first general complexity guarantee, we bound the matvec complexity of the $\MDMPImp$ subroutine in Algorithm~\ref{alg:DMP-implementation-matrix-games} by using the divergence-bounded solution oracle presented in Section~\ref{sec:sug-solver} to instantiate an approximate solution oracle $\OAPPROX$ (recall Definition~\ref{def:approx-solution} and~\ref{def:oracle}).

\begin{theorem}[Complexity of $\MDMPImp$ using $\SUG$ to implement $\OAPPROX$]\label{thm:mdmp-imp-complexity} Consider $\MDMPImp$ (Algorithm~\ref{alg:DMP-implementation-matrix-games}) instantiated with
\begin{align}\label{eq:oas-instantiate}
    \OAPPROX(\cU, c, \alpha, \bar{z}, \epsilon, \bar{\cP}, \epsilon, \delta) \gets \SUG(\cU, c, \alpha, \bar{z}, \epsilon, \bar{\cP}, \kappa(\epsilon, \delta))
\end{align}
for every constrained prox multi-point problem $(\cU, c, \alpha, \bar{z})$, matrix approximation path $\bar{\cP}$ to $\bar{z} \in \cU$, and $\epsilon, \delta > 0$. Then, $\OAPPROX$ is an $\rho$-approximate solution oracle (Definition~\ref{def:approx-solution} and~\ref{def:oracle}) and $\MDMPImp$ is a $(\beta, 2, \rho)$-kinetic $\epsilon$-$\MDMP$ (Definition~\ref{def:MDMP}.) 

Moreover, letting $(z_*, \alpha_*, \cP_*) = \MDMPImp(\cU, \cP)$ for any finite nonempty multiset $\cU \subseteq \cZ$ and a matrix approximation path $\cP$ to $z \in \cU$, $\MDMPImp(\cU, \cP)$ makes
\begin{align*}
    O\paren{\log\paren{\frac{\theta_r}{\min\{\epsilon', \beta\}}} \paren{L \left\lceil{1 + \frac{\tau}{\beta}}  \log\paren{\frac{\Gamma_\cS}{\kappa(\epsilon, \delta)}} \right\rceil + \frac{L^2}{(\pi(\iota(5)) \tau)^2}\{[\size(\cP) - \size(\cP')] + \zeta (2\alpha_*)^\rho\}}}
\end{align*}
matvecs to $A$, where $\epsilon'$ is as defined in \eqref{eq:epsilonprime}. 
\end{theorem}

\begin{proof} First, recall that by the definition of robustness (Definition~\ref{def:robustness}), $\OAPPROX$ is a $\rho$-approximate solution oracle (Definitions~\ref{def:approx-solution} and~\ref{def:oracle}). Thus, by Theorem~\ref{thm:dmp-basic}, we have that $\MDMPImp$ us an $\epsilon$-$\MDMP$. This completes the proof of the first two claims. 

To prove the final claim, we first bound the matvec complexity of a single call to the subroutine $\Validate(\alpha).$ To this end, consider a single call to $\Validate(\alpha)$ for arbitrary $\alpha > 0$. Letting $\cP$ and $\cP'$ denote the matrix-approximation path to $z$ before and after (respectively) the in-place update in Line~\ref{line:inplace}, we claim that $\Validate(\alpha)$ makes at most 
\begin{align}\label{eq:validate-complexity}
    O(1) + (L+1) \left\lceil{1 + \frac{\tau}{\alpha}}  \log\paren{\frac{\Gamma_\cS}{\kappa(\epsilon, \delta)}} \right\rceil + \frac{(L+1)^2}{(\pi(\iota(5)) \tau)^2}\{[\size(\cP) - \size(\cP')] + 3\zeta \alpha^\rho\}
\end{align}
matvecs to $A$. 

To prove this, we split into two cases. First, if $\Validate(\alpha)$ returns on Line~\ref{line:safety-check}, then it runs in $O(1)$ matvecs to $A$ (which is the cost of computing $\tilde{z} = \bestresponse(\alpha, \cU)$), thus the claimed bound in \eqref{eq:validate-complexity} is trivially true. Otherwise, by Theorem~\ref{thm:sug-main}, the call to $\OAPPROX$ in Line~\ref{line:wrap} runs in 
\begin{align}\label{eq:complexity-of-oas}
    (L+1) \left\lceil{1 + \frac{\tau}{\alpha}}  \log\paren{\frac{\Gamma_\cS}{\kappa(\epsilon, \delta)}} \right\rceil + \frac{(L+1)^2}{(\pi(\iota(5)) \tau)^2}\left[\size(\{\Delta_\ell, M_\ell\})_{\ell \in [L+1]} - \size(\{\Delta_\ell, M'_\ell\})_{\ell \in [L+1]}\right],
\end{align}
matvecs to $A$, where $\{\Delta_\ell, M'_\ell\}_{\ell \in [L+1]}$ and $\{\Delta_\ell, M_\ell\}_{\ell \in [L+1]}$ are as in Line~\ref{line:wrap}. Now, by the definition of $\size(\cdot)$ (Definition~\ref{def:matrix-approx-path}) and the update in Line~\ref{line:inplace}, we have that 
\begin{align*}
    &\size(\cP) - \size(\cP') = \left[\size(\{\Delta_\ell, M_\ell\})_{\ell \in [L]} - \size(\{\Delta_\ell, M'_\ell\})_{\ell \in [L]}\right] \\
    &= \left[\size(\{\Delta_\ell, M_\ell\})_{\ell \in [L+1]} - \size(\{\Delta_\ell, M'_\ell\})_{\ell \in [L+1]}\right] - \paren{\normInline{\Delta_{L+1} - M_{L+1}}_F^2 - \normInline{\Delta_{L+1} - M'_{L+1}}_F^2}. 
\end{align*}
Further, note that Line~\ref{line:null-model} ensures $M_{L+1} = 0$ and $\normInline{\Delta_{L+1}}_F^2 = \normInline{\ground{A}{\zcenter} - \ground{A}{z}}_F^2$. Substituting this into the display above, we have that
\begin{align*}
    &\size(\cP) - \size(\cP') \\
     =& \left[\size(\{\Delta_\ell, M_\ell\})_{\ell \in [L+1]} - \size(\{\Delta_\ell, M'_\ell\})_{\ell \in [L+1]}\right] - \paren{\normInline{\Delta_{L+1}}_F^2 - \normInline{\Delta_{L+1} - M'_{L+1}}_F^2} \\
    \geq& \left[\size(\{\Delta_\ell, M_\ell\})_{\ell \in [L+1]} - \size(\{\Delta_\ell, M'_\ell\})_{\ell \in [L+1]}\right] - \normInline{\Delta_{L+1}}_F^2 \\
    =& \left[\size(\{\Delta_\ell, M_\ell\})_{\ell \in [L+1]} - \size(\{\Delta_\ell, M'_\ell\})_{\ell \in [L+1]}\right] - \normInline{\ground{A}{\zcenter} - \ground{A}{z}}_F^2.
\end{align*}
Finally, recalling that $\cS$ is $\zeta$-compatible with respect to $A$, note that $\normInline{\ground{A}{\zcenter} - \ground{A}{z}}_F^2 \leq \zeta \breg{z}{\zcenter}$. Consequently, substituting this bound into the display above and rearranging, 
\begin{align*}
     \left[\size(\{\Delta_\ell, M_\ell\})_{\ell \in [L+1]} - \size(\{\Delta_\ell, M'_\ell\})_{\ell \in [L+1]}\right]  \leq [\size(\cP) - \size(\cP')] + \zeta \breg{z}{\zcenter}, 
\end{align*}
where by the check in Line~\ref{line:safety-check}, and the fact that $z \in \cU$, we have that $\breg{z}{\zcenter} \leq \breg{\cU}{\zcenter} \leq 3\alpha^\rho$ and consequently, 
\begin{align*}
    \left[\size(\{\Delta_\ell, M_\ell\})_{\ell \in [L+1]} - \size(\{\Delta_\ell, M'_\ell\})_{\ell \in [L+1]}\right]  \leq [\size(\cP) - \size(\cP')] + 3\zeta\alpha^\rho. 
\end{align*}
Thus, the bound in \eqref{eq:validate-complexity} holds by substituting the above bound into \eqref{eq:complexity-of-oas}. 

Now, let $(z_*, \alpha_*, \cP_*) = \MDMPImp(\cU, \cP)$ for any finite nonempty multiset $\cU \subseteq \cZ$ and a matrix approximation path $\cP$ to $z \in \cU$. By Theorem~\ref{thm:dmp-basic}, $\Validate(\alpha)$ is only ever called for $\alpha \in [\beta, \min\{2\alpha_*, \theta_r\}]$. Consequently, by \eqref{eq:validate-complexity} and Theorem~\ref{thm:dmp-basic}$, \MDMPImp(\cU, \cP)$ makes at most 
\begin{align*}
    O\paren{\log\paren{\frac{\theta_r}{\min\{\epsilon', \beta\}}} \paren{L \left\lceil{1 + \frac{\tau}{\beta}}  \log\paren{\frac{\Gamma_\cS}{\kappa(\epsilon, \delta)}} \right\rceil + \frac{L^2}{(\pi(\iota(5)) \tau)^2}\{[\size(\cP) - \size(\cP')] + \zeta (2\alpha_*)^\rho\}}}
\end{align*}
matvecs to $A$. 
\end{proof}

Combining this complexity guarantee with Theorem~\ref{thm:matrix-games-outer-loop-guarantee}, we obtain the following general result. 

\begin{theorem}[General framework complexity guarantee for matrix games]\label{thm:main-general-result} Consider Algorithm~\ref{alg:final-algo-outer-loop} instantiated with $\OAPPROX$ as in \eqref{eq:oas-instantiate} and $\OMDMP(\cU, \cP) \gets \MDMPImp(\cU, \cP)$ (Algorithm~\ref{alg:DMP-implementation-matrix-games}) and $K \gets \inceil{5 \log_2 \inparen{\Gamma_\dgfsetup (\beta \epsilon^{-1} +  \epsilon^{- \frac{\rho}{\rho + 1}}) + 2} } + 5$. Then Algorithm~\ref{alg:final-algo-outer-loop} makes
\begin{align*}
    {O} \Bigg(&\paren{K \Gamma_\dgfsetup (\beta \epsilon^{-1} +  \epsilon^{- \frac{\rho}{\rho + 1}})} \cdot \log\paren{\frac{\theta_r}{\min\{\epsilon', \beta\}}} \cdot \paren{K \left\lceil{1 + \frac{\tau}{\beta}}  \log\paren{\frac{\Gamma_\cS}{\kappa(\epsilon, \delta)}}\right\rceil} \\
    &~~+ \log\paren{\frac{\theta_r}{\min\{\epsilon', \beta\}}} \cdot \frac{K^2}{(\pi(\iota(5)) \tau)^2} \left( \zeta 2^\rho \left( K\Gamma_\cS + (T - 1)\beta^\rho + \theta_r \right) + \left(\normInline{\ground{A}{z^{(0)}}}_F^2 + 2 \zeta K \Gamma_\cS\right)\right)\Bigg)
\end{align*}
matvecs to $A$ and the output $\zbar$ is a $2 \epsilon$-solution of \eqref{eq:intro-general-matrix-game}. 
\end{theorem}
\begin{proof} This result follows from summing the complexity bound from Theorem~\ref{thm:mdmp-imp-complexity} and applying the bounds \eqref{eq:sum-alpha-bound-takeaway} and \eqref{eq:final-paths-diff-bound-takeaway} from Theorem~\ref{thm:matrix-games-outer-loop-guarantee}. Indeed, by Theorem~\ref{thm:mdmp-imp-complexity}, the overall query complexity can be bounded (up to big-$O$) as 
\begin{align*}
    &\log\paren{\frac{\theta_r}{\min\{\epsilon', \beta\}}} \sum_{t \in [T]} \insquare*{ K \left\lceil{1 + \frac{\tau}{\beta}}  \log\paren{\frac{\Gamma_\cS}{\kappa(\epsilon, \delta)}} \right\rceil + \frac{K^2}{(\pi(\iota(5)) \tau)^2}\left([\size({\cP}^{(t)}) - \size({\cP'}^{(t)})] + \zeta (2\alpha^{(t)})^\rho\right) } \\
    &= \log\paren{\frac{\theta_r}{\min\{\epsilon', \beta\}}} \cdot TK \left\lceil{1 + \frac{\tau}{\beta}}  \log\paren{\frac{\Gamma_\cS}{\kappa(\epsilon, \delta)}}\right\rceil \\
    &~~+ \log\paren{\frac{\theta_r}{\min\{\epsilon', \beta\}}} \cdot \frac{K^2}{(\pi(\iota(5)) \tau)^2} \paren{ \sum_{t \in [T]} [\size(\cP^{(t)}) - \size({\cP'}^{(t)})] + \sum_{t \in [T]} \zeta 2^\rho (\alpha^{(t)})^\rho }. 
\end{align*}
The theorem now follows immediately from the bounds on $\sum_{t \in [T]} [\size(\cP^{(t)}) - \size({\cP'}^{(t)})]$ and $\sum_{t \in [T - 1]}  (\alpha^{(t)})^\rho$ from \eqref{eq:final-paths-diff-bound-takeaway} and \eqref{eq:sum-alpha-bound-takeaway} respectively in Theorem~\ref{thm:matrix-games-outer-loop-guarantee}, along with the fact that $\alpha^{(T)} \le \theta_r$ since the output $\alpha_*$ of \Cref{alg:DMP-implementation-matrix-games} is always at most $\theta_r$, since the latter is an upper bound on the bisection search interval (see \Cref{lemma:cautious_search}).
\end{proof}

\subsection{Applications to $\ellOneEllOne$ and $\ellTwoEllOne$ Matrix Games}\label{sec:applications}

In order to prove our main results Theorem~\ref{thm:final-result-l1-l1-aka-zero-sum} and \ref{thm:final-result-l2-l1-aka-SVM}, we first define the canonical setups that we consider in this paper. 

\begin{definition}[$\ellOneEllOne$ and $\ellTwoEllOne$ setups] \label{def:matrix-games-setups}
With $d \defeq n + m$, we refer to the tuples  $(\xset, \yset, \xtrunc, \ytrunc, \cZint, \rx : \xset \to \R, \ry : \yset \to \R, \Gamma_{\dgfsetup})$ defined in Table \ref{table:matrix-games-setups} as the \emph{$\ellOneEllOne$ and $\ellTwoEllOne$ setups} respectively. In the context of these setups, we further define $\zset \defeq \xset \times \yset$ and $r : \zset \to \R$ via $r(z) \defeq \rx(z\x) + \ry(z\y)$. Furthermore, we define what we call \emph{truncated} domains, which restrict simplex-constrained coordinates to be at least $\nu$, with $\ztrunc \defeq \xtrunc \times \ytrunc$. Finally, we make the standard (see, e.g., \cite{kornowski2024oracle,kornowski2024oracleupdated,carmon2019variance, carmon2024whole, karmarkar2025solvingzerosumgames}) normalization assumptions $\inmaxnorm{A} \le 1$ in the $\ellOneEllOne$ setup and $\innorm{A}_{2 \to \infty} \le 1$ in the $\ellTwoEllOne$ setup.
\end{definition}

\begin{table*}[ht]
\centering
\begin{tabular}{ c c c }
\hline
 & \textbf{$\ell_1$-$\ell_1$} & \textbf{$\ell_2$-$\ell_1$}   \\ \hline
$\cX$     & $\Delta^n$     & $\mathbb{B}^n$       \\ 
$\cY$     & $\Delta^m$     & $\Delta^m$     \\ 
$\cX_\nu$ & $\Delta_\nu^n$     & $\mathbb{B}^n$       \\ 
$\cY_\nu$ & $\Delta_\nu^m$     & $\Delta_\nu^m$     \\
$\cZint$ & $\Delta_{>0}^n \times \Delta_{>0}^m$ & $\mathbb{B}^n \times \Delta_{>0}^n $      \\ 
$\rx(x)$     & $\frac{1}{2}\norm{x}_2^2$    & $\frac{1}{2}\norm{x}_2^2$    \\
$\ry(y)$     & $\frac{1}{2}\norm{y}_2^2$    & $\sum_{i \in [m]} [y]_i \log([y]_i)$    \\
$\Range_{\cS}$     & $\log(mn)$    & $\frac{1}{2} + \log(m)$   \\ 
$V^r_{z'}(z)$ & $\KL(z || z')$ & $\frac{1}{2} \norm{z\x - z'\x}^2_2 + \KL(z\y||z'\y)$\\
\hline
\end{tabular}
\caption{$\ellOneEllOne$ and $\ellTwoEllOne$ setups (Definition~\ref{def:matrix-games-setups}) and associated notation.}
\label{table:matrix-games-setups}
\end{table*}

Note that for both the $\ellOneEllOne$ and $\ellTwoEllOne$ setups, we have that ${\dgfsetup\x}_\nu \defeq (\xset_\nu, \rx)$ and ${\dgfsetup\y}_\nu \defeq (\yset, \ry)$ are dgf setups (\Cref{def:dgf-setup}) with $\dgfsetup_\nu \defeq (\zset_\nu, r) = \prodsetup({\dgfsetup\x}_\nu, {\dgfsetup\y}_\nu)$ (\Cref{def:product-dgf-setups}). Furthermore, $\Range_\dgfsetup = \max_{z, z' \in \zset} r(z) - r(z') \ge \max_{z, z' \in \ztrunc} r(z) - r(z')$.

In the remainder of this section, we verify the assumptions outlined in Section~\ref{sec:general-analysis} for our application to (appropriately truncated) $\ellOneEllOne$ or $\ellTwoEllOne$ setups (Definition~\ref{def:matrix-games-setups}). First, we show that it suffices to solve the problem constrained the \emph{truncated} setup $\cS_\nu$ for appropriate $\nu > 0$. Then, we show how to instantiate the best-response mapping $\bestresponse$ for these setups for use in the the $\MDMP$ implementation from Section~\ref{sec:MDMP-implementation} and prove stability (Definition~\ref{def:best-response-stability}). Next we show that the mapping $z \mapsto \ground{A}{z}$ defined in Definition~\ref{def:product-dgf-setups} is $O(1)$-compatible in these setups. We then show that these setups are also appropriately locally-bounded (in the sense of Definition~\ref{def:consistency}), allowing use to invoke the inner subproblem solver $\SUG$ from Section~\ref{sec:sug-solver}. Finally, we discuss how, for these truncated setups, we can prove a robustness condition (Definition~\ref{def:robustness}). Combining these results, we prove Theorem~\ref{thm:final-result-l1-l1-aka-zero-sum} and Theorem~\ref{thm:final-result-l2-l1-aka-SVM}. For notational convenience, and to avoid redundancy, we handle the $\ellOneEllOne$ and $\ellTwoEllOne$ setups \emph{jointly} in our analysis, with distinctions between the two setups being deferred to the proofs of intermediate lemmas.  

\paragraph{Assumptions.}
In the remainder of this section, we assume that is  $\cS = (\xset, \yset, \xtrunc, \ytrunc, \cZint, \rx : \xset \to \R, \ry : \yset \to \R, \norm{\cdot} : \R^d \to \R, \Gamma_{\dgfsetup})$ is fixed to be any of the setups defined in Definition~\ref{def:matrix-games-setups} for a fixed but arbitrary $\nu \in (0, 1/d)$. In particular, we may also use the notation ${\dgfsetup\x}_\nu$, ${\dgfsetup\y}_\nu$, and ${\dgfsetup}_\nu$ defined above. Furthermore, we fix $A \in \R^{m\times n}$ to be any matrix satisfying the normalization assumptions in Definition~\ref{def:matrix-games-setups}. 

\paragraph{Truncation.}
First, we restate the following reduction from \citep{karmarkar2025solvingzerosumgames} shows that in order to compute an approximate solution to \eqref{eq:intro-general-matrix-game} 
it suffices to compute an approximate solution to the same problem over the truncated domain $\cZ_\nu = \cX_\nu \times \cY_\nu$ for an appropriate $\nu$ which scales inverse polynomially in the problem parameters ($m, n, \epsilon$). 

\begin{lemma}[Lemma 6.2 of \citep{karmarkar2025solvingzerosumgames}, restated] \label{lem:truncation-for-ell2ell1-ell1ell1}
For $\epsilon > 0$ and $0 < \nu \le \frac{\min \inbraces{\epsilon, 1}}{8 \max \inbraces{m, n}}$, any $\epsilon / 2$-solution $z' \in \ztrunc$ of 
\begin{align}
    \min_{x \in \xset_\nu} \max_{y \in \yset_\nu} \inangle*{y, Ax}. \label{eq:truncated-minimax-prob}
\end{align}
is an $\epsilon$-solution of \eqref{eq:intro-general-matrix-game}.
\end{lemma}
\paragraph{Stability.} Here, we show that the dgf stables considered in the section are $\iota$-stable with respect to the following mapping $(\alpha, \cU) \mapsto \bestresponse(\alpha, \cU)$. The following definition builds upon Definition 6.3 of \citep{karmarkar2025solvingzerosumgames}. 

\begin{definition}[$(\alpha, \cU)$-best-response]\label{def:best-response} Let $\alpha > 0$, $\cU \subseteq \cZ_\nu$ be a finite and nonempty multiset, and $\mean{\cU} \defeq \frac{1}{\abs{\cU}} \sum_{u \in \cU} u$. We define $\bestresponse(\alpha, \cU) = \prox_{\cU}^\alpha(\nabla_\pm f_A(\mean{\cU}; \cZ_\nu)$ (recall Definition~\ref{def:proximal-mappings}). That is, letting $\tilde{z} = (\tilde{x}, \tilde{y}) = \bestresponse(\alpha, \cU)$,
\begin{align*}
    \Tilde{x} = \argmin_{x \in \cX_\nu} \inangle*{{\mean{\cU}}\y,  Ax} + \alpha \xbreg{\cU\x}{x} \text{~~~and~~~}  \Tilde{y} = \argmax_{y \in \cY_\nu} \inangle*{y, A~{\mean{\cU}\x}} - \alpha \ybreg{\cU\y}{y}. 
\end{align*}
In particular, note that for any $\alpha > 0$ and finite nonempty multiset $\cU \subseteq \cZ$, $\bestresponse(\alpha, \cU)$ can be computed with $O(1)$ matvecs to $A$. 
\end{definition}

In the case of zero-sum games (the $\ellOneEllOne$ setup) the $\bestresponse(\alpha, \cU)$ can be interpreted as follows. Each player calculates each player's \emph{best response} (over $\cX_\nu$ and $\cY_\nu$) to their opponent, holding the opponents' strategy to be \emph{fixed} to be the \emph{average} of the strategies in $\cU$, subject to an $\alpha$-regularization penalty for each $u \in \cU$. We now prove the following analog of Lemma 6.4 of \citep{karmarkar2025solvingzerosumgames}.

\begin{lemma}\label{lemma:new-stability} $\cS_\nu$ is $(\iota, 2)$-stable (Definition~\ref{def:best-response-stability}) with respect to the mapping $\bestresponse(\alpha, \cU)$ defined in Definition~\ref{def:best-response} for $\iota: c \mapsto \exp(2 \sqrt{2c})$. 
\end{lemma}

Recall from Definition~\ref{def:best-response-stability} that to prove this result, we need to show that for any $\alpha, c > 0$ and any finite nonempty set $\cU \subseteq \cZ_\nu$, letting $z^\star = \prox_{\cU}^\alpha(\nabla_\pm f_A; \cZ_\nu)$ and $\Tilde{z} = \bestresponse(\alpha, \cU)$, whenever $\breg{\cU}{z^\star} \leq c\alpha^2$, we must have $z^\star \in \cB^\cZ_{\exp(2\sqrt{2c}), \Tilde{z}}$. To prove this, we reduce to Lemma 6.6 of \citep{karmarkar2025solvingzerosumgames}. 

\begin{lemma}[Lemma 6.6 of \citep{karmarkar2025solvingzerosumgames}]
    \label{lem:general-stability-helper}
    For $\alpha > 0$, vectors $\theta, \xi \in \R^d$, and $q \in \simplex^d_\nu$, define
    \begin{align*}
        u_\theta \defeq \argmin_{z \in \simplex^d_\nu} \inangle*{\theta, z} + \alpha \cdot \inKL{q}{z} ~~~\text{and}~~~ u_\xi \defeq \argmin_{z \in \simplex^d_\nu} \inangle*{\xi, z} + \alpha \cdot \inKL{q}{z}.
    \end{align*}
    Then $u_\theta \approx_{\delta} u_\xi$ with $\delta \defeq \exp \inparen*{\frac{2 \norm{\theta - \xi}_\infty}{\alpha}}$.
\end{lemma}
  
\begin{proof}[Proof of Lemma~\ref{lemma:new-stability}] 
Let $z^\star = \prox_{\cU}^\alpha(\nabla_\pm f_A; \cZ_\nu)$ and $\Tilde{z} = \bestresponse(\alpha, \cU)$. Suppose that $\breg{\cU}{z^\star} \leq c \alpha^2$. By the optimality conditions, we have that 
\begin{align*}
    z^\star\x = \argmin_{x \in \cX_\nu} \inangle*{{z\y^\star}, A x} + \alpha V^{\rx}_{\cU\x}(x), ~~~\text{ and }~~~  z^\star\y = \argmax_{y \in \cY_\nu} \inangle*{y, A z\x^\star} - \alpha V^{\ry}_{\cU\y}(y). 
\end{align*}
Corollary~\ref{corr:reduce-sum-to-one-point} ensures the existence of a $\collapsed{\cU} \in \cZ_\nu$ such that 
\begin{align*}
    z^\star\x = \argmin_{x \in \cX_\nu} \inangle*{z\y^\star, A x} + \alpha \cdot |\cU| \cdot V^{\rx}_{{\collapsed{\cU}}\x}(x), &~\text{ and }~ z^\star\y = \argmax_{y \in \cY_\nu} \inangle*{y, A z\x^\star} - \alpha \cdot |\cU| \cdot V^{\ry}_{{\collapsed{\cU}}\y}(y), \\
    \tilde{z}\x = \argmin_{x \in \cX_\nu} \inangle*{{\mean{\cU}}\y, A x} + \alpha \cdot |\cU| \cdot V^{\rx}_{{\collapsed{\cU}}\x}(x), &~\text{ and }~ \Tilde{z}\y = \argmin_{y \in \cY_\nu} \inangle*{y, A{\mean{\cU}\x}} - \alpha \cdot |\cU| \cdot V^{\ry}_{{\collapsed{\cU}}\y}(y). 
\end{align*}
In the $\ell_1$-$\ell_1$ setup (Definition~\ref{def:matrix-games-setups}), by Lemma~\ref{lem:general-stability-helper}, we have that 
\begin{align*}
    {z}\x^\star &\approx_{\delta\x} \Tilde{z}\x \text{ for } \delta\x = \exp\paren{\frac{2\normInline{A^\top\mean{\cU}\y - A^\top z\y^\star}_\infty}{\alpha}}, \\
    {z}\y^\star &\approx_{\delta\y} \Tilde{z}\y \text{ for } \delta\y = \exp\paren{\frac{2\normInline{A\mean{\cU}\x - Az\x^\star}_\infty}{\alpha}}. 
\end{align*}
Now, 
\begin{align*}
    \normInline{A^\top\mean{\cU}\y - A^\top z\y^\star} &\leq \normInline{A}_{\max} \normInline{{\mean{\cU}}\y - z\y^\star}_1 \leq \sqrt{\frac{1}{\abs{\cU}} \sum_{u \in \cU} \normInline{ \mean{\cU}\y- z^\star\y}_1^2} \\
    &\leq \sqrt{\frac{2}{\abs{\cU}} V^{\ry}_{\cU\y}(z^\star\y)} \leq \alpha \sqrt{2c}
\end{align*}
where the second inequality used Jensen's inequality and the convexity of $\normInline{\cdot}^2$, as well as the property that $\normInline{A}_{\max} \leq 1$. Hence, $\delta\x \leq \exp(2\sqrt{2c})$. An identical argument shows $\delta\y \leq \exp(2\sqrt{2c})$. 

Now, consider the $\ell_2$-$\ell_1$ setup (Definition~\ref{def:matrix-games-setups}). Again, by Lemma~\ref{lem:general-stability-helper}, we have that 
\begin{align*}
    {z}\y^\star &\approx_{\delta} \Tilde{z}\y \text{ for } \delta = \exp\paren{\frac{2\normInline{A^\top\mean{\cU}\x - A z\x^\star}_\infty}{\alpha}}. 
\end{align*}
Similar to before,
\begin{align*}
    \normInline{A\mean{\cU}\x - A z\x^\star} &\leq \max_i \norm{
    A_{:, i}
    }_2 \normInline{{\mean{\cU}}\y - z\y^\star}_2 \leq \sqrt{\frac{1}{\abs{\cU}} \sum_{u \in \cU} \normInline{ u\y- z^\star\y}_2^2} \\
    &\leq \sqrt{\frac{2}{\abs{\cU}}  V^{\ry}_{\cU\y}(z^\star\y)} \leq \alpha \sqrt{2c}
\end{align*}
where the second inequality used Jensen's inequality and the convexity of norms, and the property that $\max_i \normInline{A_{:, i}}_2 \leq 1$. Hence, $\delta \leq \exp(2\sqrt{2c})$. 
\end{proof}

\paragraph{Compatibility.} Here, we show that the setup $(\cZ_\nu, \normInline{\cdot}, r)$ is $2$-compatible with respect to $A$ (Definition~\ref{def:zeta-compatible-mapping}). To aid in the proof, for any $z_1, z_2 \in \R^d$ we let $H^2(z_1, z_2) \defeq \sum_{i \in [d]} \paren{\sqrt{[z_1]_i} - \sqrt{[z_2]_i}}^2$ denote the squared \emph{Hellinger distance} between $z_1, z_2 \in \cZ$. 

\begin{lemma}[Compatibility]\label{lemma:compatibility} 
For any $A \in \R^{m \times n}$ satisfying the normalization assumptions of the setup (Definition~\ref{def:matrix-games-setups}), the setup $\cS_\nu$ is $2$-compatible (Definition~\ref{def:zeta-compatible-mapping}) with respect to $A$. 
\end{lemma}

\begin{proof} In the $\ell_1$-$\ell_1$ setup, 
\begin{align*}
    \normInline{\ground{A}{z} - \ground{A}{z'}}_F^2 &= \normInline{\diag(z\y)^{1/2} \cdot A \cdot \diag(z\x)^{1/2} - \diag(z'\y)^{1/2} \cdot A \cdot \diag(z'\x)^{1/2}}_F^2 \\
    &= \sum_{i \in [m]} \sum_{j \in [n]} \paren{\sqrt{[z\y]_i} A_{ij} \sqrt{[z\x]_j} - \sqrt{[z'\y]_i} A_{ij} \sqrt{[z'\x]_j}}^2 \\
    &\leq \normInline{A}_{\max}^2 \sum_{i \in [m]} \sum_{j \in [n]} \paren{\sqrt{[z\y]_i} \sqrt{[z\x]_j} - \sqrt{[z'\y]_i} \sqrt{[z'\x]_j}}^2
 \end{align*}
Now, using the property that for any real numbers $a, b, c, d$ we have
\begin{align*}
    (ab-cd)^2 = (a(b-d) + d(a-c))^2 \leq 2a^2(b-d)^2 + 2d^2(a-c)^2, 
\end{align*}
we have that (taking $a = \sqrt{[z\y]_i}, b = \sqrt{[z\x]_j}, c = \sqrt{[z'\y]_i}$ and $d = \sqrt{[z'\x]_j}$ above),  
\begin{align*}
    \normInline{\ground{A}{z} - \ground{A}{z'}}_F^2 
    &\leq 2\sum_{i \in [m]} \sum_{j \in [n]} [z\y]_i \paren{\sqrt{[z\x]_i} - \sqrt{[z'\x]_i}}^2 + 2\sum_{i \in [m]} \sum_{j \in [n]} [z'\x]_i \paren{\sqrt{[z\y]_i} - \sqrt{[z'\y]_i}}^2 \\
    &\leq 2H^2(z\x, z'\x) + 2H^2(z\y, z'\y) \\
    &\leq 2\KL(z || z'). 
 \end{align*}
where the second-to-last step uses that $z\y, z\x$ are in the probability simplex, and the last step is true by Fact~\ref{fact:hellinger-trick}. In the $\ell_2$-$\ell_1$ setup,
\begin{align*}
    \normInline{\ground{A}{z} - \ground{A}{z'}}_F^2 &= \normInline{(\diag({z\y})^{1/2} - \diag({z'\y})^{1/2}) A}_F^2 \\
    &= \sum_{i \in [m]} (\sqrt{[z\y]_i} - \sqrt{[z'\y]_i})^2 \sum_{j \in [n]} A_{ij}^2 \\
    &\leq \max_{i \in [n]} \normInline{A_{i, :}}_2^2 \cdot H^2(z\y, z'\y) \leq \KL(z\y || z'\y) \leq \breg{z'}{z}, 
 \end{align*}
 where the last step is true by Fact~\ref{fact:hellinger-trick}. 
\end{proof}

\paragraph{Local-boundedness.} Here we verify the local boundedness condition introduced in Definition~\ref{def:best-response}.

\begin{lemma}\label{lemma:local-boundedness} There exists an explicit function $\pi: \R_{>1} \mapsto \R_{>0}$ such that $\cS_\nu$ is $\pi$-locally bounded (Definition~\ref{def:best-response}). Moreover, $\pi$ is a \emph{universal} function independent of any problem parameters. 
\end{lemma}
\begin{proof} Let $z_1, z_2 \in \cB_{c, \tilde{z}}$ for some $\tilde{z} \in \cZ$ and $c > 1$. In the $\ell_2$-$\ell_1$ setup, 
\begin{align*}
    \breg{z_1}{z_2} &= \frac{1}{2}\normInline{{z_1}\x - {z_2}\x}_2^2 + \KL(z_1|| z_2)
\end{align*}
and
\begin{align*}
    \normInline{z_1 - z_2}_{z'}^2 &= \frac{1}{2}\normInline{{z_1}\x - {z_2}\x}_2^2 + \inangle*{(z_1 -z_2)\y, \diag(z'\y)^{-1} (z_1-z_2)\y}, 
\end{align*}
in which case the result is an immediate consequence of Lemma 5.3 of \citep{karmarkar2025solvingzerosumgames}. The argument for $\ell_1$-$\ell_2$ and $\ell_1$-$\ell_1$ setups is analogous. 
\end{proof}
\citep{karmarkar2025solvingzerosumgames} explicitly characterize $\pi$ and show that it is a \emph{universal} function independent of any problem parameters; however, the specific formula is not important for our purposes, and hence, we omit it for brevity.

\paragraph{Robustness.} Here, we show that for any $\epsilon, \delta > 0$, the setup $\cS_\nu$ is $(\epsilon, \delta, 2, \kappa(\epsilon,\delta))$-robust (Definition~\ref{def:robustness}) for an appropriately defined $\kappa: \R_{>0} \times \R_{>0} \to \R_{>0}$. 

\begin{lemma}\label{lem:reducetoinner} There exists an absolute constant $C > 0$ such that for any $\epsilon, \delta > 0$ and
\begin{align*}
    \kappa(\epsilon, \delta) \defeq C \min\left\{ \delta^2 \nu^2, \frac{\alpha^{4}}{|\uset|^2  ( 1 + \log^2(\nu^{-1}))}, \frac{\epsilon^2 }{1 + (\alpha\abs{\cU})^2 \nu^{-2}}\right\} \, , 
\end{align*}
the setup $\mathcal{S}_\nu$ is $(\epsilon, \delta, 2, \kappa(\epsilon, \delta))$-robust (Definition~\ref{def:robustness}).
\end{lemma}
\begin{proof}
    Per Lemma~\ref{lem:sufficient-cond-for-robustness}, it suffices to show that for any $(\uset, c, \alpha, z, \dgfsetup_\nu)$-constrained prox multi-point problem where $\uset, c, \alpha, z$ are arbitrary (up to the restrictions of Definition~\ref{def:subproblem}) and with solution $\zopt \defeq \prox^{\alpha}_{\uset}(\gm f_A ; \sball^{\dgfsetup_\nu}_{c, z})$, any $z' \in \sball^{\dgfsetup_\nu}_{c, z}$ with $\breg{\zopt}{z'} \le \kappa(\epsilon, \delta)$ satisfies
    \begin{itemize}
    \item $\abs{\breg{\cU}{z'} - \breg{\cU}{z^\star}} < \alpha^2 / 10$, 
    \item $z' \in \cB^{\cS_\nu}_{1+\delta, z^\star}$, and
    \item if $\prox_\cU^\alpha(\nabla_\pm f_A; \cZ_\nu) \in \cB_{c, z}^{\cS_\nu}$ then $\inangle{\nabla_{\pm}f_A(z') + \alpha \grad \breg{\cU}{z'}, z' -u} \leq \epsilon, \text{ for all } u \in \cZ_\nu$. 
    \end{itemize}
    The first property follows from Corollary~\ref{corr:reduce-sum-to-one-point} and \cite[Lemma B.5]{karmarkar2025solvingzerosumgames}, and the second property follows from \cite[Lemma B.4]{karmarkar2025solvingzerosumgames}. To prove the third property, we will follow the proof of \cite[Lemma 6.15]{karmarkar2025solvingzerosumgames} with minor modifications. Note that by Corollary~\ref{corr:reduce-sum-to-one-point}, there exists some $\qU \in \ztrunc$ such that $\grad \breg{\uset}{w} = |\uset| \grad \breg{\qU}{w}$ for all $w \in \ztrunc$. Then combining this with \eqref{eq:Bregman-three-point-equality}, we have for all $u \in \ztrunc$:
    \begin{align*}
       \inangle*{\nabla_{\pm}f_A(z') + \alpha \grad \breg{\cU}{z'}, z' -u} &= \inangle{\nabla_{\pm}f_A(z') + \alpha |\uset| \grad \breg{\qU}{z'}, z' -u} \\
       &= \inangle{\nabla_{\pm}f_A(z'), z' - u} - \alpha |\uset| \cdot \insquare{\breg{\qU}{u} - \breg{z'}{u} - \breg{\qU}{z'}} \, .
    \end{align*}
    Using analogous manipulations, we have for all $u \in \ztrunc$:
    \begin{align*}
        \inangle{\gm f_A(\zopt), \zopt - u} \le \alpha |\uset| \cdot \insquare{\breg{\qU}{u} - \breg{\zopt}{u} - \breg{\qU}{\zopt}} \, .
    \end{align*}
    The remainder of the proof follows the same steps as the proof of \cite[Lemma 6.15]{karmarkar2025solvingzerosumgames}, except we set $z_{\mathsf{c}}$, $\alpha$, $\gammav$, and $w$ in the proof of \cite[Lemma 6.15]{karmarkar2025solvingzerosumgames} to $\qU$, $\alpha |\uset|$, $\epsilon$, and $z'$ respectively. 
\end{proof}

\paragraph{Lipschitzness bound and binary search range.} Below, we provide a lemma to motivate our instantiation of the range upper bound $\theta_r$ in the implementation of $\MDMPImp$ (Algorithm~\ref{alg:DMP-implementation-matrix-games}). The proof of the following lemma is very similar to that of Lemma B.12 of \citep{karmarkar2025solvingzerosumgames}.

\begin{lemma}
    \label{lem:starting-value-b-search}
For any finite nonempty multiset $\cU \subseteq \zset_\nu$ and $\alpha \ge 1$, $w \defeq \prox_\cU^{\alpha}(\gm f_A; \zset_\nu)$ satisfies $\breg{\cU}{w} \le 12\abs{\cU} \log\paren{\frac{1}{\nu d}}$.
\end{lemma}

\begin{proof} First, note that by the normalization assumptions on the matrix $A$ in Definition~\ref{def:matrix-games-setups}, $|f_A(z)| \le 1$ for all $z \in \zset_\nu$. For the sake of contradiction, suppose that $\breg{\cU}{w} > 12|\cU|\log\paren{\frac{1}{\nu d}}$. Then, we must have that either $\breg{\cU\x}{w\x} > 6|\cU|\log\paren{\frac{1}{\nu d}}$ or $\breg{\cU\y}{w\y} > 6|\cU|\log\paren{\frac{1}{\nu d}}$. If $\breg{\cU\x}{w\x} > 6|\cU|\log\paren{\frac{1}{\nu d}}$, then we have 
\begin{align*}
    w\x = 
    \argmin_{x \in \xset_\nu} f_A(x, w\y) + \alpha \xbreg{\cU\x}{x}.
\end{align*}
Lemmas~\ref{lemma:uncondition-simplex-bound} and \ref{lemma:unconditional-euclidean-bound} guarantee that there exists an $x' \in \cX_\nu$ such that 
\begin{align*}
    f_A(x', w\y) + \alpha \xbreg{\cU\x}{x'} &= f_A(x', w\y) \le 1 + 4\alpha \abs{\cU}\log\paren{\frac{1}{\nu d}} \\
    &\leq 5\alpha \abs{\cU}\log\paren{\frac{1}{\nu d}} < f_A(w\x, w\y) + \alpha \xbreg{\cU\x}{w\x}, 
\end{align*}
which is a contradiction. A similar argument holds if $\breg{\cU\y}{w\y} > 6\alpha|\cU|\log\paren{\frac{1}{\nu d}}$.
\end{proof}

\begin{corollary}
    \label{corr:starting-value-b-search}
For any finite nonempty multiset $\cU \subseteq \zset_\nu$ and $\alpha \geq \frac{1}{2} \sqrt{12 \abs{\cU} \log\paren{\frac{1}{\nu d}}}$, we have that $w \defeq \prox_\cU^{\alpha}(\gm f_A; \zset_\nu)$ satisfies $\breg{\cU}{w} < 2\alpha^2$.
\end{corollary}
\begin{proof} By Lemma~\ref{lem:starting-value-b-search}, for any $\alpha \geq 1$, we have that 
\begin{align*}
    \breg{\cU}{w} \leq 12 \abs{\cU} \log\paren{\frac{1}{\nu d}}. 
\end{align*}
Consequently, for $\alpha \geq \frac{1}{2} \sqrt{12 \abs{\cU} \log\paren{\frac{1}{\nu d}}}$, the bound holds. 
\end{proof}

We prove the following guarantee regarding the Lipschitzness of the function $h$ defined in Theorem~\ref{thm:dmp-basic}. 

\begin{restatable}{lemma}{lip}\label{lem:Lipschitzness-of-h} For any $z \in \cZ_\nu$, finite nonempty multiset $\cU \subseteq \cZ_\nu$, and $\alpha > 0$, define $h: \R_{>0} \to \R$ by $h(\alpha) \defeq \breg{\cU}{\prox_{\cU}^\alpha(\nabla_\pm f_A; \cZ_\nu}$. Then, for any $0 < \theta_\ell < \theta_r$ and $\alpha, \alpha' \in [\theta_\ell, \theta_r]$ there exists an absolute constant $M > 0$ such that 
\begin{align*}
    \abs{h(\alpha) - h(\alpha')} \leq M \abs{\cU}^3 \paren{\frac{1 + \log(\nu^{-1})}{\nu \theta_\ell} + \abs{\cU}\theta_r} \cdot \abs{\alpha - \alpha'}. 
\end{align*}
\end{restatable}

Our proof leverages the following lemma of \citep{karmarkar2025solvingzerosumgames}.

\begin{lemma}[Lemma B.11 of \citep{karmarkar2025solvingzerosumgames}, restated]
    \label{lem:Lipschitzness-of-h-og}
For a fixed $q \in \ztrunc$ and parameter $\alpha > 0$, let $w_\alpha \defeq \prox_q^{\alpha}(\gm f_A; \ztrunc)$ (namely, $w_\alpha$ is parameterized by $\alpha$) and define $\vartheta : \R_{>0} \to \R$ via $\vartheta(\alpha) \defeq \breg{q}{w_\alpha} - 2 \alpha^2$. Then for any $\alpha, \alpha' \in [b, c]$ for some $c > b > 0$, there exists an absolute constant $M' > 0$ such that
\begin{align*}
    |\vartheta(\alpha) - \vartheta(\alpha')| \le M' \inparen*{\frac{1 + \log \nu^{-1}}{\nu b} + c} \cdot |\alpha - \alpha'|.
\end{align*}
\end{lemma}

\begin{proof}[Proof of Lemma~\ref{lem:Lipschitzness-of-h}] Let $w_\alpha \defeq \prox_{\cU}^\alpha(\nabla_\pm f_A; \cZ_\nu)$ and $w_{\alpha'} \defeq \prox_{\cU}^{\alpha'}(\nabla_\pm f_A; \cZ_\nu)$. By Corollary~\ref{corr:reduce-sum-to-one-point}, there is a point $\collapsed{\cU}$ such that $w_\alpha =\prox_{\collapsed{\cU}}^{\alpha\abs{\cU}}(\gm f_A; \ztrunc)$ 
and $w_{\alpha'} = \prox_{\collapsed{\cU}}^{\alpha\abs{\cU}}(\gm f_A; \ztrunc)$. Consequently, by Corollary~\ref{corr:reduce-sum-to-one-point}, we have that 
\begin{align*}
    \abs{h(\alpha) - h(\alpha')} &= \abs{\abs{\cU} \cdot \breg{\collapsed{\cU}}{w_\alpha} - \abs{\cU}\breg{\collapsed{\cU}}{w_{\alpha'}}} \\
    &= \abs{\abs{\cU} \cdot \breg{\collapsed{\cU}}{w_\alpha} - \alpha^2 - \abs{\cU}\breg{\collapsed{\cU}}{w_{\alpha'}} + {\alpha'}^2 + (\alpha^2 - {\alpha'}^2)} \\
    &= \abs{\cU}\abs{\vartheta(\abs{\cU}\alpha) - \vartheta(\abs{\cU}\alpha')} + 4R|\alpha - \alpha'| \\
    &\leq M' \abs{\cU}^2 \paren{\frac{1 + \log(\nu^{-1})}{\nu \abs{\cU}L} + \abs{\cU} R} \cdot \abs{\abs{\cU} (\alpha - \alpha')} + 4R (\alpha - \alpha') \\
    &\leq M' \abs{\cU}^4 \paren{\frac{1 + \log(\nu^{-1})}{\nu \abs{\cU}b} + 5R} \abs{\alpha - \alpha'}, 
\end{align*}
where in the third line we used the fact that the function $x \mapsto x^2$ is $4R$-Lipschitz on $[L, R]$ and in the fourth line we used Lemma~\ref{lem:Lipschitzness-of-h-og}. 
\end{proof}

\paragraph{Proof of main results.} Finally, we conclude by proving our main results. In the proof of the main results, we fix the following: 
\begin{itemize}
    \item Setup: $\cS_\nu$ as defined in Definition~\ref{def:matrix-games-setups}; 
    \item $\nu := \min\left\{\frac{\min\{\epsilon, 1\}}{8 \max\{m, n\}}, \frac{1}{d}\right\}$;
    \item Compatibility (Definition~\ref{def:zeta-compatible-mapping}): $\zeta = 2$; 
    \item Stability mapping (Definition~\ref{def:best-response-stability}): $\iota: c \mapsto \exp(2\sqrt{2c})$ as in Lemma~\ref{lemma:new-stability};
    \item Local boundedness mapping (Definition~\ref{def:consistency}): $\pi$ as guaranteed by Lemma~\ref{lemma:local-boundedness}; 
    \item $\MDMP$ parameters: $\beta = \epsilon^{1/3}, \rho = 2, \gamma = 2$; 
    \item Approximation parameters: $\delta = \frac{1}{2} \paren{\exp(2\sqrt{10} - 4\sqrt{2}) - 1} \approx .47$ (it is easy to verify that this satisfies the constraints of Line~\ref{line:delta2}) and $\kappa(\epsilon, \delta)$ as defined in Lemma~\ref{lem:reducetoinner};
    \item $\SUG$ parameter: $\tau = \beta$; 
    \item Bisection-search parameters: $\theta_\ell = \beta, \theta_r = 1/2\cdot\sqrt{12 \abs{\cU} \log(1/(\nu d)}$ as in Corollary~\ref{corr:starting-value-b-search}, and $M$ as in Lemma~\ref{lem:Lipschitzness-of-h}. 
\end{itemize}

We also use the following bound from \citep{karmarkar2025solvingzerosumgames}. 
\begin{lemma}[Remark 6.12 of \citep{karmarkar2025solvingzerosumgames}, restated]\label{lemma:uniform-bound-start} For any $A \in \R^{m \times n}$ satisfying the assumptions of the setup (Definition~\ref{def:matrix-games-setups}) and any $z \in \cZ_\nu$, we have that $\normInline{\ground{A}{z}}_F^2 \leq 1$. 
\end{lemma}

\begin{proof}[Proof of Theorem~\ref{thm:final-result-l1-l1-aka-zero-sum} and Theorem~\ref{thm:final-result-l2-l1-aka-SVM}] The proof follows immediately from Theorem~\ref{thm:main-general-result} and Lemma~\ref{lem:truncation-for-ell2ell1-ell1ell1}. Indeed, note that the parameters $K, \Gamma_\cS = \tilde{O}(1)$, the parameters $\pi(\iota(5)), \rho, \cT_{\max}, \zeta$ are absolute constants, and $\kappa(\epsilon
, \delta), \epsilon', \beta$ are all inverse polynomial in the problem parameters. Consequently, Theorem~\ref{thm:main-general-result} guarantees a matvec complexity of 
\begin{align*}
    \tilde{O}\paren{ \beta\epsilon^{-1} \paren{1 + \frac{\tau}{\beta}} + \frac{1 + \beta^\rho}{\tau^2}}. 
\end{align*}
Substituting $\tau = \beta = \epsilon^{1/3}$ and $\rho = 2$ now yields the final complexity of $\tilde{O}(\epsilon^{-2/3})$ matvecs to $A$, as desired. 
\end{proof}

\section*{Acknowledgements}

Ishani Karmarkar was funded in part by NSF Grant CCF-1955039, and a PayPal research award.
Liam O'Carroll was funded in part by NSF Grant CCF-1955039.
Aaron Sidford was funded in part by a Microsoft Research Faculty Fellowship, NSF Grant CCF1955039, and a PayPal research award. We thank anonymous STOC reviewers and Jelena Diakonikolas for their helpful feedback on an earlier version of this paper. 

\newpage

\bibliographystyle{plainnat}

\newpage

\addtocontents{toc}{\protect\setcounter{tocdepth}{0}} %

\appendix

\section{Appendix}
\subsection{Properties of KL divergence and Hellinger distance}

\begin{restatable}[Equation 2.27 of \citep{tsybakov2008nonparametric}, restated]{fact}{hellinger}\label{fact:hellinger-trick} Let $z_1, z_2 \in \Delta^n$. Then $\KL(z_1 || z_2) \geq H(z_1,z_2)$. 
\end{restatable}
\begin{proof} The function $g(s) = -\log(s)$ is convex, hence, for any $s > 0$, 
\begin{align*}
    0 = g(1) \geq g(s) + g'(s) (1-s) = -\log(s) + \frac{-1}{s} (1-s). 
\end{align*}
Thus, $\log(s) \geq 1 - \frac{1}{s}$ and consequently $s^2 \log(s^2) \geq 2s^2 - 2s.$ Taking $r = s^2$, we have that for any $r > 0$,
\begin{align*}
    r\log(r) \geq 2r - 2\sqrt{r}
\end{align*}
and 
\begin{align*}
    r\log(r) - (r-1) \geq r - 2\sqrt{r} + 1 = (1-\sqrt{r})^2. 
\end{align*} 
Taking $r_i = [z_1]_i/[z_2]_i$, we have 
\begin{align*}
    \breg{z_2}{z_1} &= \sum_{i} [z_1]_i \log(r_i) = \sum_i [z_2]_ir_i\log(r_i) = \sum_i [z_2]_i[r_i\log(r_i) - (r_i - 1)] \\&\geq \sum_i [z_2]_i (\sqrt{r}_i - 1)^2 
    = \sum_i (\sqrt{[z_2]_i} \sqrt{r}_i - \sqrt{[z_2]_i})^2 = H^2(z_1,z_2). 
\end{align*} 
where the third equality on the first line uses that $[z_2]_i r_i = [z_1]_i$ and hence $\sum_{i \in [d]} [z_2]_i r_i = 1$.  
\end{proof}

\subsection{Collapsing sums of divergences to the divergence from a single point}
\begin{lemma}\label{lemma:reduce-sum-to-one-point}
Let $\cV = \{v^1, \dots, v^k\} \subset \simplex^d$ be a multiset and $\geomean{\cV} \in \simplex^d$ be defined via 
\begin{align*}
    [\geomean{\cV}]_j \defeq \frac{g_j}{\sum_{j \in[d]} g_j}, \text{ where } g_j \defeq \prod_{i \in [k]} [v^i]_j^{1/k} \text{ for each } j \in [d]. 
\end{align*}
Then for all $w \in \simplex^d$, we have
\begin{align*}
    \sum_{i \in [k]} \inKL{v^i}{w} = k \cdot \inKL{\geomean{\cV}}{w} + C,
\end{align*}
where $C$ is a quantity which does not depend on $w$. Moreover, if for some $\nu > 0$, $\cV \subset \Delta_\nu^d$, then $\geomean{\cV} \in \Delta_\nu^d$.
\end{lemma}
\begin{proof}
Note that using standard properties of the negative entropy function (e.g., Lemmas 10 and 21 in \citet{carmon2025extractingdualsolutionsprimal}), we can equivalently express $\geomean{\cV} = \argmin_{q' \in \simplex^d} \frac{1}{k} \sum_{i \in [k]} \inKL{u^i}{q'}$. The remainder of the proof of the first claim uses identical reasoning to the proof of Lemma 2 in \citep{carmon2025extractingdualsolutionsprimal}. For the additional claim that $\cV \in \Delta^d_\nu$ implies $\geomean{\cV} \in \Delta^d_\nu$, note first that $\cV \in \Delta^d_\nu$ implies 
\begin{align*}
    \nu \leq g_j \leq \frac{1}{k} \sum_{i \in [k]} [v^i]_j, 
\end{align*}
where the left inequality holds because the geometric mean is larger than the minimum, and the right inequality holds because of the AM-GM inequality. Consequently, 
\begin{align*}
    \sum_{j \in [d]} g_j \leq \sum_{j \in [d]} \frac{1}{k} \sum_{i \in [k]} [v^i]_j = \frac{1}{k} \sum_{i \in [k]} \sum_{j \in [d]} [v^i]_j = 1, 
\end{align*}
where the last inequality used that each $v^i \in \Delta^d_\nu \subset \Delta^d$. Thus, $[\geomean{\cV}]_j \geq \nu$. 
\end{proof}

\begin{lemma}\label{lemma:reduce-sum-to-one-point-euclidean}
Let $\cV = \{v^1, \dots, v^k\} \subset \B^d$ be a multiset and $\mean{\cV} \defeq \frac{1}{\abs{\cV}} \sum_{v \in \cV} v$. Then for all $w \in \B^d$, we have
\begin{align*}
    \sum_{i \in [k]} \normInline{v^i - w}_2^2 = |\cV| \cdot \normInline{\mean{\cV} -w}_2^2 + C,
\end{align*}
where $C$ is a quantity which does not depend on $w$.
\end{lemma}
\begin{proof} Note that 
\begin{align*}
     \sum_{i \in [k]} \normInline{v^i - w}_2^2 &=  \sum_{i \in [k]} \normInline{v^i}_2^2 - 2\inangle*{v^i, w} + \normInline{w}_2^2 \\
     &= k \normInline{w}_2^2 - 2 \inangle*{ \sum_{i \in [k]} v^i, w} + \sum_{i \in [k]} \normInline{v^i}^2_2 \\
     &= k \normInline{w}_2^2 - 2n \inangle*{\mean{\cV}, w} + \sum_{i \in [k]} \normInline{v^i}_2^2 \\
    &= k \normInline{w}_2^2 - 2k \inangle*{\mean{\cV}, w} + k\normInline{\mean{\cV}}_2^2 + \sum_{i \in [k]} \normInline{v^i}_2^2 - k\normInline{\mean{\cV}}_2^2 \\
    &= k \normInline{w - \mean{\cV}}_2^2 + \sum_{i \in [k]} \normInline{v^i}_2^2 - n\normInline{\mean{\cV}}_2^2, 
\end{align*}
where the second two terms do not depend on $w$. 
\end{proof}

In the following corollary, we let $(\xset, \yset, \xtrunc, \ytrunc, \cZint, \rx : \xset \to \R, \ry : \yset \to \R, \Gamma_{\dgfsetup})$ be any of the setups defined Definition~\ref{def:matrix-games-setups}. 

\begin{corollary}\label{corr:reduce-sum-to-one-point} Consider a finite nonempty multiset $\cU \subset \cZ_\nu$. There exists a point $\collapsed{\cU}$ such that for all $w \in \cZ_\nu$, $\breg{\cU}{w} = |\cU| \cdot \breg{\collapsed{\cU}}{w} + C$ where $C$ is a quantity that does not depend on $w$. 
\end{corollary}
\begin{proof} This follows from Lemma~\ref{lemma:reduce-sum-to-one-point-euclidean} and Lemma~\ref{lemma:reduce-sum-to-one-point}. 
\end{proof}

\subsection{Range bound for bisection search analysis}\label{sec:range}

Throughout Appendix~\ref{sec:range}, we let $(\xset, \yset, \xtrunc, \ytrunc, \cZint, \rx : \xset \to \R, \ry : \yset \to \R, \Gamma_{\dgfsetup})$ be any of the setups defined Definition~\ref{def:matrix-games-setups}. 

\begin{lemma}\label{lemma:uncondition-simplex-bound} Let $\{u^1, ..., u^k\} \subset \Delta^d_\nu$ be a multiset. Then, 
\begin{align*}
    \min_{x \in \Delta^d_\nu} \sum_{i \in [k]} \KL(x || u^i) \leq k \cdot \log\paren{\frac{1}{\nu d}}
\end{align*}
\end{lemma}
\begin{proof} Consider $x = [\geomean{\cV}]$ where, we recall from Lemma~\ref{lemma:reduce-sum-to-one-point} that
\begin{align*}
        [\geomean{\cU}]_j \defeq \frac{g_j}{\sum_{j \in[d]} g_j}, \text{ where } g_j \defeq \prod_{i \in [k]} [u^i]_j^{1/k} \text{ for each } j \in [d]. 
\end{align*}
For notational convenience, let $G \defeq \sum_{j \in[d]} g_j$. Then, 
\begin{align*}
    \sum_{i \in [k]} \KL(x || u) &= \sum_{i \in [k]} \sum_{j \in [d]} [x]_j \log\paren{\frac{[x]_j}{[u^i]_j}} = \sum_{j \in [d]} [x]_j \paren{k \log([x]_j) - \sum_{i \in [k]} \log([u^i]_j)} \\
    &= \sum_{j \in [d]} [x]_j \paren{k \log([x]_j) - k \log([g]_j)} \\
    &= \sum_{j \in [d]} [x]_j k \paren{\log\paren{\frac{[g]_j}{G}} - \log([g]_j)} \\
    &= \sum_{j \in [d]} [x]_j k \log\paren{\frac{1}{G}} \\
    &= k \log\paren{\frac{1}{G}} 
    \leq k \log\paren{\frac{1}{\nu d}}
\end{align*} 
where the second line used the fact that 
\begin{align*}
    g_j \defeq \prod_{i \in [k]} [u^i]_j^{1/k}~~\text{ if and only if }~~\log(g_j) = \frac{1}{k} \sum_{i \in [k]} \log([u^i]_j), 
\end{align*}
the second-to-last line used that $x \in \Delta^d$, and the last line used that $g_j \geq \nu$ (since the geometric mean is lower bounded by the minimum) implies that $G \geq \nu d$. 
\end{proof}

\begin{lemma}\label{lemma:unconditional-euclidean-bound} Let $\{u^1, ..., u^k\} \subset \B^d_\nu$ be a multiset. Then, 
\begin{align*}
    \min_{x \in \B^n_{\nu}} \sum_{i \in [k]} \normInline{x - u^i}_2^2 \leq k. 
\end{align*}
\end{lemma}
\begin{proof} The minimizer is achieved by the mean, in which case the minimum value is given by 
\begin{align*}
    \sum_{i \in [k]} \normInline{u^i - \mean{\cU}}_2^2 \leq \sum_{i \in [k]} (\normInline{u^i}_2 + \normInline{\mean{\cU}}_2)^2 \leq \sum_{i \in [k]} 4 = 4k, 
\end{align*}
where the first inequality uses triangle inequality, and the second uses that $u^i, \mean{\cU} \in \B^d$. 
\end{proof} 

\end{document}